\documentclass[a4paper]{article}
\usepackage[utf8]{inputenc}
\usepackage{graphicx}
\usepackage{fullpage}
\usepackage[affil-it]{authblk}
\usepackage{longtable}
\usepackage{amsfonts,amssymb,amsmath,amscd,amsthm,amsbsy,hyperref,tikz}

\usepackage{array}
\usepackage{multirow}
\usepackage{multicol}
\usepackage{booktabs}
\usepackage{caption}
\usepackage{enumitem}
\usepackage{mathtools}

\usepackage{setspace}
\numberwithin{equation}{section}

\newcommand{\dd}{\mathrm{d}}

\newtheorem{thm}{Theorem}
\newtheorem{cor}{Corollary}

\newtheorem{prop}{Proposition}
\newtheorem{ass}{Assumption}
\newtheorem{lem}{Lemma}
\newtheorem{rmk}{Remark}

\title{Computationally-assisted proof of a novel \\ $\mathsf{O}(3)\times \mathsf{O}(10)$-invariant Einstein metric on $S^{12}$}
\author[1]{Timothy Buttsworth}
\author[2]{Liam Hodgkinson}

\affil[1]{\small School of Mathematics and Statistics\\
The University of New South Wales\\
Kensington, Sydney\\ NSW 2052, Australia\\ 
\emph{E-mail address}: {\tt t.buttsworth@unsw.edu.au}}
\vspace{0.5cm}
\affil[2]{\small School of Mathematics and Statistics \\
The University of Melbourne \\
Parkville, Melbourne\\
  VIC 3052, Australia\\ 
\emph{E-mail address}: {\tt lhodgkinson@unimelb.edu.au}}

\date{}

\begin{document}

\maketitle

\abstract{
We prove existence of a non-round Einstein metric $g$ on $S^{12}$ that is invariant under the usual cohomogeneity one action of $\mathsf{O}(3)\times\mathsf{O}(10)$ on $S^{12}\subset \mathbb{R}^{13}= \mathbb{R}^3\oplus \mathbb{R}^{10}$. The proof involves using several rigorous numerical analysis techniques to produce a Riemannian metric $\hat{g}$ which approximately satisfies the Einstein condition to known high precision, and then demonstrating that $\hat{g}$ can be perturbed into a true Einstein metric $g$.}

\section{Introduction}







Let $M$ be a smooth manifold. An \textit{Einstein metric} on $M$ is a smooth Riemannian metric $g$ for which there exists a constant $\lambda$ so that
\begin{equation}\label{GRS}
 \mathrm{Ric} (g)=\lambda g \ \text{on} \ M,
\end{equation}
where $\mathrm{Ric}(g)$ is the Ricci curvature of $g$. Finding and classifying complete Einstein metrics on a given smooth manifold $M$ (i.e., those Einstein metrics $g$ for which the Riemannian metric space $(M,g)$ is complete) is one of the most important topics in Riemannian geometry, mainly because the Einstein condition offers a coherent notion of what it means for a manifold's geometry to be `ideal'. See \cite{Besse} for a more thorough discussion of this idea, including why Einstein metrics are also `ideal' from the perspective of physics. 

If the dimension of $M$ is at most $3$, then any Einstein metric has constant sectional curvature, so the Killing-Hopf theorem implies that any complete Einstein manifold $(M,g)$ must be a quotient of the round sphere, flat Euclidean space, or hyperbolic space. Starting in dimension $4$, this classification result fails, and we are forced to use other methods in the study of the Einstein condition. 
In local coordinates, \eqref{GRS} is a system of quasi-linear, second order, non-elliptic partial differential equations, so one might hope to use PDE methods to solve it. While the local theory of solutions to this problem is now well-understood (see, for example, Chapter 5 of \cite{Besse}), the global theory is much more complicated. 


One of the most popular ways of producing complete solutions of \eqref{GRS} is by imposing \textit{symmetry}, i.e., one often assumes that $g$ is invariant under a certain Lie group action $G$ of $M$. Rather than directly reducing the number of equations being solved, the main point of imposing symmetry is that it reduces the number of independent variables in the PDE (although the symmetry is often powerful enough to reduce the number of equations as well). In case $G$ acts transitively on $M$, the manifold is diffeomorphic to the homogeneous space $G/H$ for some choice of isotropy subgroup $H$, and the Einstein equation for $G$-invariant Riemannian metrics becomes a system of algebraic equations. This algebraic problem has been studied thoroughly, and while there are many results and important articles that could be cited, we instead feel that it is only appropriate to highlight the following key observations:
\begin{itemize}
\item If $\lambda>0$, the Bonnet-Myers theorem implies that $M=G/H$ is necessarily compact. There are now many examples of compact homogeneous Einstein metrics, and the general theory has been well-developed in, for example, \cite{Bohm04,BWZ}. 
\item If $\lambda=0$, the Alekseevskii-Kimel'fel'd theorem \cite{AK} implies that any homogeneous solution of \eqref{GRS} is Riemann-flat (in particular, the universal cover is flat Euclidean space).
\item If $\lambda<0$, a theorem of Bochner (Theorem 1.84 of \cite{Besse}) implies that the homogeneous space must be non-compact.  The Alekseevskii conjecture (proven in \cite{BLA}) demonstrates that the resulting homogeneous space is diffeomorphic (though certainly not isometric) to $\mathbb{R}^n$. By \cite{Lauret}, the structure of homogeneous Einstein metrics on $\mathbb{R}^n$ is quite well-understood. 
\end{itemize}

After assuming that $G$ acts transitively, one natural ansatz asserts that $G$ acts with \textit{cohomogeneity one}, which means that the generic orbits of the action of $G$ in $M$ have dimension one less than that of the manifold, in which case, the Einstein condition becomes a system of ODEs. The first non-trivial compact Einstein metrics in this context were constructed by Page \cite{Page} and B\'erard-Bergery \cite{BB}, and the general theory for the short-time existence of solutions to these ODEs was developed in \cite{EschenburgWang}. Many complete cohomogeneity one Einstein metrics (both compact and non-compact) have been constructed, and are too numerous to discuss thoroughly. 

The focus of this paper is the problem of solving \eqref{GRS} on $S^{n}$ with $\lambda=n-1$ 
for a Riemannian metric $g$ which is invariant under the usual cohomogeneity one 
action of $\mathsf{O}(d_1+1)\times \mathsf{O}(d_2+1)$, with $d_1+d_2+1=n$. The relevant results about this problem are as follows: 
\begin{itemize}
\item The usual round Einstein metric on $S^{d_1+d_2+1}$ is invariant under this group action, so authors typically search for non-round solutions.
    \item If $d_1=1$ or $d_2=1$ (in particular, if the dimension of the sphere is $n=d_1+d_2+1=4$), then any invariant Einstein metrics are round. One way of seeing this is to reduce the equations to a system of two first order ODEs and analyse these directly (c.f. Remark 5.4 of \cite{NienhausWink}). Alternatively, one can use a maximum principle argument to show that the curvature operator of any invariant Einstein metric must be non-negative, and then invoke the B\"ohm-Wilking rounding theorem \cite{BW} to conclude that the Einstein metric must be round (cf. \cite{Donovan}). 
    \item If $n=5,6,7,8,9$, B\"ohm showed \cite{Bohm98} that there are infinitely-many non-round, pairwise non-isometric solutions.
    \item Nienhaus-Wink \cite{NienhausWink} constructed three new non-round Einstein metrics on $S^{10}$, one for each pair $(d_1,d_2)\in \{(2,7),(3,6),(4,5)\}$, and conjectured that this exhausts all Einstein metrics that are invariant under this particular group action, up to isometry. 
\end{itemize}
These existence results for $n\in [5,10]$ were obtained by studying the linearisation of the Einstein equation at a \textit{singular} Einstein metric that is also $\mathsf{O}(d_1+1)\times \mathsf{O}(d_2+1)$-invariant. In fact, this singular Einstein metric, also called the \textit{cone solution}, can be described explicitly. For $n=5,6,7,8,9$, B\" ohm's study of the linearisation of the Einstein equation reveals new Einstein metrics that are `arbitrarily close' to the cone solution, while for $n=10$, Nienhaus-Wink combine the study of the linearisation with a winding number argument to produce new Einstein metrics that are `close' but not `\textit{arbitrarily} close' to the cone solution.    Even though a non-round Einstein metric has been numerically detected on $S^{12}$ (for $d_1=2$ and $d_2=9$), the authors of \cite{NienhausWink} have explained that its existence does not follow from their techniques; it appears that this non-round Einstein metric is `too far away' from the cone solution for the winding number approach to reach it. In this paper, we use new techniques to prove existence of this Einstein metric. 
\begin{thm}\label{CT}
 There exists a non-round Einstein metric on $S^{12}$ which is $\mathsf{O}(3)\times \mathsf{O}(10)$-invariant.
\end{thm}
Our proof of this result is again perturbative in nature. However, the key new insight we offer is that by taking careful advantage of computational resources, we are free to linearise about a metric that is arbitrarily close to a perceived non-round Einstein metric. We quantify the extent to which this metric satisfies the Einstein equation with an \textit{a posteriori} error bound, and then ``fix the error" with a perturbation, guided by the linearisation of the Ricci curvature operator at this same metric. 
For our numerics, the general approach we take, which is likely adaptable to find solutions of other PDEs, is as follows:
\begin{enumerate}[label=(\Roman*)]
\item \label{step:Heur} Obtain an \emph{heuristic pointwise solution} with no guaranteed regularity, accuracy, or integrability in elementary terms. In our work, this approximate solution is obtained using an arbitrary-precision Taylor series solver.
\item \label{step:Interp} Approximate the heuristic solution by Chebyshev interpolation to obtain $\hat{g}$. Then, using interval arithmetic and symbolic integration, rigorously compute a high-precision upper bound on the \emph{a posteriori} error in $\hat{g}$.
\item \label{step:Lin} To high accuracy, solve for the linearisation around $\hat{g}$ as a system of linear differential equations.
\item \label{step:Theory} Apply fixed-point theory to conclude the existence of a nearby solution. 
\end{enumerate}
In general, we can only expect this methodology to be successful if the \emph{heuristic pointwise solution} of Step \ref{step:Heur} is highly-accurate, as verified in Step \ref{step:Interp}. As a result, \textit{any} method for constructing an \emph{heuristic pointwise solution} will be acceptable, provided a high-level of accuracy is expected. Steps \ref{step:Interp} and \ref{step:Lin} involve only linear problems, which can generally be reliably solved to very high precision. However, in our approach, we will find that our solver in Step \ref{step:Heur} is sufficiently effective that we may obtain a numerical solution using only a finite difference scheme. We note that for functions of more than one real-variable, Chebyshev interpolation in Step \ref{step:Interp} can also be replaced by another general regression procedure (e.g. kernel regression). 
The main purpose of this paper is to illustrate that this general method can be applied successfully to prove Theorem \ref{CT}. The paper is organised according to these steps, and the specific sections are broken up as follows:
\begin{itemize}
\item In Section \ref{geomprelim}, we discuss the basics of the Einstein equation for $\mathsf{O}(d_1+1)\times \mathsf{O}(d_2+1)$-invariant Riemannian metrics on $S^{d_1+d_2+1}$, and formulate the existence question as a shooting problem;
\item In Section \ref{prelimapprox}, we conduct some low-precision numerical work which indicates the existence of a new Einstein metric on $S^{12}$, and make some observations about what this hypothetical (at this stage) Einstein metric should look like;
    \item In Section \ref{Guarantees}, the `theoretical' and estimate-heavy component of the paper, we prove existence of a novel $\mathsf{O}(3)\times \mathsf{O}(10)$ Einstein metric $g$, under the \textit{assumption} that it is possible to construct an approximate $\mathsf{O}(3)\times \mathsf{O}(10)$ Einstein metric $\hat{g}$ whose Ricci tensor is \textit{very} close to being a constant multiple of the identity, and whose Ricci tensor linearisation is sufficiently non-degenerate (this achieves Step \ref{step:Theory});
   \item In Section \ref{numerics}, we construct the required approximate solution $\hat{g}$ according to Steps \ref{step:Heur}--\ref{step:Lin}. This completes the proof.   
\end{itemize}
We are hopeful that our methods will go on to be used to find solutions of other important geometric equations, including in situations with less symmetry. 


\section*{Acknowledgments} The first author is grateful to the \textit{Mathematisches Forschungsinstitut Oberwolfach} for hosting the event \textit{Geometrie} in June 2024, where the idea for this project began. The first author is also grateful to Matthias Wink for insightful discussions on this project. Both authors were supported financially by the Australian Government through the Australian Research Council grants DE220100919 and DE240100144, respectively. 

\section{Geometric Preliminaries}\label{geomprelim}

In general, the imposition of cohomogeneity one symmetries on solutions of \eqref{GRS} reduces the problem to a system of ordinary differential equations. The specific cohomogeneity one symmetry of interest in this paper is the standard product action of $\mathsf{O}(d_1+1)\times \mathsf{O}(d_2+1)$ on  $S^{d_1+d_2+1}\subset \mathbb{R}^{d_1+1}\times \mathbb{R}^{d_2+1}$. In this case, the problem becomes a system of three second-order ODEs. Since our group action has two singular orbits, this system of ODEs comes equipped with singular boundary conditions. In this section, we explicitly write these ODEs, the boundary conditions, perform some changes of variables, and describe the associated `shooting problem' that must be solved in order to produce a new Einstein metric. 
\subsection{The boundary value problem}
Choose integers $d_1,d_2\ge 2$. Under the $\mathsf{O}(d_1+1)\times \mathsf{O}(d_2+1)$ action, the principal orbits of $S^{d_1+d_2+1}$ are product spheres $S^{d_1}\times S^{d_2}$, and there are two singular orbits: $S^{d_1}$ and $S^{d_2}$. For any $T>0$, there is an equivariant diffeomorphism $\Phi$ between the principal part of $S^{d_1+d_2+1}$ and $(0,T)\times S^{d_1}\times S^{d_2}$.  Using this identification, any $\mathsf{O}(d_1+1)\times \mathsf{O}(d_2+1)$-invariant Riemannian metric on $S^{d_1+d_2+1}$ which has a distance of $T$ between the two singular orbits appears as
\begin{equation}\label{metricform}
g= dt^2+f_1^2Q_{d_1}+f_2^2Q_{d_2},
\end{equation} on the principal part of $S^{d_1+d_2+1}$, 
up to diffeomorphism, where $f_i:(0,T)\to \mathbb{R}^+$ are smooth functions and
 $Q_{d_i}$ is the round metric of unit radius on $S^{d_i}$. In order for the Riemannian metric $g$ to close up smoothly at 
 the singular orbits, the functions $f_1$ and $f_2$ must be smoothly extendable to functions on $(-T,2T)$ so that
\begin{align}\label{metricsmoothness}
\begin{split}
 f_1(0)=\frac{\sqrt{d_1-1}}{\alpha}\ \text{and $f_1(t)$ is even about $t=0$},\qquad
 f_2'(0)=1\ \text{and $f_2(t)$ is odd about $t=0$},\\
f_1'(T)=-1\ \text{and $f_1(t)$ is odd about $t=T$},\qquad
f_2(T)=\frac{\sqrt{d_2-1}}{\omega}\ \text{and $f_2(t)$ is even about $t=T$},
 \end{split}
\end{align}
for some $\alpha,\omega\in (0,\infty)$. A standard computation (cf. Proposition 2.1 in \cite{EschenburgWang}) shows that the Ricci curvature of \eqref{metricform} on the principal part of the manifold is given by 
\begin{align*}
    \mathrm{Ric}(g)&=-\left(d_1\frac{f_1''}{f_1}+d_2\frac{f_2''}{f_2}\right)dt^2\\
    &-\left(f_1f_1''+(d_1-1)(f_1')^2+d_2\frac{f_1f_1' f_2'}{f_2}-(d_1-1)\right)Q_{d_1}\\
    &-\left(f_2f_2''+(d_2-1)(f_2')^2+d_1\frac{f_2f_2' f_1'}{f_1}-(d_2-1)\right)Q_{d_2}.
\end{align*} Thus, provided the smoothness conditions \eqref{metricsmoothness} are satisfied, a necessary and sufficient condition for the Riemannian metric $g$ in \eqref{metricform} to be Einstein (with Einstein constant $\lambda$) is that the following ODEs are satisfied on $(0,T)$:
\begin{align}\label{solitonequations}
\begin{split}
 d_1\frac{f_1''}{f_1}+d_2\frac{f_2''}{f_2}&=-\lambda,\\
 \frac{f_1''}{f_1}+(d_1-1)\frac{(f_1')^2}{f_1^2}+d_2\frac{f_1' f_2'}{f_1f_2}-\frac{d_1-1}{f_1^2}&=-\lambda,\\
 \frac{f_2''}{f_2}+(d_2-1)\frac{(f_2')^2}{f_2^2}+d_1\frac{f_1' f_2'}{f_1f_2}-\frac{d_2-1}{f_2^2}&=-\lambda.
 \end{split}
\end{align}
In this paper, we will always choose $\lambda=d_1+d_2$, so that $f_1=\cos(t)$ and $f_2=\sin(t)$ with $T=\frac{\pi}{2}$ gives the round Einstein metric.
\subsection{A conserved quantity}
Since \eqref{solitonequations} consists of three scalar second-order ODEs for two scalar functions, it would appear that our equations are over-determined, but this is not the case. Indeed, a linear combination of these three equations gives 
\begin{align}\label{firstintegral}
    (d_1+d_2-1)\lambda=\frac{d_1(d_1-1)}{f_1^2}+\frac{d_2(d_2-1)}{f_2^2}-d_1(d_1-1)\frac{(f_1')^2}{f_1^2}-d_2(d_2-1)\frac{(f_2')^2}{f_2^2}-2d_1d_2\frac{f_1'f_2'}{f_1f_2},
\end{align}
which turns out to be a conserved quantity of \eqref{solitonequations}. We make this precise with the following lemma (cf. Lemma 2.4 of \cite{EschenburgWang}). 
\begin{lem}\label{bianchi}
    If $f_1,f_2$ are smooth and positive functions on a non-empty open interval $I$ satisfying at least two of the three equations of \eqref{solitonequations} for some constant $\lambda$, then the final equation also holds, provided \eqref{firstintegral} holds at some $t_0\in I$. 
\end{lem}
\begin{proof}
   Define $F:(\mathbb{R}^+)^2\times \mathbb{R}^2\to \mathbb{R}$ with 
   \begin{align*}
       (d_1+d_2-1) F(u,v,x,y)=\frac{d_1(d_1-1)}{u^2}+\frac{d_2(d_2-1)}{v^2}-d_1(d_1-1)\frac{x^2}{u^2}-d_2(d_2-1)\frac{y^2}{v^2}-2d_1d_2\frac{xy}{uv}.
   \end{align*}
Define $\tilde{f}_1,\tilde{f}_2$ to be the unique functions in a small neighbourhood of $t_0$ so that 
$\tilde{f}_i(t_0)=f_i(t_0)$ and $\tilde{f}_i'(t_0)=f_i'(t_0)$ for $i=1,2$, and so that they solve the same two equations as $f_1,f_2$, except with $\lambda$ replaced by $F(\tilde{f}_1,\tilde{f}_2,\tilde{f}_1',\tilde{f}_2')$. For example, if $f_1,f_2$ solve the first two equations, then $\tilde{f}_1,\tilde{f}_2$ are chosen to satisfy the well-posed IVP
\begin{align*}
 d_1\frac{\tilde{f}_1''}{\tilde{f}_1}+d_2\frac{\tilde{f}_2''}{\tilde{f}_2}&=-F(\tilde{f}_1,\tilde{f}_2,\tilde{f}_1',\tilde{f}_2'),\\
 \frac{\tilde{f}_1''}{\tilde{f}_1}+(d_1-1)\frac{(\tilde{f}_1')^2}{\tilde{f}_1^2}+d_2\frac{\tilde{f}_1' \tilde{f}_2'}{\tilde{f}_1\tilde{f}_2}-\frac{d_1-1}{\tilde{f}_1^2}&=-F(\tilde{f}_1,\tilde{f}_2,\tilde{f}_1',\tilde{f}_2'),\\
 \tilde{f}_i(t_0)=f_i(t_0), \qquad \tilde{f}'_i(t_0)=f_i'(t_0) \ \text{for} \ i=1,2. 
\end{align*}

By combining these two equations with the definition of $F$, we find that $\tilde{f}_1,\tilde{f}_2$ solve 
\begin{align*}
 d_1\frac{\tilde{f}_1''}{\tilde{f}_1}+d_2\frac{\tilde{f}_2''}{\tilde{f}_2}&=-\Lambda,\\
 \frac{\tilde{f}_1''}{\tilde{f}_1}+(d_1-1)\frac{(\tilde{f}_1')^2}{\tilde{f}_1^2}+d_2\frac{\tilde{f}_1' \tilde{f}_2'}{\tilde{f}_1\tilde{f}_2}-\frac{d_1-1}{\tilde{f}_1^2}&=-\Lambda,\\
 \frac{\tilde{f}_2''}{\tilde{f}_2}+(d_2-1)\frac{(\tilde{f}_2')^2}{\tilde{f}_2^2}+d_1\frac{\tilde{f}_1' \tilde{f}_2'}{\tilde{f}_1\tilde{f}_2}-\frac{d_2-1}{\tilde{f}_2^2}&=-\Lambda,
\end{align*}
where $\Lambda(t)=F(\tilde{f}_1(t),\tilde{f}_2(t),\tilde{f}_1'(t),\tilde{f}_2'(t))$. As a result, the Riemannian metric $\tilde{g}=dt^2+\tilde{f}_1^2 Q_{d_1}+\tilde{f}_2^2 Q_{d_2}$ satisfies $\mathrm{Ric}(\tilde{g})=\Lambda \tilde{g}$ in a tubular neighbourhood of the principal orbit corresponding to $t_0$. Since $d_1+d_2+1>2$, the second contracted Bianchi identity $d(\text{tr}_{\tilde{g}}\mathrm{Ric}(\tilde{g}))=2\text{div}(\mathrm{Ric}(\tilde{g}))$  then implies that $d\Lambda=0$, so $\Lambda(t)$ must be constant. Since \eqref{firstintegral} holds at $t_0$, $\Lambda(t)=\lambda$ everywhere. Uniqueness of solutions to ODEs then implies that $\tilde{f}_i=f_i$, as required. 
\end{proof}

\begin{rmk}
It is possible to obtain this proof directly, i.e., verifying with direct computation that \eqref{firstintegral} is a conserved quantity of any two of the three ODEs, but this will provide no new information beyond what the second contracted Bianchi identity already tells us. 
\end{rmk}

\subsection{Shooting problem}\label{shootingproblem}
A convenient way to break up the singular boundary value problem \eqref{metricsmoothness}-\eqref{solitonequations} is to consider the two singular value problems obtained by solving \eqref{solitonequations} from either the $t=0$ or $t=T$ end, and ask that the two solutions meet smoothly somewhere on the interior. The following existence result for this singular IVP is standard (cf. \cite{Bohm98}), and relies crucially on the conserved quantity \eqref{firstintegral}. 
\begin{prop}\label{IVPf}
    Fix $\lambda=d_1+d_2$. For each $\alpha>0$ and $\lambda>0$, there is a unique $t_{\alpha}>0$ and a unique pair of positive scalar functions $(f_1,f_2)$ defined in $[0,t_{\alpha}]$ solving  \eqref{solitonequations}, as well as the $t=0$ initial conditions of \eqref{metricsmoothness}, as well as the stopping condition $\left(\frac{d_1 f_1'}{f_1}+\frac{d_2 f_2'}{f_2}\right)\vert_{t_{\alpha}}=0$. Similarly, for each $T>0$ and $\omega>0$, there is a unique $t_{\omega}>0$ and a unique pair of positive scalar functions on $[T-t_{\omega},T)$ satisfying \eqref{solitonequations}, the $t=T$ initial conditions of \eqref{metricsmoothness}, and the stopping condition $\left(\frac{d_1 f_1'}{f_1}+\frac{d_2 f_2'}{f_2}\right)\vert_{T-t_{\beta}}=0$. The solution in both cases depends smoothly on $\alpha$, or $\omega$, respectively. 
\end{prop}
We do not include a proof of this result since it is standard, but we include some remarks on the proof at the end of Section \ref{reformulate}. 
 We can use Proposition \ref{IVPf} to define two continuous functions $A:\mathbb{R}^+\to \mathbb{R}^2$ and $\Omega:\mathbb{R}^+\to \mathbb{R}^2$ as
\begin{align}\label{defAO}
    A(\alpha)=\left(\frac{\sqrt{d_2-1}}{f_2(t_{\alpha})},-\frac{d_2 f_2'(t_{\alpha})}{f_2(t_{\alpha})}\right), \qquad \Omega(\omega)=\left(\frac{\sqrt{d_2-1}}{f_2(T-t_{\omega})},-\frac{d_2 f_2'(T-t_{\omega})}{f_2(T-t_{\omega})}\right),
\end{align}
where for $A$, $f_2$ is the solution from Proposition \ref{IVPf} starting at $t=0$, while the $\Omega$ definition uses the $f_2$ solution evolving backwards from $T$. Note that the definition of $\Omega$ does not depend on $T$ because the ODEs we are solving do not depend explicitly on $t$, so solutions can be time-shifted to new solutions. 
The $A$ and $\Omega$ functions can be used to find Einstein metrics as described in the following:  
\begin{prop}\label{shooting}
    If the pair $(\alpha,\omega)\in (0,\infty)\times (0,\infty)$ satisfies $A(\alpha)=\Omega(\omega)$, then we can choose $T=t_{\alpha}+t_{\beta}$, so that the two pairs of functions $(f_1,f_2)$ can be glued together on $[0,T]$ to give a smooth Einstein metric. 
\end{prop}
\begin{proof}
    It is clear that the glued functions $f_1,f_2$ satisfy \eqref{metricsmoothness}, and are smooth and satisfy \eqref{solitonequations} on $(0,t_{\alpha})\cup (t_{\alpha},T)$. The condition $A(\alpha)=\Omega(\omega)$ implies that $f_2$ is continuous and has a continuous first derivative over $t_{\alpha}$. Furthermore, the conserved quantity \eqref{firstintegral} and the fact that $\frac{d_1 f_1'}{f_1}+\frac{d_2 f_2'}{f_2}=0$ at $t_{\alpha}$ (in the limit sense from both sides) implies that $f_1$ is continuous, and also has a continuous first derivative over $t_{\alpha}$. We then conclude that the functions are in fact smooth over $t_{\alpha}$ because of existence, uniqueness and smoothness of solutions to the last two equations of \eqref{solitonequations} starting at $t_{\alpha}$, with prescription of the values of the functions, as well as their first derivatives. 
\end{proof}

\subsection{Reformulating the singular initial value problem}\label{reformulate}
To fully understand the functions $A$, $\Omega:\mathbb{R}^+\to \mathbb{R}^2$ (whose intersections we want to find) we must understand the solutions of the ODEs \eqref{solitonequations}, equipped with the singular IVPs from \eqref{metricsmoothness}, from either $t=0$ or $t=T$. In fact, it is possible to treat both of these problems at the same time because the transformation $t\mapsto T-t$, $f_1\mapsto f_2$, $f_2\mapsto f_1$, $\alpha\mapsto \omega$, $\omega\mapsto \alpha$, $d_1\mapsto d_2$, $d_2\mapsto d_1$ leaves \eqref{solitonequations} and \eqref{metricsmoothness} unchanged. As a result, it suffices to study the singular IVP from $t=0$, with the understanding that for a given sphere, we will need to study the same IVP twice for the different combinations of $d_1,d_2$.

Let $W=\frac{\sqrt{d_2-1}}{f_2}$, $X=\frac{\sqrt{d_1-1}}{f_1}$, $Y=\frac{d_1f_1'}{f_1}$ and $Z=d_1 \frac{f_1'}{f_1}+d_2\frac{f_2'}{f_2}$, so that $d_2\frac{f_2'}{f_2}=Z-Y$. It follows from definition that 
\[
\begin{aligned}
    W'&=-\frac{f_2'\sqrt{d_2-1}}{f_2^2}\\
    &=\frac{W(Y-Z)}{d_2},
\end{aligned}
\qquad\qquad\qquad\qquad
\begin{aligned}
    X'&=-\frac{f_1'\sqrt{d_1-1}}{f_1^2}\\
    &=-\frac{YX}{d_1}. 
\end{aligned}
\]
The first and second equations of \eqref{solitonequations} become 
\begin{equation}\label{XYZEq}
\begin{split}
    Y'&=d_1\left(\frac{f_1''}{f_1}-\left(\frac{f_1'}{f_1}\right)^2\right)\\
    &=d_1\left(-d_1 \frac{(f_1')^2}{f_1^2}-d_2\frac{f_1'f_2'}{f_1f_2}+\frac{d_1-1}{f_1^2}-\lambda\right)\\
    &=-Z Y+d_1X^2-d_1\lambda
    \end{split}
    \qquad\qquad
    \begin{split}
    Z'&=d_1\left(\frac{f_1''}{f_1}-\left(\frac{f_1'}{f_1}\right)^2\right)+d_2\left(\frac{f_2''}{f_2}-\left(\frac{f_2'}{f_2}\right)^2\right)\\
    &=-d_1\left(\frac{f_1'}{f_1}\right)^2-d_2\left(\frac{f_2'}{f_2}\right)^2-\lambda\\
    &=-\frac{Y^2}{d_1}-\frac{(Z-Y)^2}{d_2}-\lambda\\
    &=-\frac{Z^2}{d_2}-Y^2\left(\frac{1}{d_1}+\frac{1}{d_2}\right)+\frac{2YZ}{d_2}-\lambda.
    \end{split}
\end{equation}
We thus obtain a system of four ODEs for four functions $W,X,Y,Z$ that are equivalent to the equations of \eqref{solitonequations} (recalling the quantity \eqref{firstintegral} which is conserved for these equations by Lemma \ref{bianchi}). The main reason for this change of variables is for reduced complexity of computational implementation. Indeed, the ODEs for $X,Y,Z$ are independent of $W$, so we can solve for $X,Y,Z$, and then recover $W$ from its own differential equation, or alternatively we can recover $W$ from the conserved quantity \eqref{firstintegral}, which now reads
\begin{align}\label{IntegralWXYZ}
    (d_1+d_2-1)\lambda=d_1X^2+d_2W^2-\frac{(d_1-1)Y^2}{d_1}-\frac{(d_2-1)(Y-Z)^2}{d_2}-2Y(Z-Y).
\end{align}
In the language of the $W,X,Y,Z$ functions, the smoothness conditions of \eqref{metricsmoothness} around $t=0$ become the following:
\begin{align}\label{WXYZsmoothness}
\begin{split}
    &W(t)-\frac{\sqrt{d_2-1}}{t} \ \text{is a smooth and odd function;}\\
    &X(t)-\alpha \ \text{is a smooth, even function which vanishes at $t=0$;}\\
    &Y(t) \ \text{is a smooth and odd function;}\\
    &Z(t)-\frac{d_2}{t} \ \text{is a smooth and odd function.}
\end{split}
\end{align}

Since the $X,Y,Z$ equations are independent of $W$, we write 
$X(t)=\alpha+\eta_1(t)$, $Y(t)=\eta_2(t)$, $Z(t)=\frac{d_2}{t}+\eta_3(t)$, where $\eta_i$ are smooth functions with $\eta_i(0)=0$. Putting these expressions into \eqref{XYZEq} gives 
\begin{align}\label{etaeq123}
\begin{split}
    \eta_1'&=-\frac{\eta_2(\alpha+\eta_1)}{d_1},\\
    \eta_2'&=-\frac{d_2\eta_2}{t}-\eta_2\eta_3+d_1(\alpha+\eta_1)^2-d_1\lambda\\
    \eta_3'&=\frac{2\eta_2}{t}-\frac{2\eta_3}{t}-\frac{\eta_3^2}{d_2}-\eta_2^2\left(\frac{1}{d_1}+\frac{1}{d_2}\right)+\frac{2\eta_2 \eta_3}{d_2}-\lambda.
    \end{split}
\end{align}
To initiate the systematic study of \eqref{etaeq123} subject to the conditions \eqref{WXYZsmoothness}, we think of $\alpha,\lambda$ as functions of time (with vanishing first derivatives), and consider the evolution of the vector-valued function $\eta(t)=(\eta_1(t),\eta_2(t),\eta_3(t),\alpha,\sqrt{\lambda})$. This change of variables, combined with the first integral of Lemma \ref{bianchi} gives the following. 
\begin{prop}\label{ftoeta}
A smooth vector-valued function $\eta:[0,\epsilon]\to \mathbb{R}^5$ will generate (via the change of variables discussed earlier in this subsection) a smooth solution $(f_1,f_2)$ of \eqref{solitonequations} on $[0,\epsilon]$ subject to the $t=0$ initial conditions of \eqref{metricsmoothness}, if and only if $\eta_1$ is even, $\eta_2$ is odd, $\eta_3$ is odd, and $\eta(t)$ satisfies the initial value problem 
\begin{align}\label{etaequationO}
    \eta'(t)=\frac{1}{t}L_{d_1d_2} \eta(t)+B_{d_1d_2}(\eta(t),\eta(t)), \qquad \eta(0)=(0,0,0,\alpha,\sqrt{\lambda}),
\end{align}
where 
\begin{align*}
    L_{d_1d_2}=\begin{pmatrix}
        0&0&0&0&0\\
        0&-d_2&0&0&0\\
        0&2&-2&0&0\\
        0&0&0&0&0\\
        0&0&0&0&0
    \end{pmatrix},
\end{align*}
and $B_{d_1d_2}:\mathbb{R}^5\to \mathbb{R}^5$ is the symmetric bi-linear form 
\begin{align*}
    B_{d_1d_2}(x,y)=
    \begin{pmatrix}
        -\frac{1}{2d_1}(x_2y_4+x_4y_2+x_1y_2+x_2y_1)\\
        -\frac{1}{2}(x_2y_3+x_3y_2)+d_1\left(x_1y_1+x_4y_4+x_1y_4+x_4y_1\right)-d_1x_5y_5\\
        -\frac{1}{d_2}x_3y_3+\frac{1}{d_2}\left(x_2y_3+x_3y_2\right)-(\frac{1}{d_1}+\frac{1}{d_2})x_2y_2-x_5y_5\\
        0\\
        0
    \end{pmatrix}.
\end{align*} 
\end{prop}
Proposition \ref{ftoeta} demonstrates that the solutions described in Proposition \ref{IVPf} can be found by solving the IVP \eqref{etaequationO}, with the stated parity constraints on $(\eta_1,\eta_2,\eta_3)$. Even though \eqref{etaequationO} is a singular initial value problem, it is well-posed for short time because the eigenvalues of $L_{d_1d_2}$ are all non-positive integers, and $(0,0,0,\alpha,\sqrt{\lambda})$ is in the kernel of $L_{d_1d_2}$. Furthermore, the corresponding solution $(f_1,f_2)$ of \eqref{solitonequations} can be extended up until the unique point where the stopping condition $\left(\frac{d_1 f_1'}{f_1}+\frac{d_2 f_2'}{f_2}\right)=0$ is satisfied. Indeed, $\lim_{t\to 0}\left(\frac{d_1 f_1'}{f_1}+\frac{d_2 f_2'}{f_2}\right)=+\infty$, 
 and the first integral \eqref{firstintegral} can be expressed as 
\begin{align*}
    (d_1+d_2-1)\lambda +\left(\frac{d_1f_1'}{f_1}+\frac{d_2f_2'}{f_2}\right)^2&=\frac{d_1(d_1-1)}{f_1^2}+\frac{d_2(d_2-1)}{f_2^2}+d_1 \left(\frac{f_1'}{f_1}\right)^2+d_2\left(\frac{f_2'}{f_2}\right)^2,
\end{align*} 
so as long as $\frac{d_1f_1'}{f_1}+\frac{d_2f_2'}{f_2}$ is positive and decreasing (follows from the first equation of \eqref{solitonequations}), the quantities $\frac{1}{f_1},\frac{f_1'}{f_1},\frac{1}{f_2},\frac{f_2'}{f_2}$ are uniformly bounded. Proposition \ref{IVPf} follows. 

\subsection{Shooting problem revisited}
As we saw in  \ref{shootingproblem}, a convenient way to find an Einstein metric is to take the two solutions of the initial value problem from Proposition \ref{IVPf}, and vary $\alpha$ and $\omega$ until we find a combination satisfying $A(\alpha)=\Omega(\omega)$. We now describe these functions in terms of the $\eta(t)$ solution of \eqref{etaequationO}. First, for a given pair of positive integers $(d_1,d_2)$ (in this paper, we take $(2,9)$), we use $\eta(\alpha,t)$ to refer to the solution of \eqref{etaequationO}, and use $t_{\alpha}$ to refer to the unique time satisfying the stopping condition of Proposition \ref{IVPf}. In the language of $\eta$ functions, $t_{\alpha}>0$ is the unique time satisfying $\frac{d_2}{t_{\alpha}}+\eta_3(t_{\alpha})=0$. On the other hand, for the same $(d_1,d_2)$ pair, we use $\zeta(\omega,t)$ to denote the solution of \eqref{etaequationO} with $d_1,d_2$ reversed, and $\alpha$ replaced with $\omega$, i.e., 
\begin{align}
\label{zetaequation0}
    \zeta'(t)=\frac{1}{t}L_{d_2d_1}\zeta(t)+B_{d_2d_1}(\zeta(t),\zeta(t)), \qquad \zeta(0)=(0,0,0,\omega,\sqrt{\lambda}). 
\end{align}
Similarly, we define the stopping time $t_{\omega}>0$ to be the unique solution of $\frac{d_1}{t_{\omega}}+\zeta_3(t_{\omega})=0$. 

We now write the functions $A(\alpha)$ and $\Omega(\omega)$ in terms of this notation. First, note that when evolving from the $\alpha$ end, we have  $\frac{\sqrt{d_2-1}}{f_2}=W$, and $-\frac{d_2 f_2'}{f_2}=Y-Z$, which is the same as $Y$ at time $t_{\alpha}$.   
The integral \eqref{IntegralWXYZ} implies that 
$W$ at the stopping time can be found with
 \begin{align*}
    W=\sqrt{\left(\frac{
    (d_1+d_2-1)\lambda-d_1X^2-\left(\frac{1}{d_1}+\frac{1}{d_2}\right)Y^2}{d_2}\right)}.
  \end{align*}
  Consequently, we have 
\begin{align}\label{Aalphadef}
A(\alpha)&=\left(\sqrt{\left(\frac{
    (d_1+d_2-1)\lambda-d_1(\alpha+\eta_1(t_{\alpha}))^2-\left(\frac{1}{d_1}+\frac{1}{d_2}\right)\eta_2(t_{\alpha})^2}{d_2}\right)},\eta_2(t_{\alpha})\right).
     \end{align}
     Also, when evolving from the $\omega$ end, we have $\frac{\sqrt{d_2-1}}{f_2}=X$, and $-\frac{d_2 f_2'}{f_2}=Y$ at the stopping time, so
\begin{align}\label{Omegaomegadef}
\Omega(\omega)&=\left(\zeta_1(t_{\omega})+\omega,\zeta_2(t_{\omega})\right).
\end{align}

\section{Preliminary Numerical Estimates}\label{prelimapprox}
In this section, we give a `plausibility' argument for the existence of a non-round Einstein metric on $S^{12}$ using low-precision numerics, and make some rough observations about what such an Einstein metric would look like. We also use this section to formulate some \textit{assumptions} about our ability to produce highly-accurate approximate solutions that are close to the perceived Einstein metric. 

Henceforth, we set $d_1=2$, and $d_2=9$. Recall that by Proposition \ref{shooting}, it is sufficient to find a pair $(\alpha,\omega)\in (\mathbb{R}^+)^2$ such that $A(\alpha)=\Omega(\omega)$, as the corresponding two solutions $\eta,\zeta$ of the two singular IVPs \eqref{etaequationO} ($d_1,d_2$ are reversed for $\zeta$) can be combined to create a single pair of functions $(f_1,f_2)$ which solves \eqref{solitonequations}, meets up smoothly in the middle, and satisfies the smoothness conditions \eqref{metricsmoothness}. To find a non-round Einstein metric, it suffices to find such a pair $(\alpha,\omega)\neq (\sqrt{d_1-1},\sqrt{d_2-1})$.

The natural first step in this task is to solve approximately for $(\alpha,\omega)$ via the differential equations at low precision. Einstein metrics coincide precisely with the intersections of the two continuous curves $A:\mathbb{R}^+\to \mathbb{R}^2$ and $\Omega:\mathbb{R}^+\to \mathbb{R}^2$. Using a standard adaptive fourth-order Dormand--Prince solver, the curves for $d_1=2,d_2=9$ are illustrated in Figure \ref{fig:curves_29}. 
\begin{figure}[t]
\centering
\includegraphics[height=5cm]{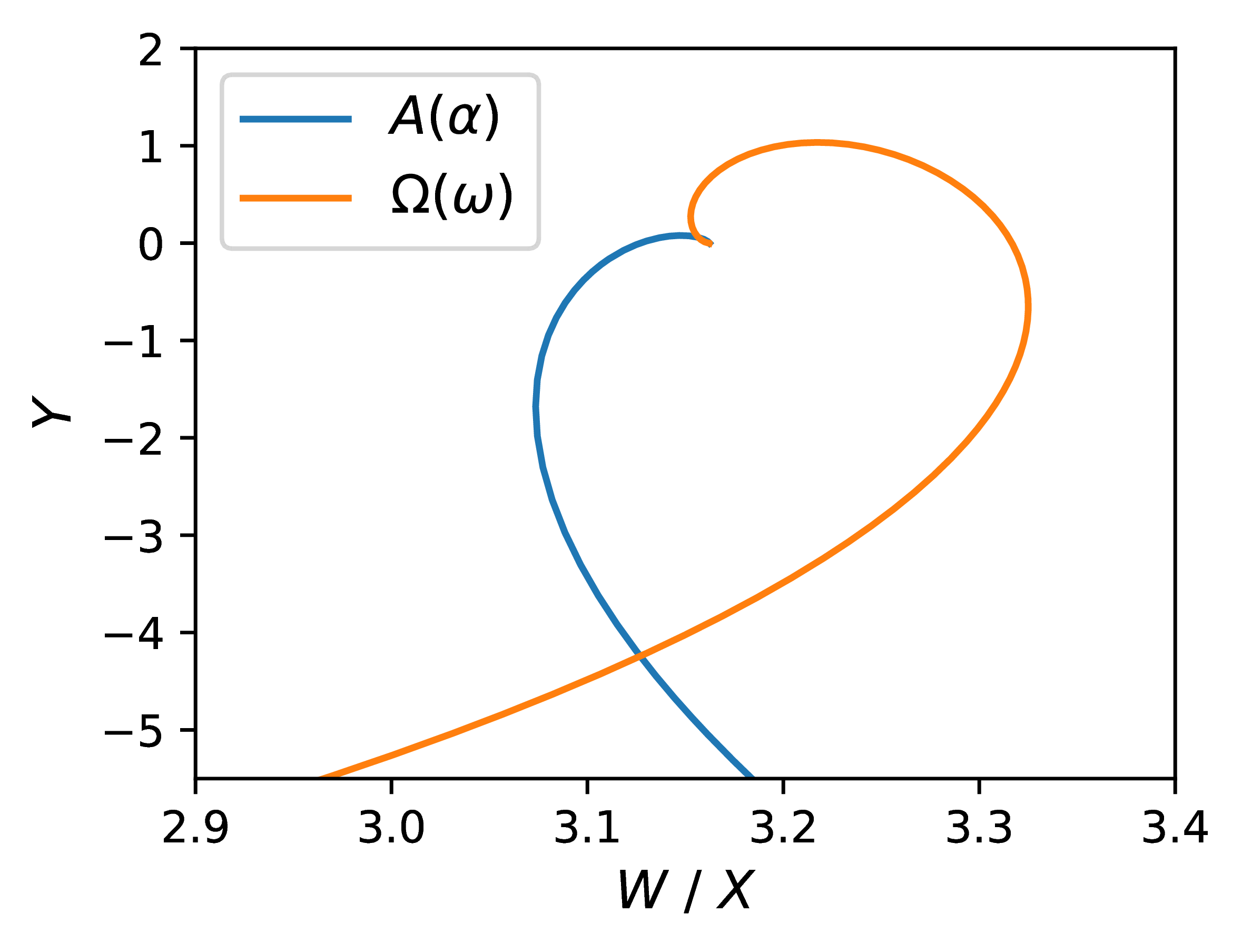}
\includegraphics[height=5cm]{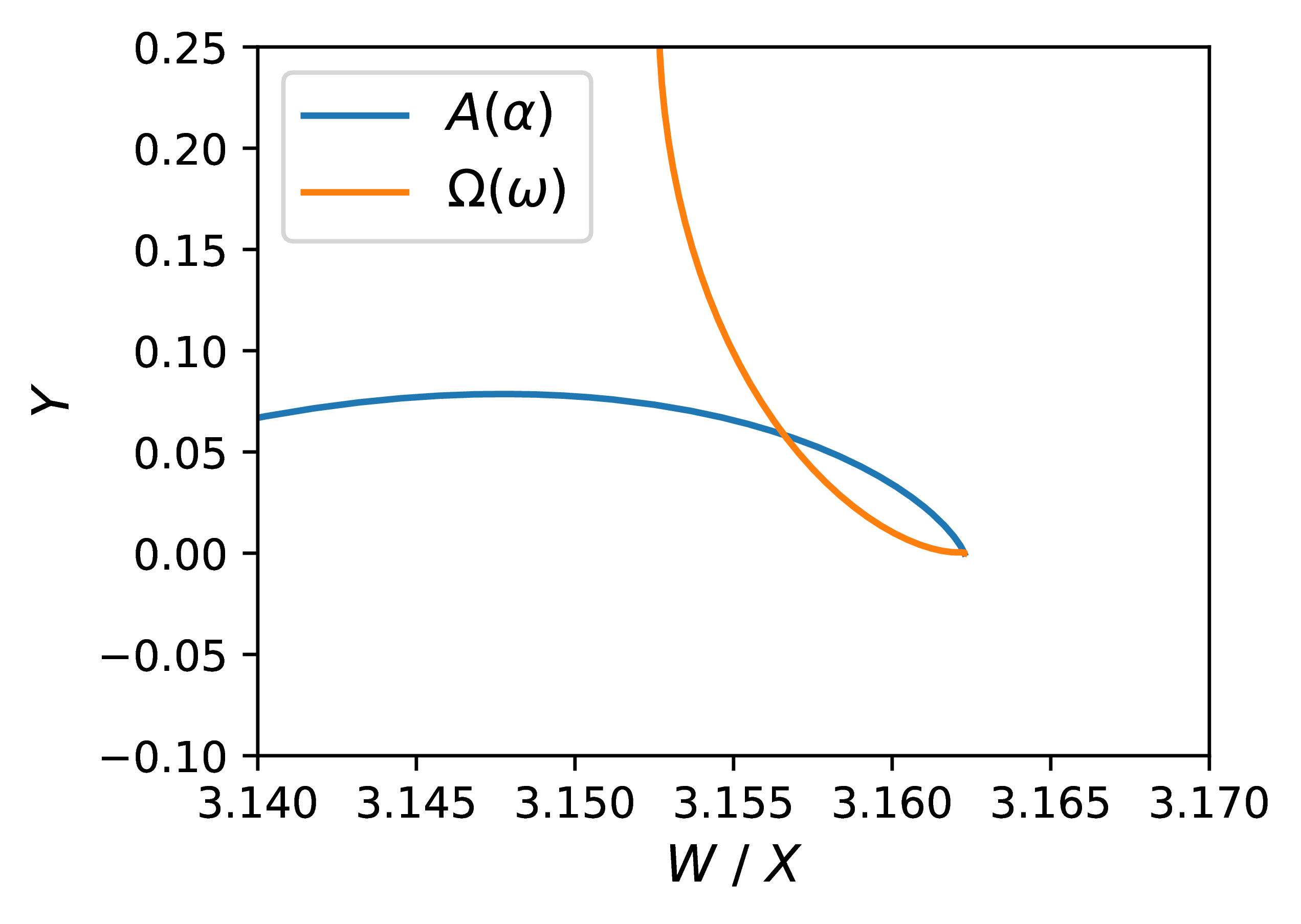}
\caption{\label{fig:curves_29}The $A(\alpha)$ and $\Omega(\omega)$ parametric curves in $\mathbb{R}^2$ for $\alpha,\omega > 0$ and $d_1 = 2$, $d_2 = 9$. The first intersection seen at left is the round Einstein metric at $\alpha = \sqrt{d_1 -1}$ and $\omega = \sqrt{d_2 - 1}$. A second intersection is best seen at right: the new non-round metric constructed in this paper.}
\end{figure}
The non-round Einstein metric arises at 
\begin{equation}
\label{eq:ValEsts}
\alpha \approx 6.0838655, \quad \omega \approx 6.1859148,\quad t_{\alpha} \approx 1.4923108,\quad t_{\omega} \approx 1.1621186,
\end{equation}
and it appears that the solutions satisfy the properties listed in Tables \ref{nonlinearstop} and \ref{starnorms} below. 
\begin{table}[h]
\centering
\begin{tabular}{crrrrrr}
\toprule
\multirow{2}{*}{Index} & \multicolumn{4}{c}{Quantities at stopping time}\\
\cmidrule(lr){2-5} \cmidrule(lr){6-7}
& $\eta(t_{\alpha})$ & $\eta'(t_{\alpha})$ & $\zeta(t_{\omega})$ & $\zeta'(t_{\omega})$ \\
\midrule
1 & $-2.896297$ & $-0.093145$ & $-3.029322$ & $-0.020498$\\
2 & $0.058442$ & $-1.678808$ & $0.058442$ & $-9.323279$  \\ 
3 & $-6.030915$ & $-6.960761$ & $-1.720995$ & $-9.521176$ \\
\bottomrule
\end{tabular}
\caption{Apparent values of $\eta_i$, $\zeta_i$ and their derivatives at $t_{\alpha}$, $t_{\omega}$, respectively.}\label{nonlinearstop}
\end{table}
\begin{table}[h]
\centering
\begin{tabular}{crrrrrr}
\toprule
\multirow{2}{*}{Norms} & \multicolumn{6}{c}{Functions} \\
\cmidrule{2-7}
& ${\eta}_1$ & ${\eta}_2$ & ${\eta}_3$ & ${\zeta}_1$ & ${\zeta}_2$ & ${\zeta}_3$ \\
\midrule
$L^2$ & 2.54 & 1.18 & 3.79 & 2.52 & 6.50 & 4.18 \\
$H^1$ & 5.26 & 3.93 & 9.48 & 5.95 & 32.66 & 21.87 \\
$H^2$ & 12.56 & 12.22 & 15.74 & 22.61 & 216.01 & 133.31 \\
$H^3$ & 40.68 & 49.34 & 38.32 & 162.73 & 2307.32 & 1328.67 \\
\bottomrule
\end{tabular}
\caption{Apparent upper bounds for Sobolev norms of $\eta_i$, $\zeta_i$ over the interval $[0,t_{\alpha}+10^{-3}]$, $[0,t_{\omega}+10^{-3}]$, respectively.}\label{starnorms}
\end{table}

These low-precision estimates function as good heuristics but do not qualify as a solution. To produce a true solution, we instead produce \textit{approximate} solutions that are reasonably close to these heuristics, and then perturb these approximations. The production of these approximate solutions will not occur until Section \ref{numerics}. Instead, we now identify properties that we want these approximations to satisfy; for example, that our approximate solution is reasonably close to the heuristics we already found. In the sequel, we let $\hat{x}$ denote numerical approximations of a particular quantity $x$.

\begin{ass}\label{aposerror}
There exist four real numbers $\hat{\alpha},\hat{\omega},\hat{t}_{\hat{\alpha}},\hat{t}_{\hat{\omega}}$ and smooth functions 
\begin{align*}
    \hat{\eta}:[0,\hat{t}_{\hat{\alpha}}+10^{-3}]\to \mathbb{R}^5, \qquad \hat{\zeta}:[0,\hat{t}_{\hat{\omega}}+10^{-3}]\to \mathbb{R}^5
\end{align*} so that the following quantities are all less than $\varepsilon$:
 \begin{itemize}
     \item The $H^3$ norm of the differential equation errors 
    \begin{align*}
        \hat{E}_{1,\mathrm{start}}(t)=\frac{L_{29}\hat{\eta}(t)}{t}+B_{29}(\hat{\eta}(t),\hat{\eta}(t))-\hat{\eta}'(t), \qquad  \hat{E}_{1,\mathrm{end}}(t)=\frac{L_{92}\hat{\zeta}(t)}{t}+B_{92}(\hat{\zeta}(t),\hat{\zeta}(t))-\hat{\zeta}'(t)
    \end{align*}
    on $[0,\hat{t}_{\hat{\alpha}}+10^{-3}]$ and $[0,\hat{t}_{\hat{\omega}}+10^{-3}]$.
    \item The shooting problem error $|\hat{A}(\hat{\alpha})-\hat{\Omega}(\hat{\omega})|$, where $$\hat{A}(\hat{\alpha}) = \left(\frac13 \sqrt{
    110-2(\hat{\alpha}+\hat{\eta}_1(\hat{t}_{\hat{\alpha}}))^2-\left(\frac{1}{2}+\frac{1}{9}\right)\hat{\eta}_2(\hat{t}_{\hat{\alpha}})^2},\,\hat{\eta}_2(\hat{t}_{\hat{\alpha}})\right),\quad \hat{\Omega}(\hat{\omega}) = (\hat{\zeta}_1(\hat{t}_{\hat{\omega}})+\hat{\omega}, \hat{\zeta}_2 (\hat{t}_{\hat{\omega}})).$$
    \item The stopping time errors $|\frac{9}{\hat{t}_{\hat{\alpha}}}+\hat{\eta}_3(\hat{t}_{\hat{\alpha}})|$ and $|\frac{2}{\hat{t}_{\hat{\omega}}}+\hat{\zeta}_3(\hat{t}_{\hat{\omega}})|$.
 \end{itemize}
Furthermore, the quantities in (\ref{eq:ValEsts}) are accurate to within seven decimal places, all quantities listed in Table \ref{nonlinearstop} are accurate to within six decimal places, and the values in Table \ref{starnorms} are upper bounds on the relevant quantities.
\end{ass}
Essentially, we use $\varepsilon$ to quantity how good the solution is. We will proceed through the theory in Section \ref{Guarantees} keeping $\varepsilon$, and will only specify what it needs to be to guarantee existence of a true solution at the end of that section. 

We will later verify Assumption \ref{aposerror} in Section \ref{numerics}. However, even verifying this assumption for extremely small $\varepsilon$ is not enough to guarantee existence of a true solution; we still need a mechanism with which to fix the error. To that end, we study the linearisation of the problem at our approximate solutions by solving the equations
\begin{subequations}
\label{lineariseIC}
\begin{align}
        \mu_{\mathrm{IC}}'(t)&=\frac{1}{t}L_{29}\mu_{\mathrm{IC}}(t)+2B_{29}(\hat{\eta}(t),\mu_{\mathrm{IC}}(t)), \qquad  \mu_{\mathrm{IC}}(0)=(0,0,0,1,0)\\
    \nu_{\mathrm{IC}}'(t)&=\frac{1}{t}L_{92}\mu_{\mathrm{IC}}(t)+2B_{92}(\hat{\zeta}(t),\nu_{\mathrm{IC}}(t)), \qquad  \nu_{\mathrm{IC}}(0)=(0,0,0,1,0).
    \end{align}
\end{subequations}
Equations (\ref{lineariseIC}) are derived from (\ref{etaequationO}) and (\ref{zetaequation0}) by differentiating both sides with respect to $\alpha$. By replacing $\hat{\eta}$ and $\hat{\zeta}$ with our previous heuristics and solving these equations, we observe the properties for solutions of the linear problem listed in Table \ref{icnorms}. 

\begin{table}[h]
\centering
\begin{tabular}{crrrrrr}
\toprule
\multirow{2}{*}{Index $k$} & \multicolumn{4}{c}{Quantities at stopping time} & \multicolumn{2}{c}{Norms} \\
\cmidrule(lr){2-5} \cmidrule(lr){6-7}
& $\mu_{\mathrm{IC}}(t_{\alpha})$ & $\mu_{\mathrm{IC}}'(t_{\alpha})$ & $\nu_{\mathrm{IC}}(t_{\omega})$ & $\nu_{\mathrm{IC}}'(t_{\omega})$ & $\|(\mu_{\mathrm{IC}})_k\|_{H^3}$ & $\|(\nu_{\mathrm{IC}})_k\|_{H^3}$ \\
\midrule
1 & $-1.012567$ & $-0.009517$ & $-0.994609$ & $-0.210300$ & $27.63$ & $116.98$ \\
2 & $0.006202$ & $-0.168800$ & $0.599502$ & $0.259384$ & $35.52$ & $1818.89$  \\ 
3 & $0.146589$ & $0.001461$ & $0.802500$ & $0.004078$ & $21.53$ & $1038.25$ \\
\bottomrule
\end{tabular}
\caption{Properties of solutions to the linearised problem}\label{icnorms}
\end{table}
We want to make sure that these properties remain accurate for the true solution of \eqref{lineariseIC}.
\begin{ass}\label{linearassumption}
Let $\hat{\eta},\hat{\zeta}$ be as in Assumption \ref{aposerror}. There are functions $\hat{\mu}_{\mathrm{IC}}$ and $\hat{\nu}_{\mathrm{IC}}$ which are initially $(0,0,0,1,0)$, have the same properties as in Table \ref{icnorms} (accurate to six decimal places), where presented norms are upper bounds, and the $H^3$ norm of the differential equation errors 
\begin{align*}
    \hat{E}_{2,\mathrm{start}}(t)&=(\hat{\mu}_{\mathrm{IC}})'(t)-\frac{1}{t}L_{29}\hat{\mu}_{\mathrm{IC}}(t)-2B_{29}(\hat{\eta}(t),\hat{\mu}_{\mathrm{IC}}(t)), \\
    \hat{E}_{2,\mathrm{end}}(t)&=(\hat{\nu}_{\mathrm{IC}})'(t)-\frac{1}{t}L_{92}\hat{\nu}_{\mathrm{IC}}(t)-2B_{92}(\hat{\zeta}(t),\hat{\nu}_{\mathrm{IC}}(t)).
\end{align*}
are all less than $\varepsilon$. 
\end{ass}


\section{Guarantees for a Solution}\label{Guarantees}

We proceed with our existence proof (Step \ref{step:Theory}) under Assumptions \ref{aposerror} and \ref{linearassumption}. This step involves fixed-point theory, combined with estimates on solutions of the ODEs \eqref{etaequationO} to prove existence of a new Einstein metric. 
The key theoretical tool we will use to estimate solutions of our differential equations, as well as to prove existence of a non-round Einstein metric, is the Schauder fixed point theorem \cite{Granas}. We present a slight refinement of this theorem that is useful for our purposes.
\begin{thm}[\textsc{Schauder}]\label{Schauder}
    Let $X$ be a Banach space, $L:X \to X$ a bounded linear operator, and $Q:X \to X$ a map satisfying $\|Q(x)\| \leq q\|x\|^2$ whenever $\|x\| \leq s_0$. If $\|x_0\| + \|L\| s + q s^2 \leq s$ for some $x_0\in X$ and $0 < s \leq s_0$, and $L,Q$ are both completely continuous, then the fixed point equation
    \[
    x = x_0 + L(x) + Q(x),
    \]
    has a solution $x^\ast \in X$ satisfying $\|x^\ast\| \leq s$.
\end{thm}
\begin{proof}
Letting $\mathcal{A}(x) = x_0 + L(x) + Q(x)$ for $x \in X$ and $K = \{x:\|x\|\leq s\}$, by hypothesis, $\mathcal{A}$ maps $K$ into itself. Since $K$ is convex and $\mathcal{A}$ is continuous and compact, the Schauder fixed point theorem \cite[Theorem 6.3.2]{Granas} implies $\mathcal{A}$ has a fixed point in $K$. 
\end{proof}
While Theorem \ref{Schauder} guarantees existence, we rely on Lemma \ref{uniquenessODE} to assert uniqueness as well.
The main goal of Step \ref{step:Theory} is to write the equation $A(\alpha)=\Omega(\omega)$ from \eqref{defAO} in such a fixed point form, and to estimate solutions of the associated ODEs to the extent that we gather enough information about the constants $\|x_0\|,l,q$ to establish the existence of a solution via the Schauder fixed point theorem.

\subsection{$C^k$ norms of the hat quantities}\label{Ckhat}
For this entire section, it is useful to have explicit control over the classical $C^k$ norms (for $k=0,1,2$) of the approximate solutions $\hat{\eta},\hat{\zeta}$ of the non-linear problem described in Assumption \ref{aposerror}, as well as our approximate solutions $\hat{\mu}_{\mathrm{IC}}$ and $\hat{\nu}_{\mathrm{IC}}$ of the linearised problem described in Assumption \ref{linearassumption}. 
\begin{lem}\label{hatcknorms}
    If $\hat{\eta}$ and $\hat{\zeta}$ satisfy Assumption \ref{aposerror}, then 
    \[
    \begin{aligned}
        \| \hat{\eta}\|_{C^0}&\le 16.1,\\
        \| \hat{\eta}\|_{C^1}&\le 63.3,\\
        \| \hat{\eta}\|_{C^2}&\le 212.5,
    \end{aligned}
    \qquad
    \begin{aligned}
        \| \hat{\zeta}\|_{C^0}&\le 50.3, \\
        \| \hat{\zeta}\|_{C^1}&\le 560.1, \\
        \| \hat{\zeta}\|_{C^2}&\le 5895.2.
    \end{aligned}
    \]
\end{lem}
\begin{proof}
     Since $\hat{\eta}_i(0)=\hat{\zeta}_i(0)=0$ for $i=1,2,3$, Assumption \ref{aposerror} (specifically, the estimates of Table \ref{starnorms})
     combined with the Sobolev Embedding Theorem (Lemma \ref{Sobolevembedding} (a)) gives 
     \[
     \begin{aligned}
     \|\hat{\eta}_1\|_{C^0} &\le 6.6,\\
     \|\hat{\zeta}_1\|_{C^0}&\le 7.5,
     \end{aligned}
     \qquad
     \begin{aligned}
         \|\hat{\eta}_2\|_{C^0}&\le 5.0,\\
         \|\hat{\zeta}_2\|_{C^0}&\le 40.9,
     \end{aligned}
     \qquad
     \begin{aligned}
         \|\hat{\eta}_3\|_{C^0}&\le 11.9,\\
         \|\hat{\zeta}_3\|_{C^0}&\le 27.4,
     \end{aligned}
     \]
     so that 
     \begin{align*}
         \|\hat{\eta}\|_{C^0}&\le \sqrt{(6.6)^2+(5.0)^2+(11.9)^2+\hat{\alpha}^2+\lambda}\le 16.1,\\
         \|\hat{\zeta}\|_{C^0}&\le \sqrt{(7.5)^2+(40.9)^2+(27.4)^2+\hat{\omega}^2+\lambda}\le 50.3.
     \end{align*}
     Proceeding with higher norms, we use Lemma \ref{Sobolevembedding}(b) to obtain 
     \[
     \begin{aligned}
     \|\hat{\eta}_1'\|_{C^0} &\le 25.2,\\
     \|\hat{\zeta}_1'\|_{C^0}&\le 45.3,
     \end{aligned}
     \qquad
     \begin{aligned}
         \|\hat{\eta}_2'\|_{C^0}&\le 24.5,\\
         \|\hat{\zeta}_2'\|_{C^0}&\le 432.1,
     \end{aligned}
     \qquad
     \begin{aligned}
         \|\hat{\eta}_3'\|_{C^0}&\le 31.5,\\
         \|\hat{\zeta}_3'\|_{C^0}&\le 266.7,
     \end{aligned}
     \]
     and
     \[
     \begin{aligned}
     \|\hat{\eta}_1''\|_{C^0} &\le 81.4,\\
     \|\hat{\zeta}_1''\|_{C^0}&\le 325.5,
     \end{aligned}
     \qquad
     \begin{aligned}
         \|\hat{\eta}_2''\|_{C^0}&\le 98.7,\\
         \|\hat{\zeta}_2''\|_{C^0}&\le 4614.7,
     \end{aligned}
     \qquad
     \begin{aligned}
         \|\hat{\eta}_3''\|_{C^0}&\le 76.7,\\
         \|\hat{\zeta}_3''\|_{C^0}&\le 2657.4.
     \end{aligned}
     \]
     Taking the root of the sum of squares as before gives the result.
\end{proof}
We similarly estimate $\hat{\mu}_{\mathrm{IC}}$ and $\hat{\nu}_{\mathrm{IC}}$, although since we only recorded the $H^3$ norms in Table \ref{icnorms} we will only estimate the entire $C^2$ norm (rather than the $C^1$ and $C^0$ norms individually). 
\begin{lem}\label{linearhat2ndd}
If $\hat{\mu}_{\mathrm{IC}}$ and $\hat{\nu}_{\mathrm{IC}}$ satisfy Assumption \ref{linearassumption} (with respect to a choice of $\hat{\eta}$ and $\hat{\zeta}$ satisfying Assumption \ref{aposerror}), then 
\begin{align*}
            \|\hat{\mu}_{\mathrm{IC}}\|_{C^2}, \quad \| \hat{\nu}_{\mathrm{IC}}\|_{C^2}\le 10^{5}.
    \end{align*}
\end{lem}
\begin{proof}
 As before, we can estimate the $C^0$ norms of $\hat{\mu}_{\mathrm{IC}},\hat{\mu}_{\mathrm{IC}}$ as well as their first and second derivatives using (b) of Lemma \ref{Sobolevembedding}. Using only the $H^3$ norm, we find $\|f\|_{C^2} \leq 6 \|f\|_{H^3}$. The result follows from Table \ref{icnorms}.
\end{proof}
Finally, we can also get $C^2$ norms for the errors. 
\begin{lem}\label{errorc2bounds}
  The $C^2$ norms of $\hat{E}_{1,\mathrm{start}}$, $\hat{E}_{1,\mathrm{end}}$, $\hat{E}_{2,\mathrm{start}}$, $\hat{E}_{2,\mathrm{end}}$ are all less than $6\varepsilon$. 
\end{lem}
\begin{proof}
    By Assumptions \ref{aposerror} and \ref{linearassumption}, the $H^3$ norms of these four quantities are all under $\varepsilon$. Then using (b) of Lemma \ref{Sobolevembedding}, we conclude  
    \begin{align*}
        \left|\left|\hat{E}_{1,\text{start}}\right|\right|_{C^2}&=\left|\left|\hat{E}_{1,\text{start}}\right|\right|_{C^0}+\left|\left|(\hat{E}_{1,\text{start}})'\right|\right|_{C^0}+\left|\left|(\hat{E}_{1,\text{start}})''\right|\right|_{C^0} \\
        &\le 2 \left|\left|\hat{E}_{1,\text{start}}\right|\right|_{H^1}+2\left|\left|(\hat{E}_{1,\text{start}})'\right|\right|_{H^1}+2\left|\left|(\hat{E}_{1,\text{start}})''\right|\right|_{H^1}\\
        &\le 6 \left|\left|\hat{E}_{1,\text{start}}\right|\right|_{H^3}\\
        &\le 6\varepsilon, 
    \end{align*}
    and similarly for the other quantities. 
\end{proof}

\subsection{Linear theory}\label{linearisation}
Having now estimated the size of our approximate solutions $\hat{\eta}$ and $\hat{\zeta}$ in Section \ref{Ckhat}, we proceed with the steps required to show that this approximate solution can be perturbed to a true solution, i.e., there is a true solution nearby. Recall that we let $\eta(\alpha,t)$ be the true solution of \eqref{etaequationO} for $d_1=2$, $d_2=9$ with the initial condition $\eta(\alpha,0)=(0,0,0,\alpha,\sqrt{\lambda})$, and similarly define $\zeta(\omega,t)$ to be the true solution of \eqref{etaequationO}, with $d_1=9$, $d_2=2$, satisfying the initial condition $\zeta(\omega,0)=(0,0,0,\omega,\sqrt{\lambda})$, specifically, we have 
\begin{align}
\label{etaequation1}
    \eta'(t)=\frac{1}{t}L_{29}\eta(t)+B_{29}(\eta(t),\eta(t)), \qquad \eta(0)=(0,0,0,\alpha,\sqrt{\lambda}). 
\end{align}
and
\begin{align}
\label{zetaequation1}
    \zeta'(t)=\frac{1}{t}L_{92}\zeta(t)+B_{92}(\zeta(t),\zeta(t)), \qquad \zeta(0)=(0,0,0,\omega,\sqrt{\lambda}). 
\end{align}
Define the perturbation functions as    $\mu(\alpha,t)=\eta(\alpha,t)-\hat{\eta}(t)$ and $\nu(\omega,t)=\zeta(\omega,t)-\hat{\zeta}(t)$, so that, as functions of time, they satisfy the IVPs
\begin{align}\label{mueq}
    \mu'(t)=\frac{L_{29} \mu(t)}{t}+2B_{29}(\hat{\eta}(t),\mu(t))+\hat{E}_{1,\text{start}}(t)+B_{29}(\mu(t),\mu(t)), \quad \mu(0)=(0,0,0,\alpha-\hat{\alpha},0)
\end{align}
and 
\begin{align}\label{nueq}
    \nu'(t)=\frac{L_{92} \nu(t)}{t}+2B_{92}(\hat{\zeta}(t),\nu(t))+\hat{E}_{1,\text{end}}(t)+B_{92}(\nu(t),\nu(t)), \quad \nu(0)=(0,0,0,\omega-\hat{\omega},0).
\end{align}
Recall that we are looking for a solution of $A(\alpha)=\Omega(\omega)$; this equation says that the two solutions $\eta,\zeta$ corresponding to $\alpha$ and $\omega$ respectively meet up smoothly in the middle (at the principal orbit whose mean curvature vanishes). The estimates of Assumption \ref{aposerror} imply that we should look for solutions with $\alpha-\hat{\alpha}$ and $\omega-\hat{\omega}$ quite small. It therefore follows that most of the relevant dynamics will come from the linear inhomogeneous equation, i.e., we ignore the non-linear $B_{29}$ and $B_{92}$ terms in \eqref{mueq} and \eqref{nueq}. To this end, we now study the linearisation of these two problems around the reference solutions. For the sake of clarity, we specify that in this subsection, we are particularly interested in solutions of the linear inhomogeneous problems
\begin{align}\label{mulinear}
     \mu'(t)&=\frac{L_{29} \mu(t)}{t}+2B_{29}(\hat{\eta}(t),\mu(t))+F_{\text{start}}(t), \ t\in [0,\hat{t}_{\hat{\alpha}}+10^{-3}]
     \end{align}
     and 
     \begin{align}\label{nulinear}
      \nu'(t)&=\frac{L_{92} \nu(t)}{t}+2B_{92}(\hat{\zeta}(t),\nu(t))+F_{\text{end}}(t),\ t\in [0,\hat{t}_{\hat{\alpha}}+10^{-3}].
\end{align}
It is straightforward to demonstrate that if $F_{\text{start}}$ and $F_{\text{end}}$ are smooth, then the solutions of these two problems, subject to the initial conditions $\mu(0)=0$ and $\nu(0)=0$ exist and are unique, and we write these solutions as 
\begin{align*}
    \mu=\mathcal{S}_{29}F_{\text{start}}, \qquad \nu=\mathcal{S}_{92}F_{\text{end}},
\end{align*}
for linear solution operators $\mathcal{S}_{29}$ and $\mathcal{S}_{92}$. 
The key result in this subsection is the following.

\begin{prop}\label{GreenEstimate}
   Under Assumption \ref{aposerror}:
   \begin{enumerate}[label=(\alph*)]
   \item The linear operators 
   \begin{align*}
       \mathcal{S}_{29}&: L^2 ([0,\hat{t}_{\hat{\alpha}}+10^{-3}];\mathbb{R}^5)\to C^0([0,\hat{t}_{\hat{\alpha}}+10^{-3}];\mathbb{R}^5), \\ \mathcal{S}_{92}&: L^2 ([0,\hat{t}_{\hat{\omega}}+10^{-3}];\mathbb{R}^5)\to C^0([0,\hat{t}_{\hat{\omega}}+10^{-3}];\mathbb{R}^5) 
   \end{align*} are bounded, with $\|\mathcal{S}_{29}\|,\|\mathcal{S}_{92}\| \leq M$ where $M\coloneqq e^{250} \leq 10^{109}$. 
\item For each $k=0,1$, the restriction of $\mathcal{S}_{29}$ to $C^k ([0,\hat{t}_{\hat{\alpha}}+10^{-3}];\mathbb{R}^5)$ maps into $C^{k+1} ([0,\hat{t}_{\hat{\alpha}}+10^{-3}];\mathbb{R}^5)$, and the two norms are less than $10^{5} M$ and $10^{13}M$, respectively. The same estimates hold for the restriction of the linear operator $\mathcal{S}_{92}$ to $C^k ([0,\hat{t}_{\hat{\omega}}+10^{-3}];\mathbb{R}^5)$. 
   \end{enumerate}
\end{prop}
\begin{proof}
If $\mu$ solves \eqref{mulinear} and satisfies $\mu(0)=0$, we combine Lemma \ref{linearprops} (a) and Lemma  \ref{Btermok} (both from Appendix A) to estimate
\begin{align*}
    (|\mu(t)|^2)'&\le 4 C(2,9,\hat{\eta}(t)) |\mu(t)|^2+2|F_{\text{start}}(t)||\mu(t)|\\
    &\le 4\left(C(2,9,\hat{\eta}(t))+\frac{1}{4}\right)|\mu(t)|^2+|F_{\text{start}}(t)|^2. 
\end{align*}
Gr\"onwall's inequality then gives 
\begin{align*}
    \left|\mu(t)\right|^2 \le \mathcal{I}_{start}(t)\left(\int_0^t \mathcal{I}_{start}(s)^{-1}\left|F_{\text{start}}(s)\right|^2 \dd s\right),
\end{align*}
where $\mathcal{I}_{start}(t)=e^{\int_0^t 4 C(2,9,\hat{\eta}(s))+1 \dd s}\ge 1.$
On the other hand, we similarly obtain 
\begin{align*}
    \left|\nu(t)\right|^2\le \mathcal{I}_{\text{end}}(t)\left(\int_0^t \mathcal{I}_{\text{end}}(s)^{-1}\left|F_{\text{end}}(s)\right|^2\dd s\right),
\end{align*}
where $\mathcal{I}_{\text{end}}(t)=e^{\int_0^t 4 C(9,2,\hat{\zeta}(s))+1 \dd s}\ge 1$, 
if $\nu$ solves \eqref{nulinear} and $\nu(0)=0$. 
To estimate both $\mathcal{I}_{start}$ and $\mathcal{I}_{\text{end}}$, we recall the definition $C$ in Lemma \ref{Btermok}:
 \begin{align*}
       C(d_1,d_2,y)&= |y_1|\left(d_1+\frac{1}{4d_1}\right)+|y_2|\left(\frac{1}{2d_1}+\frac{3}{2d_2}+\frac{1}{4}\right)\\&+|y_3|\left(\frac{1}{2d_2}+\frac{1}{2}\right)+|y_4|\left(d_1+\frac{1}{4d_1}\right)+|y_5|\left(\frac{d_1+1}{2}\right).
   \end{align*}
   As a consequence, we use the $L^2$ estimates for $\hat{\eta}$ and $\hat{\zeta}$ described in Table \ref{starnorms}, and H\"older's inequality $\|f\|_{L^{1}([0,\hat{t}_{\hat{\alpha}}+10^{-3}])}\leq\sqrt{\frac{3}{2}}\|f\|_{L^{2}([0,\hat{t}_{\hat{\alpha}}+10^{-3}])}$ to estimate 
   \begin{align*}
       &\int_{0}^{\hat{t}_{\hat\alpha}+10^{-3}} 4 C(2,9,\hat{\eta})+1\\
       &\le 1.5+\sqrt{1.5}\times  4\times  \left(2.125\times 2.54+0.7\times 1.18 +0.6\times 3.79 +2.125\times 6.1+1.5\times \sqrt{11} \right)\\
       &\le 132,
   \end{align*}
   and similarly,
   \begin{align*}
       &\int_{0}^{\hat{t}_{\hat \omega}+10^{-3}} 4 C(9,2,\hat{\zeta})+1\\
       &\le 1.2 +\sqrt{1.2}\times 4\times \left(9.1\times 2.52 +1.1\times 6.5+0.75\times 4.18+9.1\times 6.2 +5\times \sqrt{11}\right)\\
       &\le 470,
   \end{align*}
   implying that both integrating factors are under $e^{500}$, implying  
\begin{align*}
\|\mu\|_{L^{\infty}}\leq e^{250} \|F_{\text{start}}\|_{L^2}, \qquad \|\nu\|_{L^{\infty}}\leq e^{250} \|F_{\text{end}}\|_{L^2}
\end{align*}
 This shows (a). 

To show (b), observe that if $\mu=\mathcal{S}_{29}F_{\text{start}}$, and $\nu=\mathcal{S}_{92}F_{\text{end}}$, then the linear solution operators $\mathcal{L}_{29}$ and $\mathcal{L}_{92}$ (described in Appendix A) can be used to write 
\begin{align*}
    \mu=\mathcal{L}_{29}(F_{\text{start}}+2 B_{29}(\hat{\eta},\mu)), \qquad \nu=\mathcal{L}_{92}(F_{\text{end}}+2B_{92}(\hat{\zeta},\nu)). 
\end{align*} Then Lemmas \ref{singularlinear} (with $k=0$), \ref{hatcknorms} and \ref{BE}, combined with part (a) of this Proposition give 
\begin{align*}
    \left|\left|\mu\right|\right|_{C^1}&\le 25 \left(\left|\left|F_{\text{start}}\right|\right|_{C^0}+60 \left|\left|\hat{\eta}\right|\right|_{C^0}\left|\left|\mu\right|\right|_{C^0} \right)\\
    &\le 25 \left(\left|\left|F_{\text{start}}\right|\right|_{C^0}+60 \times \sqrt{1.5} M\left|\left|\hat{\eta}\right|\right|_{C^0}\left|\left|F_{\text{start}}\right|\right|_{C^0} \right)
\end{align*}
and similarly, 
\begin{align*}
    \left|\left|\nu\right|\right|_{C^1}\le 10^{5}M\left|\left|F_{\text{end}}\right|\right|_{C^0}.
\end{align*}If in addition, $F_{\text{start}}$ is bounded in $C^1$, then 
\begin{align*}
    \left|\left|2B_{29}(\hat{\eta},\mu)+F_{\text{start}}\right|\right|_{C^1}&\le \left|\left|F_{\text{start}}\right|\right|_{C^1}+2 \left|\left|B_{29}(\hat{\eta},\mu)\right|\right|_{C^0}+2 \left|\left|B_{29}(\hat{\eta}',\mu)\right|\right|_{C^0}+2 \left|\left|B_{29}(\hat{\eta},\mu')\right|\right|_{C^0}\\
    &\le \left|\left|F_{\text{start}}\right|\right|_{C^1}+60  \left|\left|\hat{\eta}\right|\right|_{C^1} \left|\left|\mu\right|\right|_{C^1}\\
    &\le (1+10^{5}\cdot 60\cdot \left|\left|\hat{\eta}\right|\right|_{C^1}M)\left|\left|F_{\text{start}}\right|\right|_{C^1},
\end{align*} so Lemmas \ref{hatcknorms} and \ref{singularlinear} give the required $C^2$ estimate for $\mu$. The comptuation for $\nu$ is identical. 
\end{proof}
Now that we have studied the linearisation as a whole, we look at the specific linear problem that arises by differentiating \eqref{mueq} and \eqref{nueq} with respect to $\alpha$ and $\omega$, respectively: 
\begin{align}\label{munuIC}
\begin{split}
        \mu_{\mathrm{IC}}'(t)&=\frac{1}{t}L_{29}\mu_{\mathrm{IC}}(t)+2B_{29}(\hat{\eta}(t),\mu_{\mathrm{IC}}(t)), \qquad  \mu_{\mathrm{IC}}(0)=(0,0,0,1,0)\\
    \nu_{\mathrm{IC}}'(t)&=\frac{1}{t}L_{92}\mu_{\mathrm{IC}}(t)+2B_{92}(\hat{\zeta}(t),\nu_{\mathrm{IC}}(t)), \qquad  \nu_{\mathrm{IC}}(0)=(0,0,0,1,0).
    \end{split}
    \end{align}

\begin{cor}\label{linearIC}
  Under Assumptions \ref{aposerror} and \ref{linearassumption}, we have 
  \begin{align*}
      \|\mu_{\mathrm{IC}}-\hat{\mu}_{\mathrm{IC}}\|_{C^2}, \|\nu_{\mathrm{IC}}-\hat{\nu}_{\mathrm{IC}}\|_{C^2}\le 10^{14}M\varepsilon,
  \end{align*}
  and 
  \begin{align*}
           \|\mu_{\mathrm{IC}}\|_{C^2}, \|\nu_{\mathrm{IC}}\|_{C^2}\le 10^5+10^{14}M\varepsilon.
       \end{align*}
  Furthermore, the values of the first three components of $\mu_{\mathrm{IC}}(\hat{t}_{\hat\alpha})$, $\mu_{\mathrm{IC}}'(\hat{t}_{\hat\alpha})$, $\nu_{\mathrm{IC}}(\hat{t}_{\hat\omega})$ and $\nu_{\mathrm{IC}}'(\hat{t}_{\hat\omega})$ are given by the corresponding entries of Table \ref{icnorms}, to error $10^{-5}+10^{14}M\varepsilon$. 

\end{cor}
\begin{proof}
 By Assumption \ref{linearassumption}, we have  
    \[
        \mu_{\mathrm{IC}}-\hat{\mu}_{\mathrm{IC}}=\mathcal{S}_{29}\hat{E}_{2,\text{start}},\qquad 
         \nu_{\mathrm{IC}}-\hat{\nu}_{\mathrm{IC}}=\mathcal{S}_{92}\hat{E}_{2,\text{end}}.
    \]
    Statement (b) of Proposition 
 \ref{GreenEstimate} combined with Lemma \ref{errorc2bounds} gives
 $$\|\mu_{\mathrm{IC}}-\hat{\mu}_{\mathrm{IC}}\|_{C^{2}}\leq\|\mathcal{S}_{29}\hat{E}_{2,\text{start}}\|_{C^{2}}\leq10^{13}M\|\hat{E}_{2,\text{start}}\|_{C^{1}}\leq10^{14}M\epsilon.$$ The estimate for $\nu_{\mathrm{IC}}-\hat{\nu}_{\mathrm{IC}}$ follows similarly. The  $C^2$ estimates on $\mu_{\mathrm{IC}}$ and $\nu_{\mathrm{IC}}$ follow from the triangle inequality with Lemma \ref{linearhat2ndd}. The rest follows from Assumption \ref{linearassumption}.
\end{proof}

\subsection{Estimates on $\mu,\nu$}\label{nonlinearestimates}
Recall from Section \ref{linearisation} that we use  $\mu(\alpha,t)=\eta(\alpha,t)-\hat{\eta}(t)$ and $\nu(\omega,t)=\zeta(\omega,t)-\hat{\zeta}(t)$ to track the errors between the true solutions $(\eta,\zeta)$ of \eqref{mueq}-\eqref{nueq} for various choices of $\alpha,\omega$, and the approximate solutions $\hat{\eta},\hat{\zeta}$ under Assumption \ref{aposerror}. Expecting a solution $A(\alpha)=\Omega(\omega)$ for $(\alpha,\omega)$ close to $(\hat{\alpha},\hat{\omega})$, the purpose of this section is to expand $\mu$ and $\nu$ locally around $\hat{\alpha}$ and $\hat{\omega}$, that is,   
\begin{align}\label{munualomex}
\begin{split}      \mu(\alpha,t)&=\mu^{(0)}(t)+(\alpha-\hat{\alpha})\mu^{(1)}(t)+(\alpha-\hat{\alpha})^2\mu^{(2)}(t,\alpha),\\
        \nu(\omega,t)&=\nu^{(0)}(t)+(\omega-\hat{\omega})\nu^{(1)}(t)+(\omega-\hat{\omega})^2\nu^{(2)}(t,\omega),
        \end{split}
    \end{align}
    and to provide informative estimates for all of these time-varying coefficient functions. In Section \ref{existence}, these estimates on $\mu$, $\nu$ will be used to study $A(\alpha)-\Omega(\omega)$ for $(\alpha,\omega)$ close to $(\hat{\alpha},\hat{\omega})$, and ultimately prove existence of a new zero corresponding to a non-round Einstein metric.
    \begin{lem}\label{munu0terms}
       If Assumption \ref{aposerror} holds for a choice of $\varepsilon$ satisfying 
 $M^2\varepsilon\le \frac{1}{200}$, then the constant-order terms of the expansion \eqref{munualomex} satisfy for $k=0,1,2$, 
       \begin{align*}
       \|\mu^{(0)}\|_{C^k[0,\hat{t}_{\hat{\alpha}}+10^{-3}]},\|\nu^{(0)}\|_{C^k[0,\hat{t}_{\hat{\omega}}+10^{-3}]}\le 10^{c_k} M\varepsilon,
       \end{align*}
      where $c_0=1$, $c_1=6$, $c_2=20$. 
    \end{lem}
    \begin{proof}
By setting $\alpha=\hat{\alpha}$ and $\omega=\hat{\omega}$ in \eqref{mueq} and \eqref{nueq}, we find that the zeroth order terms satisfy the equations 
        \begin{align}\label{munuaoexp}
        \begin{split}
(\mu^{(0)})'(t)&=\frac{L_{29} \mu^{(0)}(t)}{t}+2B_{29}(\hat{\eta}(t),\mu^{(0)}(t))+\hat{E}_{1,\text{start}}(t)+B_{29}(\mu^{(0)}(t),\mu^{(0)}(t)), \quad \mu^{(0)}(0)=0,\\
    (\nu^{(0)})'(t)&=\frac{L_{92} \nu^{(0)}(t)}{t}+2B_{92}(\hat{\eta}(t),\nu^{(0)}(t))+\hat{E}_{1,\text{end}}(t)+B_{92}(\nu^{(0)}(t),\nu^{(0)}(t)), \quad \nu^{(0)}(0)=0.
        \end{split}
    \end{align}
Using the linear solutions operators $\mathcal{S}_{29}$ and $\mathcal{S}_{92}$ from Section \ref{linearisation}    We can thus write    \begin{align}\label{FPmu0nu0}
        \mu^{(0)}=\mathcal{S}_{29} \left(\hat{E}_{1,\text{start}}+B_{29}(\mu^{(0)},\mu^{(0)})\right), \qquad \nu^{(0)}=\mathcal{S}_{92}(\hat{E}_{1,\text{end}}+B_{92}(\nu^{(0)},\nu^{(0)})).
    \end{align}
    Observe that by Proposition \ref{GreenEstimate} (a),  Lemma \ref{BE} and Assumption \ref{aposerror}, 
    \begin{align*}
        \left|\left|\mathcal{S}_{29} \left(\hat{E}_{1,\text{start}}+B_{29}(\mu^{(0)},\mu^{(0)})\right) \right|\right|_{C^0}&\le M\left(\varepsilon+\|B_{29}(\mu^{(0)},\mu^{(0)})\|_{L^2}\right)\\
        &\le M\left(\varepsilon+\sqrt{1.5}\times 30 \|\mu^{(0)}\|_{C^0}^2\right).
    \end{align*}
    Since $M^2\varepsilon\le \frac{1}{200}$, the Schauder Fixed Point Theorem (Theorem \ref{Schauder}) asserts that there is a solution (unique as far as it is defined) with $\|\mu^{(0)}\|_{C^0}\le 2M\varepsilon$. Similarly, we obtain $\|\nu^{(0)}\|_{C^0}\le 2M\varepsilon$. Feeding these bounds into \eqref{FPmu0nu0}, and combining with Proposition  \ref{GreenEstimate} (b) Lemma \ref{BE} and Lemma \ref{errorc2bounds} we obtain 
    \begin{align*}
        \|\mu^{(0)}\|_{C^1}&\le 10^5 M \left(\|\hat{E}_{1,\text{start}}\|_{C^0}+\|B_{29}(\mu^{(0)},\mu^{(0)})\|_{C^0}\right)\\
        &\le 10^5 M \left(6\varepsilon+30 \times (2M\varepsilon)^2\right)\\
        &\le 10^6 M\varepsilon,
    \end{align*}
    again using $M^2 \varepsilon\le \frac{1}{200}$. The estimate for $\|\nu^{(0)}\|_{C^1}$ is similar. Finally, again using Proposition \ref{GreenEstimate} (b) Lemma \ref{BE} and Lemma \ref{errorc2bounds} we obtain
    \begin{align*}
        \|\mu^{(0)}\|_{C^2}& \le 10^{13}M(\|\hat{E}_{1,\text{start}}\|_{C^1}+\|B_{29}(\mu^{(0)},\mu^{(0)})\|_{C^1})\\
        &\le 10^{13}M(6\varepsilon+\|B_{29}(\mu^{(0)},\mu^{(0)})\|_{C^0}+2\|B_{29}(\mu^{(0)},(\mu^{(0)})')\|_{C^0} )\\
        &\le 10^{13}M \left(6\varepsilon+30\times (2M\varepsilon)^2+60 (2M\varepsilon)(10^{6}M\varepsilon)\right)\\
        &\le 10^{20}M\varepsilon, 
    \end{align*}
    and similarly for $\|\nu^{(0)}\|_{C^2}$. 
    \end{proof}
    We can also estimate the first-order terms of \eqref{munualomex}.
    \begin{lem}\label{mu1est}
Assuming that $M^2 \varepsilon\le 10^{-20}$, the first order terms satisfy for $k = 0,1,2$,
\begin{align*}
    \|\mu^{(1)}-\mu_{\mathrm{IC}}\|_{C^k[0,\hat{t}_{\hat{\alpha}}+10^{-3}]},\|\nu^{(1)}-\nu_{\mathrm{IC}}\|_{C^k[0,\hat{t}_{\hat{\omega}}+10^{-3}]}\le 10^{c_k}M^2\varepsilon,
\end{align*} where $\mu_{\mathrm{IC}}$ and $\nu_{\mathrm{IC}}$ are defined according to \eqref{munuIC}, and the constants are given by $c_0=10$, $c_1=13$, $c_2=30$.
    \end{lem}
    \begin{proof}
        By putting the expression \eqref{munualomex} into \eqref{mueq} and \eqref{nueq}, and isolating the linear terms (in $\alpha-\hat{\alpha}$ or $\omega-\hat{\omega}$), we obtain 
        \begin{align*}
            (\mu^{(1)})'(t)&=\frac{L_{29} \mu^{(1)}(t)}{t}+2B_{29}(\hat{\eta}(t)+\mu^{(0)}(t),\mu^{(1)}(t)), \qquad \mu^{(1)}(0)=(0,0,0,1,0),\\(\nu^{(1)})'(t)&=\frac{L_{92} \nu^{(1)}(t)}{t}+2B_{92}(\hat{\zeta}(t)+\nu^{(0)}(t),\nu^{(1)}(t)), \qquad \nu^{(1)}(0)=(0,0,0,1,0).\\
        \end{align*}
       Using the definition of $\mu_{\mathrm{IC}},\nu_{\mathrm{IC}}$ \eqref{munuIC}, we thus find 
\begin{align}\label{mu1nu1diff}
       \begin{split}
    \left(\mu^{(1)}(t)-\mu_{\mathrm{IC}}(t)\right)'&=\frac{L_{29}}{t}\left(\mu^{(1)}(t)-\mu_{\mathrm{IC}}(t)\right)+2B_{29}(\hat{\eta}(t),\mu^{(1)}(t)-\mu_{\mathrm{IC}}(t))+2B_{29}(\mu^{(0)}(t),\mu^{(1)}(t)), \\
           \left(\nu^{(1)}(t)-\nu_{\mathrm{IC}}(t)\right)'&=\frac{L_{92}}{t}\left(\nu^{(1)}(t)-\nu_{\mathrm{IC}}(t)\right)+2B_{92}(\hat{\zeta}(t),\nu^{(1)}(t)-\nu_{\mathrm{IC}}(t))+2B_{92}(\nu^{(0)}(t),\nu^{(1)}(t)),
           \end{split}
       \end{align}
       with $\mu^{(1)}(0)=\mu_{\mathrm{IC}}(0)$ and $\nu^{(1)}(0)=\nu_{\mathrm{IC}}(0)$. Therefore, we obtain the fixed-point equations 
       \begin{align}\label{mu1ICdif}
       \begin{split}
           \mu^{(1)}-\mu_{\mathrm{IC}}=\mathcal{S}_{29}\left(2B_{29}(\mu^{(0)},\mu^{(1)}-\mu_{\mathrm{IC}})+2B_{29}(\mu^{(0)},\mu_{\mathrm{IC}})\right),\\
           \nu^{(1)}-\nu_{\mathrm{IC}}=\mathcal{S}_{92}\left(2B_{92}(\nu^{(0)},\nu^{(1)}-\nu_{\mathrm{IC}})+2B_{92}(\nu^{(0)},\nu_{\mathrm{IC}})\right).
           \end{split}
       \end{align}
       Corollary \ref{linearIC} states that 
       \begin{align*}
           \left|\left|\mu_{\mathrm{IC}}\right|\right|_{C^2}\le 10^5+10^{14}M\varepsilon, \qquad \left|\left|\nu_{\mathrm{IC}}\right|\right|_{C^2}\le 10^5+10^{14}M\varepsilon.
       \end{align*}
We can then estimate the $C^0$ norms of the  right-hand sides of \eqref{mu1ICdif} using Proposition \ref{GreenEstimate} (a), Lemma \ref{BE} and the recent Lemma \ref{munu0terms}:
\begin{multline*}
\left
\| \mathcal{S}_{29}\left(2B_{29}(\mu^{(0)},\mu^{(1)}-\mu_{\mathrm{IC}})+2B_{29}(\mu^{(0)},\mu_{\mathrm{IC}})\right)\right\|_{C^0}\\ \le M\left(2 \, \|B_{29}(\mu^{(0)},\mu_{\mathrm{IC}})\|_{L^2}+2 \, \|B_{29}(\mu^{(0)},\mu^{(1)}-\mu_{\mathrm{IC}})\|_{L^2}\right)\\
         \le 5M\left( \|B_{29}(\mu^{(0)},\mu_{\mathrm{IC}})\|_{C^0}+ \|B_{29}(\mu^{(0)},\mu^{(1)}-\mu_{\mathrm{IC}})\|_{C^0}\right)\\
         \le 150 M \|\mu^{(0)}\|_{C^0}\left(\|\mu_{\mathrm{IC}}\|_{C^0}+\|\mu^{(1)}-\mu_{\mathrm{IC}}\|_{C^0}\right)\\
         \le  1.5\times 10^{3} M^2 \varepsilon \left(10^{5}+10^{14}M\varepsilon  +\|\mu^{(1)}-\mu_{\mathrm{IC}}\|_{C^0}\right).
\end{multline*}
     The Schauder Fixed Point Theorem then implies  $\|\mu^{(1)}-\mu_{\mathrm{IC}}\| M^2\varepsilon$, and similarly for $\|\nu^{(1)}-\nu_{\mathrm{IC}}\|_{C^0}$.  
Continuing on with the higher-order estimates, we use Proposition \ref{GreenEstimate} (b), Lemma \ref{BE} and Lemma \ref{munu0terms} to obtain 
\begin{align*}
\left\lVert \mathcal{S}_{29}\left(2B_{29}(\mu^{(0)},\mu^{(1)}-\mu_{\mathrm{IC}})\right.\right.&+\left.\left.2B_{29}(\mu^{(0)},\mu_{\mathrm{IC}})\right) \right\rVert_{C^1}
\\
&\le 60\times  10^5 M\|\mu^{(0)}\|_{C^0}(\|\mu_{\mathrm{IC}}\|_{C^0}+\|\mu^{(1)}-\mu_{\mathrm{IC}}\|_{C^0})\\
&\le 6\times 10^7 M^2\varepsilon (10^{5}+10^{14}M\varepsilon +10^{10}M^2\varepsilon),
\end{align*}
so $\|\mu^{(1)}-\mu_{\mathrm{IC}}\|_{C^1}\le 10^{13}M^2\varepsilon$, and similarly for $\nu^{(1)}-\nu_{\mathrm{IC}}$. Finally, 
\begin{align*}
    &\left\| \mathcal{S}_{29}\left(2B_{29}(\mu^{(0)},\mu^{(1)}-\mu_{\mathrm{IC}})+2B_{29}(\mu^{(0)},\mu_{\mathrm{IC}})\right)\right\|_{C^2}\\
    &\le 10^{13}M \|2B_{29}(\mu^{(0)},\mu^{(1)}-\mu_{\mathrm{IC}})+2B_{29}(\mu^{(0)},\mu_{\mathrm{IC}})\|_{C^1}\\
    &\le 10^{13}M\|2B_{29}(\mu^{(0)},\mu^{(1)}-\mu_{\mathrm{IC}})+2B_{29}(\mu^{(0)},\mu_{\mathrm{IC}})\|_{C^0}\\
    &+2\cdot 10^{13}M\|B_{29}((\mu^{(0)})',\mu^{(1)}-\mu_{\mathrm{IC}})+B_{29}(\mu^{(0)},(\mu^{(1)}-\mu_{\mathrm{IC}})')+B_{29}((\mu^{(0)})',\mu_{\mathrm{IC}})+B_{29}(\mu^{(0)},(\mu_{\mathrm{IC}})')\|_{C^0}.
\end{align*}
Using the triangle inequality, Lemma \ref{BE}, and the previous $C^1$ estimate on $\mu^{(1)}-\mu_{\mathrm{IC}}$, we obtain 
\begin{align*}
    \|\mu^{(1)}-\mu_{\mathrm{IC}}\|_{C^2}&\le 2\cdot60\cdot 3\cdot  10^{13}M \|\mu^{(0)}\|_{C^1}( \|\mu_{\mathrm{IC}}\|_{C^1}+\|\mu_{\mathrm{IC}}-\mu^{(1)}\|_{C^1})\\
    &\le 2\cdot60\cdot 3\cdot  10^{13}M \cdot 10^6 M\varepsilon\left( 10^{5}+10^{14}M\varepsilon+10^{13}M^2\varepsilon\right)\\
    &\le 10^{30}M^2\varepsilon.
\end{align*}We can similarly estimate the $C^2$ norm of $\nu^{(1)}-\nu_{\mathrm{IC}}$. 
    \end{proof}
    We find finally estimate the second-order terms. 
\begin{lem}\label{munusterm}
If $\left|\alpha-\hat{\alpha}\right|+\left|\omega-\hat{\omega}\right|\le \frac{10^{-30}}{M}$ and $M^2\varepsilon\le 10^{-20}$, the second order terms satisfy for $k=0,1,2$,
    \begin{align*}
\|\mu^{(2)}\|_{C^k},\|\nu^{(2)}\|_{C^k}\le 10^{c_k}M, 
    \end{align*}
    where $c_0=13$, $c_1=18$, $c_{2}=40$. 
\end{lem}
\begin{proof}
    These terms solve the equations 
    \begin{align*}
(\mu^{(2)})'(t)&=\frac{L_{29} \mu^{(2)}(t)}{t}+2B_{29}(\hat{\eta}(t)+\mu^{(0)}(t)+(\alpha-\hat{\alpha})\mu^{(1)}(t),\mu^{(2)}(t))\\
&+(\alpha-\hat{\alpha})^2B_{29}(\mu^{(2)}(t),\nu^{(2)}(t))+B_{29}(\mu^{(1)}(t),\mu^{(1)}(t))\\
(\nu^{(2)})'(t)&=\frac{L_{92} \nu^{(2)}(t)}{t}+2B_{92}(\hat{\zeta}(t)+\nu^{(0)}(t)+(\omega-\hat{\omega})\nu^{(1)}(t),\nu^{(2)}(t))\\
&+(\omega-\hat{\omega})^2B_{92}(\nu^{(2)}(t),\nu^{(2)}(t))+B_{92}(\nu^{(1)}(t),\nu^{(1)}(t)),
    \end{align*}
    and $\mu^{(2)}(0,\alpha)=0=\nu^{(2)}(0,\omega)$. 
   In fixed-point form, this is
   \begin{align*}
    \mu^{(2)}&=\mathcal{S}_{29}\left(2B_{29}(\mu^{(0)}+(\alpha-\hat{\alpha})\mu^{(1)},\mu^{(2)})+(\alpha-\hat{\alpha})^2 B_{29}(\mu^{(2)},\mu^{(2)})+B_{29}(\mu^{(1)},\mu^{(1)})\right),\\
    \nu^{(2)}&=\mathcal{S}_{92}\left(2B_{29}(\nu^{(0)}+(\omega-\hat{\omega})\nu^{(1)},\nu^{(2)})+(\omega-\hat{\omega})^2 B_{92}(\nu^{(2)},\nu^{(2)})+B_{92}(\nu^{(1)},\nu^{(1)})\right).
   \end{align*}
   Using Proposition \ref{GreenEstimate} (a) and Corollary \ref{linearIC}, alongside Lemmas \ref{linearhat2ndd}, \ref{BE}, \ref{munu0terms} and \ref{mu1est}, the $C^0$ norm of the right hand of the first side is bounded by 
   \begin{align*}
       &\sqrt{1.5}M\left(30\cdot \|\mu^{(1)}\|_{C^0}^2+60 (\|\mu^{(0)}\|_{C^0}+\left|\alpha-\hat{\alpha}\right| \|\mu^{(1)}\|_{C^0})\|\mu^{(2)}\|_{C^0}+30(\alpha-\hat{\alpha})^2 \|\mu^{(2)}\|_{C^0}\right)\\
       &\le \sqrt{1.5}M\left(30\cdot (10^{10}M^2\varepsilon+10^{5}+10^{14}M\varepsilon)^2+60 (10 M\varepsilon+\left|\alpha-\hat{\alpha}\right|(10^{10}M^2\varepsilon+10^{5}+10^{14}M\varepsilon))\|\mu^{(2)}\|_{C^0}\right)\\
       &+\sqrt{1.5} M\left(30(\alpha-\hat{\alpha})^2\|\mu^{(2)}\|_{C^0}^2\right)\\
       &\le \sqrt{1.5}M \left(10^{12}+(600M\varepsilon+ 10^6 \left|\alpha-\hat{\alpha}\right|)\|\mu^{(2)}\|_{C^0}+30 (\alpha-\hat{\alpha})^2 \|\mu^{(2)}\|_{C^0}^2\right).
   \end{align*}
   Therefore, if $\left|\alpha-\hat{\alpha}\right|\le \frac{10^{-30}}{M}$, then the Schauder fixed-point Theorem implies that there is a solution satisfying the estimate $\|\mu^{(2)}\|_{C^0}\le 10^{13}M$. The argument for $\|\nu^{(2)}\|_{C^0}$ is identical. We can then estimate the $C^1$ norm:
   \begin{align*}
       \|\mu^{(2)}\|_{C^1}&\le 10^5 M \left(30\cdot \|\mu^{(1)}\|_{C^0}^2+60 (\|\mu^{(0)}\|_{C^0}+\left|\alpha-\hat{\alpha}\right| \|\mu^{(1)}\|_{C^0})\|\mu^{(2)}\|_{C^0}+30(\alpha-\hat{\alpha})^2 \|\mu^{(2)}\|_{C^0}\right)\\
       &\le 10^{5}M\left(10^{12}+(600M\varepsilon+ 10^6 \left|\alpha-\hat{\alpha}\right|)\|\mu^{(2)}\|_{C^0}+30 (\alpha-\hat{\alpha})^2 \|\mu^{(2)}\|_{C^0}^2\right)\\
       &\le 10^{5}M\left(10^{12}+(600M\varepsilon+ 10^6 \left|\alpha-\hat{\alpha}\right|)10^{13}M+30 (\alpha-\hat{\alpha})^2 10^{26}M^2\right)\\
       &\le 10^{18}M,
   \end{align*}
   and similarly for $\|\nu^{(2)}\|_{C^1}$. Finally, for the $C^2$ norms, we use Proposition, as well as Lemmas \ref{linearhat2ndd}, \ref{BE},  \ref{GreenEstimate}, \ref{mu1est}, \ref{munu0terms}, Corollary \ref{linearIC} and our previous $\mu^{(2)}$ estimates in this same proof to find 
   \begin{align*}
       \|\mu^{(2)}\|_{C^2}&=\|\mathcal{S}_{29}\left(2B_{29}(\mu^{(0)}+(\alpha-\hat{\alpha})\mu^{(1)},\mu^{(2)})+(\alpha-\hat{\alpha})^2 B_{29}(\mu^{(2)},\mu^{(2)})+B_{29}(\mu^{(1)},\mu^{(1)})\right)\|_{C^2} \\
       &\le 10^{13}M\|2B_{29}(\mu^{(0)}+(\alpha-\hat{\alpha})\mu^{(1)},\mu^{(2)})+(\alpha-\hat{\alpha})^2 B_{29}(\mu^{(2)},\mu^{(2)})+B_{29}(\mu^{(1)},\mu^{(1)})\|_{C^1}\\
       &\le 10^{13}M\|2B_{29}(\mu^{(0)}+(\alpha-\hat{\alpha})\mu^{(1)},\mu^{(2)})+(\alpha-\hat{\alpha})^2 B_{29}(\mu^{(2)},\mu^{(2)})+B_{29}(\mu^{(1)},\mu^{(1)})\|_{C^0}\\
       &+10^{13}M\|2B_{29}((\mu^{(0)})'+(\alpha-\hat{\alpha})(\mu^{(1)})',\mu^{(2)})+2B_{29}(\mu^{(0)}+(\alpha-\hat{\alpha})\mu^{(1)},(\mu^{(2)})')\|_{C^0}\\
       &+
       10^{13}M\|2(\alpha-\hat{\alpha})^2 B_{29}(\mu^{(2)},(\mu^{(2)})')+2B_{29}(\mu^{(1)},(\mu^{(1)}))'\|_{C^0}\\
       &\le 10^{13}M \left(60\cdot \|\mu^{(2)}\|_{C^1} (\|
       \mu^{(0)}\|_{C^1}+\left|\alpha-\hat{\alpha}\right| \|
       \mu^{(1)}\|_{C^1})+60 (\alpha-\hat{\alpha})^2\|\mu^{(2)}\|^2_{C^1}+60 \|\mu^{(1)}\|_{C^1}^2\right)\\
       &\le 10^{13}M\left(10^{20}M \cdot (10^{6}M\varepsilon+\left|\alpha-\hat{\alpha}\right|10^{6})+60(\alpha-\hat{\alpha})^2\cdot 10^{18}M+60\cdot 10^{12}\right),
   \end{align*}
   and the result for $\|\mu^{(2)}\|_{C^2}$ follows, The $\nu^{(2)}$ function is treated similarly. 
\end{proof}

Having now estimated the coefficients of \eqref{munualomex}, we apply this information to obtain accurate information about what $\mu(t,\alpha)$ looks like for $t$ close to $\hat{t}_{\hat{\alpha}}$, and $\alpha$ close to $\hat{\alpha}$. Similarly, we are interested in how $\nu(t,\omega)$ behaves for $t$ close to $\hat{t}_{\hat{\omega}}$ and $\omega$ close to $\hat{\omega}$. 
\begin{cor}\label{mouat}
    We have 
    \begin{align}\label{mubigsplit}
    \begin{split}
        \mu(\alpha,t)&=\mu^{(0,0)}+\mu^{(1,0)}(\alpha-\hat{\alpha})+\mu^{(0,1)}(t-\hat{t}_{\hat{\alpha}})+\mu^{(1,1)}(\alpha-\hat{\alpha})(t-\hat{t}_{\hat{\alpha}})\\
        &+(t-\hat{t}_{\hat{\alpha}})^2\left(\mu^{(0,2)}(t)+(\alpha-\hat{\alpha})\mu^{(1,2)}(t)\right)+(\alpha-\hat{\alpha})^2\mu^{(2,0)}(\alpha,t),
        \end{split}
    \end{align}
    and 
    \begin{align}\label{nubigsplit}
    \begin{split}
        \nu(\omega,t)&=\nu^{(0,0)}+\nu^{(1,0)}(\omega-\hat{\omega})+\nu^{(0,1)}(t-\hat{t}_{\hat{\omega}})+\nu^{(1,1)}(\omega-\hat{\omega})(t-\hat{t}_{\hat{\omega}})\\
        &+(t-\hat{t}_{\hat{\omega}})^2\left(\nu^{(0,2)}(t)+(\omega-\hat{\omega})\nu^{(1,2)}(t)\right)+(\omega-\hat{\omega})^2\nu^{(2,0)}(\omega,t),
        \end{split}
    \end{align}
    where the constant vectors satisfy 
    \begin{align*}
        |\mu^{(0,0)}|\le 10 M\varepsilon, \qquad |\mu^{(0,1)}|\le 10^{6}M\varepsilon, \qquad |\mu^{(1,0)}-\hat{\mu}_{\mathrm{IC}}(\hat{t}_{\hat{\alpha}})|\le 10^{11}M^2\varepsilon, \qquad |\mu^{(1,1)}-\hat{\mu}'_{\mathrm{IC}}(\hat{t}_{\hat{\alpha}})|\le 10^{14}M^2\varepsilon,
    \end{align*}
    and similarly for the corresponding $\nu$ constants, while the time-varying functions satisfy 
    \begin{align*}
        \|\mu^{(0,2)}\|_{C^0}\le 10^{20}M\varepsilon, \qquad \|\mu^{(1,2)}\|_{C^0}\le 10^{6}, \qquad \|\nu^{(0,2)}\|_{C^0}\le 10^{20}M\varepsilon, \qquad \|\nu^{(1,2)}\|_{C^0}\le 10^{6}
    \end{align*}
    and finally, for $\left|\alpha-\hat{\alpha}\right|+\left|\omega-\hat{\omega}\right|\le \frac{10^{-30}}{M}$, the time-varying functions $\mu^{(2,0)}$ and $\nu^{(2,0)}$ satisfy
    \begin{align*}
        \|\mu^{(2,0)}\|_{C^2}\le 10^{40}M, \qquad \|\nu^{(2,0)}\|_{C^2}\le 10^{40}M.
    \end{align*}
\end{cor}
\begin{proof}The result follows by recalling the expansion from  \eqref{munualomex}:
\begin{align*}
\mu(\alpha,t)&=\mu^{(0)}(t)+(\alpha-\hat{\alpha})\mu^{(1)}(t)+(\alpha-\hat{\alpha})^2\mu^{(2)}(t,\alpha),\\
        \nu(\omega,t)&=\nu^{(0)}(t)+(\omega-\hat{\omega})\nu^{(1)}(t)+(\omega-\hat{\omega})^2\nu^{(2)}(t,\omega),
    \end{align*} 
   and then reading off 
    \begin{align*}
        \mu^{(0,0)}=\mu^{(0)}(\hat{t}_{\hat{\alpha}}), \qquad \mu^{(1,0)}=\mu^{(1)}(\hat{t}_{\hat{\alpha}}), \qquad \mu^{(0,1)}=(\mu^{(0)})'(\hat{t}_{\hat{\alpha}}), \qquad \mu^{(1,1)}=(\mu^{(1)})'(\hat{t}_{\hat{\alpha}}),\\
        \mu^{(0,2)}(t)=\frac{\mu^{(0)}(t)-\mu^{(0)}(\hat{t}_{\hat{\alpha}})-(t-\hat{t}_{\hat{\alpha}})(\mu^{(0)})'(\hat{t}_{\hat{\alpha}})}{(t-\hat{t}_{\hat{\alpha}})^2},\\
        \mu^{(1,2)}(t)=\frac{\mu^{(1)}(t)-\mu^{(1)}(\hat{t}_{\hat{\alpha}})-(t-\hat{t}_{\hat{\alpha}})(\mu^{(1)})'(\hat{t}_{\hat{\alpha}})}{(t-\hat{t}_{\hat{\alpha}})^2},\\
        \mu^{(2,0)}(\alpha,t)=\mu^{(2)}(\alpha,t), 
    \end{align*}
    and similarly for the $\nu$ coefficients. 
\end{proof}
\subsection{Estimates on stopping time}\label{stoppingestimates}
One important consequence of the $\mu$ and $\nu$ estimates in Section \ref{nonlinearestimates} is that we can now accurately estimate the stopping times $t_{\alpha}$ and $t_{\omega}$. Recall that these are defined according to  
\begin{align}\label{stoppingtimedefinition}
    0=\eta_3(\alpha,t_{\alpha})+\frac{9}{t_{\alpha}}, \qquad 0=\zeta_3(\omega,t_{\omega})+\frac{2}{t_{\omega}}.
\end{align}
We estimate using $\eta(\alpha,t)=\hat{\eta}(t)+\mu(\alpha,t_{\alpha})$ and $\zeta(\omega,t)=\hat{\zeta}(t)+\nu(\omega,t_{\omega})$.
Lemma \ref{hatcknorms} implies that there are smooth functions $\hat{E}_{3,\text{start}}(t),\hat{E}_{3,\text{end}}(t)$, both bounded by $10^{5}$ (in $C^0$) so that 
\begin{subequations}
\label{etahat2d}
\begin{align}
    \hat{\eta}(t)&=\hat{\eta}(\hat{t}_{\hat{\alpha}})+(\hat{\eta})'(\hat{t}_{\hat{\alpha}})(t-\hat{t}_{\hat{\alpha}})+(t-\hat{t}_{\hat{\alpha}})^2 \hat{E}_{3,\text{start}}(t), \\ \hat{\zeta}(t)&=\hat{\zeta}(\hat{t}_{\hat{\omega}})+(\hat{\zeta})'(\hat{t}_{\hat{\omega}})(t-\hat{t}_{\hat{\omega}})+(t-\hat{t}_{\hat{\omega}})^2 \hat{E}_{3,\text{end}}(t).
\end{align}
\end{subequations}
On the other hand, we will use the lemmas of Section \ref{nonlinearestimates} as well as Corollary \ref{mouat} to estimate the $\mu$ and $\nu$ terms. 

Our first step is to estimate the true stopping times at $\hat{\alpha}$ and $\hat{\omega}$ (denoted $t_{\hat{\alpha}}$ and $t_{\hat{\omega}}$) in terms of our approximate stopping times $\hat{t}_{\hat{\alpha}}$ and $\hat{t}_{\hat{\omega}}$. 
\begin{cor}\label{stoppingzero}
If $M\varepsilon\le 10^{-50}$, then the true stopping times $t_{\hat{\alpha}}$ and $t_{\hat{\omega}}$ satisfy $t_{\hat{\alpha}}=\hat{t}_{\hat{\alpha}}+\rho^{(0)}$ and $t_{\hat{\omega}}=\hat{t}_{\hat{\omega}}+\sigma^{(0)}$, where $|\rho^{(0)}|,|\sigma^{(0)}|\le 10^{3}M\varepsilon$.
\end{cor}
\begin{proof}
   
    To find $\rho^{(0)}$, we set $\alpha=\hat{\alpha}$ in \eqref{stoppingtimedefinition}, and use \eqref{munualomex},\eqref{mubigsplit}, \eqref{etahat2d}, so that  
    \begin{align*}
        0&=\hat{\eta_3}(\hat{t}_{\hat{\alpha}}+\rho^{(0)})+\mu^{(0)}(\hat{t}_{\hat{\alpha}}+\rho^{(0)})+\frac{9}{\hat{t}_{\hat{\alpha}}+\rho^{(0)}}\\
        &=\left(\hat{\eta}_3(\hat{t}_{\hat{\alpha}})+\mu^{(0,0)}_3+\frac{9}{\hat{t}_{\hat{\alpha}}}\right)
+\rho^{(0)}\left((\hat{\eta_3})'(\hat{t}_{\hat{\alpha}})+\mu^{(0,1)}_3-\frac{9}{(\hat{t}_{\hat{\alpha}})^2}\right)\\
&+(\rho^{(0)})^2 \left((\hat{E}_{3,\text{start}})_3(\hat{t}_{\hat{\alpha}}+\rho^{(0)})+\mu^{(0,2)}_3(\hat{t}_{\hat{\alpha}}+\rho^{(0)})+\frac{9}{(\hat{t}_{\hat{\alpha}})^2(\hat{t}_{\hat{\alpha}}+\rho^{(0)})}\right). 
\end{align*}
Observe that by Assumption \ref{aposerror} and Corollary \ref{mouat}:
\begin{itemize}
    \item $\left|\hat{\eta}_3(\hat{t}_{\hat{\alpha}})+\mu^{(0,0)}_3+\frac{9}{\hat{t}_{\hat{\alpha}}}\right|\le 10M\varepsilon+\varepsilon$;
    \item $\left((\hat{\eta_3})'(\hat{t}_{\hat{\alpha}})+\mu_3^{(0,1)}-\frac{9}{(\hat{t}_{\hat{\alpha}})^2}\right)=-6.960761-\frac{9}{(1.4923)^2}$ up to an error of $10^{-4}+10^6 M\varepsilon$; and 
    \item $\left|(\hat{E}_{3,\text{start}})_3(\hat{t}_{\hat{\alpha}}+\rho^{(0)})+\mu_3^{(0,2)}(\hat{t}_{\hat{\alpha}}+\rho^{(0)})+\frac{9}{(\hat{t}_{\hat{\alpha}})^2(\hat{t}_{\hat{\alpha}}+\rho^{(0)})}\right|\le 10^5+10^{20}M\varepsilon+100$, provided $\left|\rho^{(0)}\right|\le \frac{1}{10}$.
\end{itemize}
Therefore, if $M\varepsilon\le 10^{-30}$, the Schauder fixed point Theorem gives that $|\rho^{(0)}|\le 10^{3} M\varepsilon$. The calculation for $|\sigma^{(0)}|$ is similar. 
\end{proof}
Now we can estimate what happens for stoppings times $t_{\alpha},t_{\omega}$ for $(\alpha,\omega)$ close to, but not necessarily identical to, $(\hat{\alpha},\hat{\omega})$. To this end, we find it convenient to expand the $\mu^{(0)},\mu^{(1)},\mu^{(2)}$ and $\nu^{(0)},\nu^{(1)},\nu^{(2)}$ time-varying functions of  
\eqref{munualomex} around $t_{\hat{\alpha}}$, rather than $\hat{t}_{\hat{\alpha}}$: 
\begin{align}\label{muexpcst}
\begin{split}
    \mu^{(i)}(t)=\mu^{(i)}(t_{\hat{\alpha}})+(t-t_{\hat{\alpha}})(\mu^{(i)})'(t_{\hat{\alpha}})+(t-t_{\hat{\alpha}})^2 \hat{E}_{4+i,\text{start}}(t),\\
    \nu^{(i)}(t)=\nu^{(i)}(t_{\hat{\omega}})+(t-t_{\hat{\omega}})(\nu^{(i)})'(t_{\hat{\omega}})+(t-t_{\hat{\omega}})^2 \hat{E}_{4+i,\text{end}}(t),
    \end{split}
\end{align}
for $i=0,1,2$. We also write 
\begin{align}\label{etaexpcst}
\begin{split}
    \hat{\eta}(t)=\hat{\eta}(t_{\hat{\alpha}})+(t-t_{\hat{\alpha}})(\hat{\eta})'(t_{\hat{\alpha}})+(t-t_{\hat{\alpha}})^2\hat{E}_{7,\text{start}}(t),\\
    \hat{\zeta}(t)=\hat{\zeta}(t_{\hat{\omega}})+(t-t_{\hat{\omega}})(\hat{\zeta})'(t_{\hat{\omega}})+(t-t_{\hat{\omega}})^2\hat{E}_{7,\text{end}}(t).
    \end{split}
\end{align}
\begin{cor}\label{stopping}
  If $M^2\varepsilon\le 10^{-40}$, then for $\left|\alpha-\hat{\alpha}\right|+\left|\omega-\hat{\omega}\right|\le \frac{10^{-50}}{M}$, the true stopping times $t_{\omega}$ and $t_{\alpha}$ satisfy  
    \begin{align}\label{stopcorrectstart}
    \begin{split}
        t_{\alpha}-t_{\hat{\alpha}}=(\alpha-\hat{\alpha})\rho^{(1)}+(\alpha-\hat{\alpha})^2\rho^{(2)}(\alpha),\\
         t_{\omega}-t_{\hat{\omega}}=(\omega-\hat{\omega})\sigma^{(1)}+(\omega-\hat{\omega})^2\sigma^{(2)}(\omega),
         \end{split}
    \end{align} where 
    \begin{align*}
        \left|\rho^{(1)}+\frac{(\hat{\mu}_{\mathrm{IC}})_3(\hat{t}_{\hat{\alpha}})}{(\hat{\eta}_3)'(\hat{t}_{\hat{\alpha}})-\frac{9}{(t_{\hat{\alpha}})^2}}\right|\le 10^{40}M^2\varepsilon, \qquad \left|\sigma^{(1)}+\frac{(\hat{\nu}_{\mathrm{IC}})_3(\hat{t}_{\hat{\omega}})}{(\hat{\zeta}_3)'(\hat{t}_{\hat{\omega}})-\frac{2}{(\hat{t}_{\hat{\omega}})^2}}\right|\le 10^{40}M^2\varepsilon,
    \end{align*}
    and $\left|\rho^{(2)}(\alpha)\right|+\left|\sigma^{(2)}(\omega)\right|\le 10^{15}M$. 
\end{cor}
\begin{proof}
Using \eqref{muexpcst} and \eqref{etaexpcst}, the $t_{\alpha}$ stopping condition is defined according to 
\begin{align*}
0&=\hat{\eta}_3(t_{\alpha})+\mu_3 (\alpha,t_{\alpha})+\frac{9}{t_{\alpha}}\\
&=\hat{\eta}_3(t_{\hat{\alpha}})+(t_{\alpha}-t_{\hat{\alpha}})(\hat{\eta}_3)'(t_{\hat{\alpha}})+(t_{\alpha}-t_{\hat{\alpha}})^2(\hat{E}_{7,\text{start}})_3(t_{\hat{\alpha}})+\mu^{(0)}_{3}(t_{\alpha})+(\alpha-\hat{\alpha})\mu^{(1)}_3(t_{\alpha})+\mu_3^{(2)}(\alpha,t_{\alpha})(\alpha-\hat{\alpha})^2\\
&+\frac{9}{t_{\hat{\alpha}}}-\frac{9(t_{\alpha}-t_{\hat{\alpha}})}{t^2_{\hat{\alpha}}}+\frac{9(t_{\alpha}-t_{\hat{\alpha}})^2}{t^2_{\hat{\alpha}}t_{\alpha}}\\
&=\left((\hat{\eta}_3)'(t_{\hat{\alpha}})-\frac{9}{t^2_{\hat{\alpha}}}+(\mu^{(0)}_3)'(t_{\hat{\alpha}})\right)(t_{\alpha}-t_{\hat{\alpha}})+\left(\hat{E}_{7,\text{start}}(t_{\alpha})+\hat{E}_{4,\text{start}}(t_{\alpha})+\frac{9}{(t_{\hat{\alpha}})^2t_{\alpha}}\right)(t_{\alpha}-t_{\hat{\alpha}})^2\\
    &+\left(\mu_3^{(1)}(t_{\hat{\alpha}})+(\mu^{(1)}_3)'(t_{\hat{\alpha}})(t_{\alpha}-t_{\hat{\alpha}})+\hat{E}_{5,\text{start}}(t_{\alpha})(t_{\alpha}-t_{\hat{\alpha}})^2\right)(\alpha-\hat{\alpha})+\mu_3^{(2)}(\alpha,t_{\alpha})(\alpha-\hat{\alpha})^2.
\end{align*}
Note that there are no zeroth order terms because the equation is already known to hold when $\hat{\alpha}=\alpha$ and $t_{\alpha}=t_{\hat{\alpha}}$, by definition. Also, we will not need to expand the $\mu^{(2)}_3(\alpha,t_{\alpha})$ into powers of $t_{\alpha}-t_{\hat{\alpha}}$ because it already has $(\alpha-\hat{\alpha}^2)$ out the front.
By feeding the expression for $t_{\alpha}-t_{\hat{\alpha}}$ in \eqref{stopcorrectstart} into this equality and equating linear terms, we obtain 
\begin{align*}
    \rho^{(1)}&=-\frac{\mu_3^{(1)}(t_{\hat{\alpha}})}{(\hat{\eta}_3)'(t_{\hat{\alpha}})-\frac{9}{t^2_{\hat{\alpha}}}+(\mu^{(0)}_3)'(t_{\hat{\alpha}})}\\
    &=-\frac{(\hat{\mu}_{\mathrm{IC}})_3(\hat{t}_{\hat{\alpha}})}{(\hat{\eta}_3)'(\hat{t}_{\hat{\alpha}})-\frac{9}{\hat{t}^2_{\hat{\alpha}}}}+\frac{(\hat{\mu}_{\mathrm{IC}})_3(\hat{t}_{\hat{\alpha}})-\mu_3^{(1)}(t_{\hat{\alpha}})}{((\hat{\eta}_3)'(t_{\hat{\alpha}})-\frac{9}{t^2_{\hat{\alpha}}}+(\mu^{(0)}_3)'(t_{\hat{\alpha}}))}+\frac{(\hat{\mu}_{\mathrm{IC}})_3(\hat{t}_{\hat{\alpha}})((\hat{\eta}_3)'(\hat{t}_{\hat{\alpha}})-\frac{9}{\hat{t}^2_{\hat{\alpha}}}-(\hat{\eta}_3)'(t_{\hat{\alpha}})+\frac{9}{t^2_{\hat{\alpha}}}+(\mu^{(0)}_3)'(t_{\hat{\alpha}}))}{((\hat{\eta}_3)'(\hat{t}_{\hat{\alpha}})-\frac{9}{\hat{t}^2_{\hat{\alpha}}})((\hat{\eta}_3)'(t_{\hat{\alpha}})-\frac{9}{t^2_{\hat{\alpha}}}+(\mu^{(0)}_3)'(t_{\hat{\alpha}}))}.
\end{align*}The $\rho^{(1)}$ estimate then follows from the estimates on $t_{\hat{\alpha}}$ (Assumption \ref{aposerror} and Corollary \ref{stoppingzero}), $\hat{\eta}_3$ (from Assumption \ref{aposerror} and Lemma \ref{hatcknorms}), $\mu^{(0)}$ (Lemma \ref{munu0terms}), $\rho^{(0)}$ (Corollary \ref{stoppingzero}) as well as $\mu^{(1)}$ and  $\mu_{\mathrm{IC}}$  (Lemma \ref{mu1est},  Corollary \ref{linearIC} and Lemma \ref{linearhat2ndd}). With the linear terms taken care of, we then divide by $(\alpha-\hat{\alpha})^2$ to find that the $\rho^{(2)}(\alpha)$ term must satisfy 
\begin{align*}
    0&=\left((\hat{\eta}_3)'(t_{\hat{\alpha}})-\frac{9}{t^2_{\hat{\alpha}}}+(\mu^{(0)}_3)'(t_{\hat{\alpha}})\right)\rho^{(2)}(\alpha)+\mu_3^{(2)}(\alpha,t_{\alpha})\\    &+\left(\hat{E}_{7,\text{start}}(t_{\alpha})+\hat{E}_{4,\text{start}}(t_{\alpha})+\frac{9}{(t_{\hat{\alpha}})^2t_{\alpha}}\right)\left((\rho^{(1)})^2+2\rho^{(1)}\rho^{(2)}(\alpha)(\alpha-\hat{\alpha})+\rho^{(2)}(\alpha)^2(\alpha-\hat{\alpha})^2\right)\\
    &+\left((\mu^{(1)}_3)'(t_{\hat{\alpha}})(\rho^{(1)}+\rho^{(2)}(\alpha)(\alpha-\hat{\alpha}))+\hat{E}_{5,\text{start}}(t_{\alpha})(\alpha-\hat{\alpha})(\rho^{(1)}+(\alpha-\hat{\alpha})\rho^{(2)}(\alpha))^2\right).
\end{align*}
Using the same estimates as before, as well as estimates on $\mu^{(2)}$ (Lemma \ref{munusterm}), $\rho^{(1)}$ (earlier in this proof), $\hat{E}_{4,\text{start}}$, $\hat{E}_{5,\text{start}}$ (follows from Lemma \ref{munu0terms}, the previous $\mu^{(1)}$ estimates and the definition of \eqref{muexpcst}) and $\hat{E}_{7,\text{start}}$ (defined in \eqref{etaexpcst}, can be estimated with Lemma \ref{hatcknorms}).
The Schauder fixed point theorem then implies that for $\left|\alpha-\hat{\alpha}\right|\le \frac{10^{-50}}{M}$, there is a solution $\rho^{(2)}(\alpha)$ whose magnitude is less than $10^{15}M$. 

\end{proof}
To conclude this section, we briefly mention that by combining Corollaries \ref{stoppingzero} and \ref{stopping}, we obtain 
\begin{align}\label{stoppingtimeexpansions}
\begin{split}
    t_{\alpha}-\hat{t}_{\hat{\alpha}}=\rho^{(0)}+(\alpha-\hat{\alpha})\rho^{(1)}+(\alpha-\hat{\alpha})^2\rho^{(2)}(\alpha),\\
    t_{\omega}-\hat{t}_{\hat{\omega}}=\sigma^{(0)}+(\omega-\hat{\omega})\sigma^{(1)}+(\omega-\hat{\omega})^2\sigma^{(2)}(\omega),
    \end{split}
\end{align} 
and all of these coefficients have now been estimated. In particular, the following table is accurate to $10^{-7}$.
\begin{table}[h]
\centering
\begin{tabular}{cccc}
$\hat{t}_{\hat{\alpha}}$ & $\rho^{(1)}$ & $\hat{t}_{\hat{\omega}}$ & $\sigma^{(1)}$ \\ \midrule
$1.4923108$ & $0.0133237$ & $1.1621186$ & $0.0729407$\\\bottomrule
\end{tabular}
\caption{Key stoppings times and their differentials in $(\alpha,\omega)$, accurate to $\pm 10^{-7}$.}  \label{stoppingtimetable}
\end{table}

\subsection{Existence}\label{existence}
We can now put all of the estimates of previous subsections together to produce our non-round Einstein metric, under Assumptions \ref{aposerror} and \ref{linearassumption} for a suitable choice of $\varepsilon$. Recall that it suffices to find a pair $(\alpha,\omega)\neq (1,\sqrt{8})$ with $A(\alpha)=\Omega(\omega)$, and these functions are defined in \eqref{Aalphadef} and \eqref{Omegaomegadef}: 
\begin{align*}
A(\alpha)&=\left(\sqrt{\left(\frac{
    110-2((\alpha+\eta_1(t_{\alpha}))^2-\left(\frac{1}{2}+\frac{1}{9}\right)\eta_2(t_{\alpha})^2}{9}\right)},\eta_2(t_{\alpha})\right),\\
\Omega(\omega)&=\left(\zeta_1(t_{\omega})+\omega,\zeta_2(t_{\omega})\right),
     \end{align*}
     where $\eta$ and $\zeta$ are defined according to \eqref{etaequation1} and \eqref{zetaequation1}.
     The existence proof in this section follows by using the estimates on $\eta,\zeta,t_{\alpha},t_{\omega}$ obtained in  Sections \ref{nonlinearestimates} and \ref{stoppingestimates}. More precisely, we will use these estimates to obtain high-precision estimates on $A(\alpha)$ and $\Omega(\omega)$ for $(\alpha,\omega)$ close to $(\hat{\alpha},\hat{\omega})$. We will thus write $A(\alpha)=\Omega(\omega)$ in fixed-point form, and consequently demonstrate that the Schauder fixed point theorem is applicable in a small region around $(\hat{\alpha},\hat{\omega})$. 
     First, we obtain some estimates on $\zeta$ and $\eta$ at the stopping time.
     \begin{table}[h]
\centering
\begin{tabular}{crrrrrr}
\toprule
\multirow{2}{*}{Index} & \multicolumn{4}{c}{Quantities at stopping time}\\
\cmidrule(lr){2-5} \cmidrule(lr){6-7}
& $\eta^{(0)}$ & $\eta^{(1)}$ & $\zeta^{(0)}$ & $\zeta^{(1)}$ \\
\midrule
1 & $-2.89630$ & $-1.01381$ & $-3.02932$ & $-0.99611$\\
2 & $0.05844$ & $-0.01617$ & $0.05844$ & $-0.08054$  \\ 
3 & $-6.03092$ & $0.05385$ & $-1.72100$ & $0.10802$ \\
\bottomrule
\end{tabular}
\caption{$\eta(t_{\alpha})$ and $\zeta(t_{\omega})$ expansions in $\alpha$ and $\omega$, accurate to $\pm 6 \times 10^{-6}$.}  \label{existencestop}
\end{table}
 \begin{lem}\label{fast}
If $M^2\varepsilon\le 10^{-50}$, then at the two stopping times, we have 
\begin{align*}
    \eta(t_{\alpha})=\eta^{(0)}+(\alpha-\hat{\alpha})\eta^{(1)}+(\alpha-\hat{\alpha})^2\eta^{(2)}(\alpha),\\
\zeta(t_{\omega})=\zeta^{(0)}+(\omega-\hat{\omega})\zeta^{(1)}+(\omega-\hat{\omega})^2\zeta^{(2)}(\omega),
\end{align*}
where $\left|\eta^{(2)}(\alpha)\right|+\left|\zeta^{(2)}(\omega)\right|\le 10^{70}M$ for all $(\alpha,\omega)$ so that $\left|\alpha-\hat{\alpha}\right|+\left|\omega-\hat{\omega}\right|\le \frac{10^{-50}}{M}$,
\begin{align*}
    \left|\eta^{(0)}-\hat{\eta}(\hat{t}_{\hat{\alpha}})\right|\le 10^{10}M\varepsilon , \qquad \left|\zeta^{(0)}-\hat{\zeta}(\hat{t}_{\hat{\omega}})\right|\le 10^{10}M\varepsilon
\end{align*}and, accurate to $6\times 10^{-6}$, Table \ref{existencestop} gives $\eta^{(0)}$, $\eta^{(1)}$, $\zeta^{(0)}$ and $\zeta^{(1)}$. 
 \end{lem}
 \begin{proof}
We shall focus on $\eta(t_\alpha)$, as the case for $\zeta(t_\omega)$ is essentially identical. 
 Recall the expression for $\mu$ from Corollary \ref{mouat}: 
\begin{align*}
        \mu(\alpha,t)&=\mu^{(0,0)}+\mu^{(1,0)}(\alpha-\hat{\alpha})+\mu^{(0,1)}(t-\hat{t}_{\hat{\alpha}})+\mu^{(1,1)}(\alpha-\hat{\alpha})(t-\hat{t}_{\hat{\alpha}})\\
        &+(t-\hat{t}_{\hat{\alpha}})^2\left(\mu^{(0,2)}(t)+(\alpha-\hat{\alpha})\mu^{(1,2)}(t)\right)+(\alpha-\hat{\alpha})^2\mu^{(2,0)}(\alpha,t).
    \end{align*}
    Combining this with \eqref{stoppingtimeexpansions} gives 
\begin{align*}
    \mu(\alpha,t_{\alpha})&=\mu^{(0,0)}+(\alpha-\hat{\alpha})\mu^{(1,0)}+(\alpha-\hat{\alpha})^2\mu^{(2,0)}(\alpha,t_{\alpha})\\
    &+\left(\rho^{(0)}+(\alpha-\hat{\alpha})\rho^{(1)}+(\alpha-\hat{\alpha})^2\rho^{(2)}(\alpha)\right)\left(\mu^{(0,1)}+(\alpha-\hat{\alpha})\mu^{(1,1)}\right)\\
        &+\left(\rho^{(0)}+(\alpha-\hat{\alpha})\rho^{(1)}+(\alpha-\hat{\alpha})^2\rho^{(2)}(\alpha)\right)^2\left(\mu^{(0,2)}(t_{\alpha})+(\alpha-\hat{\alpha})\mu^{(1,2)}(t_{\alpha})\right)\\&=\left(\mu^{(0,0)}+\rho^{(0)}\mu^{(0,1)}+(\rho^{(0)})^2\mu^{(0,2)}(t_{\alpha})\right)\\
        &+\left(\mu^{(1,0)}+\mu^{(0,1)}\rho^{(1)}+\mu^{(1,1)}\rho^{(0)}+(\rho^{(0)})^2\mu^{(1,2)}(t_{\alpha})+2\mu^{(0,2)}(t_{\alpha})\rho^{(0)}\rho^{(1)}\right)(\alpha-\hat{\alpha})\\
        &+\left(\mu^{(2,0)}(\alpha,t_{\alpha})+\rho^{(2)}(\alpha)\mu^{(0,1)}+\rho^{(1)}\mu^{(1,1)}+\mu^{(0,2)}(t_{\alpha})(\rho^{(0)}\rho^{(2)}(\alpha)+2\rho^{(1)})+2\mu^{(1,2)}(t_{\alpha})\rho^{(0)}\rho^{(1)}\right)(\alpha-\hat{\alpha})^2\\
        &+\left(\rho^{(2)}(\alpha)\mu^{(1,1)}+2\rho^{(1)}\mu^{(0,2)}(t_{\alpha})\rho^{(2)}(\alpha)+\mu^{(1,2)}(t_{\alpha})(\rho^{(0)}\rho^{(2)}(\alpha)+(\rho^{(1)})^2)\right)(\alpha-\hat{\alpha})^3\\
        &+\left((\rho^{(2)}(\alpha))^2+2\rho^{(1)}\rho^{(2)}(\alpha)\mu^{(1,2)}(t_{\alpha})\right)(\alpha-\hat{\alpha})^4+\mu^{(1,2)}(t_{\alpha})\rho^{(2)}(\alpha)^2(\alpha-\hat{\alpha})^5.
\end{align*}
Recall from Lemma \ref{hatcknorms} and   \eqref{etahat2d} that there are smooth functions $\hat{E}_{3,\text{start}}(t),\hat{E}_{3,\text{end}}(t)$, both bounded by $10^{5}$ so that 
\begin{align*}
    \hat{\eta}(t)=\hat{\eta}(\hat{t}_{\hat{\alpha}})+(\hat{\eta})'(\hat{t}_{\hat{\alpha}})(t-\hat{t}_{\hat{\alpha}})+(t-\hat{t}_{\hat{\alpha}})^2 \hat{E}_{3,\text{start}}(t), \qquad \hat{\eta}(t)=\hat{\zeta}(\hat{t}_{\hat{\omega}})+(\hat{\zeta})'(\hat{t}_{\hat{\omega}})(t-\hat{t}_{\hat{\omega}})+(t-\hat{t}_{\hat{\omega}})^2 \hat{E}_{3,\text{end}}(t)
\end{align*}
and similarly for $\hat{\zeta}(t)$. By setting $t=t_{\alpha}$ and again using \eqref{stoppingtimeexpansions}, we obtain 
\begin{align*}
    \hat{\eta}(t_{\alpha})&=\hat{\eta}(\hat{t}_{\hat{\alpha}})+\hat{\eta}'(\hat{t}_{\hat{\alpha}})\left(\rho^{(0)}+\rho^{(1)}(\alpha-\hat{\alpha})+\rho^{(2)}(\alpha)(\alpha-\hat{\alpha})^2\right)+\left(\rho^{(0)}+\rho^{(1)}(\alpha-\hat{\alpha})+\rho^{(2)}(\alpha)(\alpha-\hat{\alpha})^2\right)^2 \hat{E}_{3,\text{start}}(t_{\alpha})\\
    &=\left(\hat{\eta}(\hat{t}_{\hat{\alpha}})+\hat{\eta}'(\hat{t}_{\hat{\alpha}})\rho^{(0)}+(\rho^{(0)})^2\hat{E}_{3,\text{start}}(t_{\alpha})\right)\\
    &+\left(\hat{\eta}'(\hat{t}_{\hat{\alpha}})\rho^{(1)}+2\rho^{(0)}\rho^{(1)}\hat{E}_{3,\text{start}}(t)\right)(\alpha-\hat{\alpha})\\
    &+\left(\hat{\eta}'(\hat{t}_{\hat{\alpha}})\rho^{(2)}(\alpha)+\hat{E}_{3,\text{start}}(t_{\alpha})(\rho^{(0)}\rho^{(2)}(\alpha)+(\rho^{(1)})^2)\right)(\alpha-\hat{\alpha})^2\\
    &+\rho^{(1)}\rho^{(2)}(\alpha)\hat{E}_{3,\text{start}}(t_{\alpha})(\alpha-\hat{\alpha})^3+(\rho^{(2)}(\alpha))^2 \hat{E}_{3,\text{start}}(t_{\alpha})(\alpha-\hat{\alpha})^4.
\end{align*}

Altogether, since $\eta = \hat{\eta} + \mu$, we conclude 
\[
\eta^{(0)} = \hat{\eta}(\hat{t}_{\hat{\alpha}}) + \hat{\eta}'(\hat{t}_{\hat{\alpha}}) \rho^{(0)} + \rho^{(0)} \hat{E}_{3,\text{start}}(t_\alpha) + \mu^{(0,0)} + \rho^{(0)}\mu^{(0,1)} + \rho^{(0)} \mu^{(0,2)}(t_\alpha),
\]
\[
\eta^{(1)} = \left(\mu^{(1,0)}+\mu^{(0,1)}\rho^{(1)}+\mu^{(1,1)}\rho^{(0)}+(\rho^{(0)})^2\mu^{(1,2)}(t_{\alpha})+2\mu^{(0,2)}(t_{\alpha})\rho^{(0)}\rho^{(1)}\right)+\left(\hat{\eta}'(\hat{t}_{\hat{\alpha}})\rho^{(1)}+2\rho^{(0)}\rho^{(1)}\hat{E}_{3,\text{start}}(t)\right).
\]
For $\alpha$ in the specified range, we can estimate $\eta^{(0)}$ and $\eta^{(1)}$ using previous estimates on $\rho^{(0)}$ (Corollary \ref{stoppingzero}), $\rho^{(1)}$, (Corollary \ref{stopping}) and $\mu^{(i,j)}$ for various $(i,j)$  (Corollary \ref{mouat}). We can then use these same estimates, as well as those of $\rho^{(2)}(\alpha)$ (Corollary \ref{stopping}) to estimate $\eta^{(2)}$. 

 \end{proof}
 \begin{table}[h]
\centering
\begin{tabular}{cllllll}
\toprule
\multirow{2}{*}{Index} & \multicolumn{4}{c}{Shooting Function Terms}\\
\cmidrule(lr){2-5} \cmidrule(lr){6-7}
& $A^{(0)}$ & $A^{(1)}$ & $\Omega^{(0)}$ & $\Omega^{(1)}$ \\
\midrule
1 & $3.1566$ & $0.0031$ & $3.1566$ & $0.0039$\\
2 & $0.0584$ & $-0.0162$ & $0.0584$ & $-0.0805$  \\ 
\bottomrule
\end{tabular}
\caption{Coefficients for the first-order Taylor expansion of $(A,\Omega)$ near $(\hat{\alpha},\hat{\omega})$, accurate to $\pm 5 \times 10^{-5}$}  \label{shootingfunctions}
\end{table}
 \begin{lem}\label{penulestimate}
If $M^2\varepsilon\le 10^{-50}$ The second-order Taylor expansions of $A$ and $\Omega$ satisfy
\begin{align}\label{penequation}
\begin{split}
    A(\alpha)&=A^{(0)}+(\alpha-\hat{\alpha})A^{(1)}+(\alpha-\hat{\alpha})^2 A^{(2)}(\alpha),\\ 
    \Omega(\omega)&=\Omega^{(0)}+(\omega-\hat{\omega})\Omega^{(1)}+(\omega-\hat{\omega})^2 \Omega^{(2)}(\omega), 
    \end{split}
\end{align}
where $\left|A^{(2)}(\alpha)\right|+\left|\Omega^{(2)}(\omega)\right|\le 10^{80}M$ for all $(\alpha,\omega)$ with $\left|\alpha-\hat{\alpha}\right|+\left|\omega-\hat{\omega}\right|\le \frac{10^{-90}}{M}$, and the other terms are given by Table \ref{shootingfunctions}, accurate to $5\times 10^{-5}$. Furthermore, $\left|A^{(0)}-\Omega^{(0)}\right|\le 10^{15}M\varepsilon$.  
 \end{lem}
 \begin{proof}
     Recall that 
     \begin{align*}
         A(\alpha)&=\left(\sqrt{\left(\frac{
    110-2(\alpha+\eta_1(t_{\alpha}))^2-\left(\frac{1}{2}+\frac{1}{9}\right)\eta_2(t_{\alpha})^2}{9}\right)},\eta_2(t_{\alpha})\right)
     \end{align*}
     and 
     \begin{align*}
         \Omega(\omega)=(\zeta_1(t_{\omega})+\omega,\zeta_2(t_{\omega})). 
     \end{align*}
     The result follows immediately by expanding these terms, using Lemma \ref{fast}. For the first component of $A$, we also use the fact that 
     \begin{align}\label{rootexpand}
        \sqrt{x_0+x}-\sqrt{x_0}&=\frac{x}{2\sqrt{x_0}}+O(x_0,x),
     \end{align}
     where $\left|O(x_0,x)\right|\le 10x^2$ for all $x$ and $x_0$ with  $\left|x\right|\le \frac{1}{10}$ and $x_0\ge 1$. Consequently, 
     \[
     \left| A_1^{(1)} + \frac{4(\hat{\alpha}+\eta_1^{(0)})(1+\eta_1^{(1)})+(1+\frac29)\eta_2^{(1)}\eta_2^{(0)}}{18 A_1^{(0)}} \right| \leq 10^{-5}.
     \]
    Using this same square root estimate with Assumption \ref{aposerror} gives the required estimate on  $\left|A^{(0)}-\Omega^{(0)}\right|$.  
 \end{proof}

Finally can construct our Einstein metric. 
\begin{prop}
\label{prop:finalstep}
Theorem \ref{CT} holds under Assumptions \ref{aposerror} and \ref{linearassumption} with $\varepsilon < 10^{-350}$.
\end{prop}
\begin{proof}
Using the notation from \eqref{penequation} as well as the $2\times 2$ square matrix $J = \begin{pmatrix}
        A^{(1)}&-\Omega^{(1)}
    \end{pmatrix}$, 
    we can write  
     \begin{align*}
    A(\alpha)-\Omega(\omega)=A^{(0)}-\Omega^{(0)}+J \begin{pmatrix}
        \alpha-\hat{\alpha}\\
        \omega-\hat{\omega}
    \end{pmatrix}+(\alpha-\hat{\alpha})^2 A^{(2)}(\alpha)-(\omega-\hat{\omega})^2 \Omega^{(2)}(\omega). 
\end{align*}
If $\varepsilon<10^{-350}$, then Lemma \ref{penulestimate} is applicable, so the square matrix $J$ is invertible, with $\left|J^{-1}\right|\le 10^5$, so that the Einstein equation $0=A(\alpha)-\Omega(\omega)$ can be written in the fixed-point form 
\begin{align*}
    \begin{pmatrix}
        \alpha-\hat{\alpha}\\
        \omega-\hat{\omega}
    \end{pmatrix}&=J^{-1}(\Omega^{(0)}-A^{(0)}+(\omega-\hat{\omega})^2 \Omega^{(2)}(\omega)-(\alpha-\hat{\alpha})^2 A^{(2)}(\alpha)),
\end{align*}
and observe that 
\begin{align*}
    \left| J^{-1}(\Omega^{(0)}-A^{(0)})\right|\le 10^{20}M\varepsilon. 
\end{align*}
On the other hand, if $\left|\alpha-\hat{\alpha}\right|+\left|\omega-\hat{\omega}\right|\le \frac{10^{-90}}{M}$, then 
\begin{align*}
    \left|J^{-1}\left((\omega-\hat{\omega})^2 \Omega^{(2)}(\omega)-(\alpha-\hat{\alpha})^2 A^{(2)}(\alpha)\right)\right|\le 10^{-4}\cdot \frac{10^{-90}}{M}.
\end{align*} Existence of a fixed point satisfying $\left|\alpha-\hat{\alpha}\right|+\left|\omega-\hat{\omega}\right|\le \frac{10^{-90}}{M}$ then follows from the Schauder fixed point Theorem (Theorem \ref{Schauder}), combined with the observation that $10^{20}M\varepsilon\le \frac{10^{-100}}{M}$.  

 \end{proof}

\section{High--Precision Construction}
\label{numerics}

It is now established that a non-round Einstein metric exists provided Assumptions \ref{aposerror} and \ref{linearassumption} hold with $\varepsilon < 10^{-350}$. Although there are many methods for solving differential equations, in order to achieve this level of accuracy, we are restricted to only a handful of options. In this section, we outline what is, to our knowledge, the most effective approach to reliably solve this problem. Since it is clear that a root-finding procedure is necessary to solve the shooting problem, it is tempting to use such a procedure to estimate the solution to the entire system of differential equations as well. Many such procedures are known to converge at a doubly-exponential rate, that is, the number of correct decimal places doubles with each iteration. Unfortunately, doing so as separate problems is slow; while doing so altogether tends to converge to the round metric. By separating each task according to Steps (I)--(IV) outlined in the Introduction, a high-precision approximation can be reliably obtained. 
All implementations and quantities are available at \url{https://github.com/liamhodg/O3O10}.

\begin{proof}[Proof of Theorem \ref{CT}]
Performing Steps \ref{step:Heur}--\ref{step:Interp} outlined in the following Sections \ref{sec:apsolver} and \ref{sec:chebsolver} with a target $d = 550$, we verify that the interpolants $\hat{\eta}$ and $\hat{\zeta}$ satisfy Assumptions \ref{aposerror} and \ref{linearassumption} with $\varepsilon = 10^{-362}$. In particular, for Assumption \ref{aposerror},
$$
\begin{array}{ccc}
\|\hat{E}_{1,\text{start}}\|_{H^3} \leq 10^{-534}, & \|\hat{E}_{1,\text{end}}\|_{H^3} \leq 10^{-446}, & |\hat{A}(\hat{\alpha}) - \hat{\Omega}(\hat{\omega})| \leq 10^{-551}, \\
& & \\
|\hat{\eta}_3(\hat{t}_{\hat{\alpha}}) + \frac{9}{\hat{t}_{\hat{\alpha}}}| \leq 10^{-549}, & \text{and} & |\hat{\zeta}_3(\hat{t}_{\hat{\omega}})+\frac{2}{\hat{t}_{\hat{\omega}}}| \leq 10^{-468}.
\end{array}
$$
Next, we verify Assumption \ref{linearassumption} following from Step \ref{step:Lin} outlined in Section \ref{sec:chebsolver}. We find that
\[
\|\hat{E}_{2,\text{start}}\|_{H^3} \leq 10^{-364},\qquad \|\hat{E}_{2,\text{end}}\|_{H^3} \leq 10^{-362}.
\]
The result follows from Proposition \ref{prop:finalstep}.
\end{proof}

\subsection{Step (I): Arbitrary-Precision Solver}
\label{sec:apsolver}
To solve (\ref{etaequationO}) to arbitrarily high precision, a custom Taylor series-based solver is employed. Unlike the following steps, this initial solver is \emph{heuristic}, that is, there are no guarantees of accuracy or regularity. First, we solve (\ref{etaequationO}) locally about zero using Frobenius' method. The ansatz $\eta(t) = \sum_{k=0}^{\infty} a_k t^k$ reveals
\[
\sum_{k=0}^{\infty} ((k+1)I - L_{d_1 d_2})a_{k+1} t^k = \sum_{k=0}^{\infty} t^k \sum_{i+j = k} B_{d_1 d_2} (a_i, a_j).
\]
Equating coefficients, we find that
\[
a_{k+1} = ((k+1)I - L_{d_1 d_2})^{-1} \sum_{i=0}^k B_{d_1 d_2}(a_i, a_{k-i}),
\]
forming a solvable recurrence relation starting from the initial condition $a_0 = (0,0,0,\alpha,\sqrt{\lambda})$. Convergence of this series is easily established using the method of majorants, and holds only on a small neighbourhood about zero. To solve (\ref{etaequationO}) outside of this neighbourhood, note that by differentiating (\ref{etaequationO}) $m$ times,
\[
\eta^{(m+1)}=\sum_{k=0}^{m}\binom{m}{k}\left\{ \frac{(-1)^{m-k}(m-k)! L_{d_1 d_2}\eta^{(k)}}{t^{m-k+1}}+B_{d_1 d_2}[\eta^{(m-k)},\eta^{(k)}]\right\},\quad m \geq 0. \]
Hence, starting from $\eta(t)$, higher-order derivatives of $\eta^{(m)}(t)$ can be computed recursively. Away from zero, local Taylor series expansions of $\eta$ can be computed about any point $t$, provided $\eta(t)$ is first known to high precision.

Any solution to (\ref{etaequationO}) is naturally holomorphic, and as the power series expansion for $\eta$ at zero has finite radius of convergence, the same will be true for any local Taylor expansion of $\eta$. Let $c_k(t) : [0,T] \to \mathbb{R}^5$ be the $k$-th coefficient of a power series expansion of $\eta$ about $t$, that is, $c_k(t) = \frac{1}{k!} \eta^{(k)}(t)$. The radius of convergence of the power series expansion about $t$ can be computed using the root test: $r(t)^{-1} = \limsup_{k\to \infty} \|c_k(t)\|_{\infty}^{1/k} > 0$. We estimate this quantity by taking the maximum of $\|c_k(t)\|_\infty^{1/k}$ of the last five computed coefficients. Any local expansion of $\eta$ about $t$ is used only in the interval $[t,t+r(t)/2]$ to balance both rapid convergence and reasonable coverage of the series expansion. To estimate $\eta$ beyond this interval, another expansion is computed about $t + r(t)/2$. 

The error in truncating the expansion is estimated by a corresponding geometric series. As we intend only for each expansion to be used up to half its radius of convergence, the maximal error for each expansion by truncating at $N$ terms is
\[
\epsilon_N(t) = \sup_{s \in [t,t+r(t)/2]} \left\lVert\eta(t) - \sum_{k=0}^N c_k(t) (s - t)^k\right\rVert_{\infty} \approx \|\eta(t)\|_{\infty} \sum_{k=N+1}^{\infty} 2^{-k} = 2^{-N} \|\eta(t)\|_{\infty}. 
\]
Assuming $\|\eta(t)\|_{\infty} \leq 100$ (which appears to be safe in practice), to achieve $d$ decimal places of accuracy, $N \geq (d+2)\frac{\log10}{\log2}$ terms suffices. To perform these operations naively, roughly $3d$ decimal places of precision is required. To reduce the required precision, we instead work with coefficients normalized relative to some small parameter $\rho > 0$:
\[
c_k^\rho(t) = \frac{\rho^k}{k!} \eta^{(k)}(t),\qquad c^\rho_{m+1}=\frac{\rho}{m+1}\sum_{k=0}^{m}\left\{ (-\rho)^{m-k}\frac{L_{d_{1}d_{2}}c^\rho_{k}}{t^{m-k+1}}+B_{d_{1}d_{2}}\left[c^\rho_{m-k},c^\rho_{k}\right]\right\},\quad m \geq 0.
\]
In practice, we have found $\rho = 0.3$ to work well. The local Taylor expansions then become
\[
\eta(t+\delta) \approx \sum_{k=0}^N c_k^\rho(t) \left(\frac{\delta}{\rho}\right)^k,\qquad 0 \leq \delta < \frac{r(t)}{2}.
\]
With approximate solutions for $\eta$ (and consequently $\zeta$ by switching $d_1$ and $d_2$), next we estimate the stopping times $t_\alpha$, $t_\omega$. This is a one-dimensional root-finding problem, solved by conducting a grid search for a value where $Z(t) < 0$ (since $Z(t)$ is strictly monotone decreasing) using Brent's method \cite{Brent71}. The final shooting problem constitutes a two-dimensional root-finding problem. Computing linearisations of $A$, $\Omega$ is ultimately too expensive, so Broyden's method \cite{Broyden65} is employed.

\subsection{Steps (II)--(III): Chebyshev Methods}
\label{sec:chebsolver}

To render an heuristic function $f:[0,T]\to \mathbb{R}$ explicitly and rigorously integrable, it is first approximated by a polynomial interpolant. Expressing the system of differential equations (\ref{etaequationO}) in terms of $\upsilon = \eta'$, all operations involved, including $+$, $-$, $\times$, $\upsilon\mapsto \int_0^x \upsilon$, $\upsilon \mapsto \frac1x \int_0^x \upsilon$, and differentiation, can be conducted within the polynomial ring differential algebra. Approximating $\upsilon$ using a polynomial, any $H^k$ norm of the \emph{a posteriori error} can be computed reliably within interval arithmetic. Unfortunately, evaluating polynomials in the standard monomial basis $\mathbb{R}[1,x,x^2,\dots]$ is susceptible to extreme precision loss when the degree is high. Such issues are rectified by performing interpolation under a different polynomial basis; arguably the best in this regard are the Chebyshev polynomials.

The Chebyshev polynomials $T_n$ are the unique polynomials \cite[\S1.2.1]{Mason} which satisfy $T_n(x) = \cos(n \arccos x)$ for $|x|\leq 1$, and form an orthonormal basis on $L^2([-1,1])$. They satisfy $T_0(x) = 1$, $T_1(x) = x$, and the recursion relation $T_{n+1}(x) = 2 xT_n(x) - T_{n-1}(x)$. To fit $f$ in terms of Chebyshev polynomials, the domain $[0,T]$ is mapped to $[-1,1]$ through the transformation $x \mapsto \frac{2x}{T} - 1$. Letting $N$ be a positive integer, the polynomial $f_n$ that matches $f$ on each of the Chebyshev nodes $x_k = \cos(\pi N^{-1}(k-\frac12))$ (the zeros of $T_N$) is given by \cite[Theorem 6.7]{Mason}
\begin{equation}
\label{chebyapprox}
f_N = \frac12 c_0 T_0 + \sum_{k=1}^{N-1} c_k T_k, \qquad
c_k = \frac{2}{N} \sum_{j=1}^N f(x_j) \cos\left(\frac{\pi k(j-\frac12)}{N}\right).
\end{equation}
Observe that function evaluations are shared across coefficients. Unlike the monomial basis, the coefficients $c_k$ decrease for large degree by the Riemann-Lebesgue lemma. Indeed, one can show that the Chebyshev basis is near-optimal for polynomial approximation \cite[\S3.3]{Mason}. The Chebyshev polynomials can also be efficiently computed using complex arithmetic through the expression \cite[eqn. (1.49)]{Mason}
\[
T_n(x) = \frac12(x_{+}^n + x_{-}^n),\qquad x_\pm = x \pm \sqrt{x^2 - 1}.
\]
Altogether, Chebyshev interpolants can be computed with minimal loss of precision. They also exhibit highly accurate approximations, as we have found the choice $N = 3d + 1$ to reliably achieve $d$ decimal places of precision across $[0,T]$ in our cases. However, computation of the \emph{a posteriori error} becomes nontrivial; calculations cannot be performed in the monomial basis, as this will once again incur precision loss. Instead, products are computed using \cite[eqn. (2.38)]{Mason}, 
integrals using \cite[eqn. (2.43)]{Mason}, 
and derivatives using \cite[eqn. (2.49)]{Mason}. 
For the Ces\`{a}ro mean $C_n(x) = \frac1x \int_0^x T_n(a s + b) \dd s$, for each $n$, let $R_n(x) = x^{-1}(T_n(a x + b) - T_n(b))$, 
which satisfies $R_0(x) = 0$, $R_1(x) = a$, and for $n \geq 2$,
\[
R_n(x) = 2 a T_{n-1}(a x + b) + 2 b R_{n-1}(x) - R_{n-2}(x).
\]
Then $C_n = T_0$ if $n = 0$, $C_n = \frac12 T_1 + \frac{b}{2} T_0$ if $n = 1$, and
\[
C_{n}=
\frac{1}{a}\left[\frac{R_{n+1}}{2(n+1)}-\frac{R_{n-1}}{2(n-1)}\right], \quad \text{ if }n\geq2.
\]
Using these properties of the Chebyshev polynomials together with equation (\ref{chebyapprox}) allows for the computation of the \emph{a posteriori error} in Step \ref{step:Interp}, as one can fit $\hat{\upsilon}$ to the derivative of the solution in Step \ref{step:Heur} using (\ref{chebyapprox}) and explicitly compute
\[
\hat{E}_{1,\text{start}} = L_{29} \left(\frac1t \int_0^t \hat{\upsilon}(s) \dd s\right) + B_{29}\left(\int_0^t \hat{\upsilon}(s), \int_0^t \hat{\upsilon}(s)\right) - \hat{\upsilon}(t),
\]
as a polynomial in the Chebyshev basis. The $H^k$ norm of $f \in \{\hat{E}_{1,\text{start}},\hat{E}_{1,\text{end}},\hat{\eta},\hat{\zeta}\}$ can be explicitly computed by evaluating $\int_0^t f(s)^2 \dd s$ in the Chebyshev basis, and evaluating it at $t = \hat{t}_{\hat{\alpha}} + 10^{-3}$ or $\hat{t}_{\hat{\omega}} + 10^{-3}$. To ensure rigour, all computations are performed using interval arithmetic with the \texttt{arb} software package \cite{Arb}.

Only Step \ref{step:Lin}---solving the system of linear differential equations (\ref{lineariseIC})---remains. Ordinarily, this can be readily achieved using a Chebyshev series method as in \cite[\S10.2.1]{Mason}. However, as (\ref{lineariseIC}) involves the approximations $\hat{\eta},\hat{\zeta}$ as well, we have found that reliably achieving this task requires an uncomfortably high degree of precision. Surprisingly, finite difference approximations appear to work better: taking
\[
\hat{\mu}(t,\alpha) = \frac{\hat{\eta}(t,\hat{\alpha}+\delta) - \hat{\eta}(t,\hat{\alpha}-\delta)}{2 \delta},\qquad
\hat{\nu}(t,\omega) = \frac{\hat{\zeta}(t,\hat{\omega}+\delta) - \hat{\zeta}(t,\hat{\omega}-\delta)}{2 \delta},
\]
the approximations $\hat{\mu},\hat{\nu}$ are accurate to order $\mathcal{O}(\delta^2)$. To remain in the Chebyshev basis for these calculations, it is important that Chebyshev expansions of $\hat{\eta}(\cdot,\hat{\alpha}+\delta)$ and $\hat{\eta}(\cdot,\hat{\alpha}-\delta)$ are conducted on the same interval $[0,\hat{t}_{\hat{\alpha}}+10^{-3}]$.

For targetting $d$ decimal places of precision, letting $\delta = 10^{-[d/3]}$ appears near optimal, losing roughly $d/3$ decimal places of accuracy over $\hat{\mu},\hat{\nu}$ themselves. Hence, to target $d_{\text{target}}$ decimal places of precision, one should look to solve the original problem to at least
$d \geq \frac32 d_{\text{target}}$ decimal places of precision. As Proposition \ref{prop:finalstep} requires $d_{\text{target}} = 350$, we require $d \geq 525$. As seen, $d = 550$ suffices.


\section{Concluding Remarks}
In this work, we have developed a computationally-assisted proof of the existence of a novel Einstein metric on $S^{12}$. On the numerical side, this was achieved by first obtaining a high-precision \emph{heuristic} solution, interpolating this solution with Chebyshev polynomials to rigorously compute \emph{a posteriori} error, and using finite difference schemes to estimate the differential in $\alpha$. To our knowledge, these techniques are effectively optimal---they are fast, and reveal rigorously accurate error bounds. We assert that this procedure (Steps \ref{step:Heur}--\ref{step:Theory}) should also be applicable for developing computationally-assisted proofs of existence of solutions to differential equations more generally. As the equations involved in this paper are ordinary differential equations, this affords the use of Taylor series solvers and polynomial methods, which are remarkably efficient for high-precision computation. Such methods are unlikely to be as effective for partial differential equations, even in two spatial dimensions. Kernel interpolation methods \cite{RasmussenWilliams} scale well into higher-dimensions, and can achieve high precision along carefully chosen quadrature points. For the heuristic solution to the nonlinear problem, one can look toward the adoption of high-precision finite element method solvers including hp-FEM \cite{Melenk}, although these can struggle with singular boundary conditions. Another noteworthy alternative that performs better with singular boundary conditions are neural network-based solutions, as seen in \cite{Wangetal23}. At present, such techniques are limited to low-precision, and would require further development to achieve higher-precision results. 

On the theoretical side, the Schauder Fixed Point Theorem was repeatedly used to guarantee existence, although this does not assert uniqueness in the final solution. Local uniqueness may be achieved via the Banach Fixed Point Theorem, but this would require some care. Due to the allowances from the high-precision solvers, substantial liberties were taken with the development of the linearisation bounds, particularly in Proposition \ref{GreenEstimate}, which evoked a crude Gronwall inequality for $M$ at great expense to the final precision required. This can be substantially improved by estimating the complete Green's function, that is, the fundamental matrix for the equations of \eqref{lineariseIC}. This quantity is unbounded in any neighbourhood of $t = 0$, and therefore requires specialised treatment of the singularity.

\clearpage

\appendix

\begin{center}
\LARGE \bf Appendix
\end{center}


\section{Analytic Preliminaries}\label{analprelim}
In Section \ref{geomprelim} of this paper, we saw that $\mathsf{O}(3)\times \mathsf{O}(10)$-invariant Einstein metrics on $S^{12}$ can be constructed by analysing solutions of a certain singular value problem for a system of ODEs. The purpose of this Appendix is to discuss the basic analytical tools that we will use in this study.  

 \subsection{Norms}\label{norms}
 Let $I\subseteq \mathbb{R}$ be an interval. In this paper, we will study several vectors and vector-valued functions $v:I\to \mathbb{R}^n$, as well as matrices and matrix-valued functions $M:I\to M_{n\times n}(\mathbb{R})$; the purpose of this subsection is to discuss the norms that we will use to study these objects.   
The finite-dimensional vector space $\mathbb{R}^n$ will always be equipped with the Euclidean norm, i.e., if $x=(x_1,\cdots,x_n)\in \mathbb{R}^n$, then we write 
 \begin{align*}
     \left|x\right|=\left(\sum_{i=1}^{n}\left|x_i\right|^2\right)^{\frac{1}{2}}.
 \end{align*}
 On the other hand, $ M_{n\times n}(\mathbb{R})$ is equipped with the operator norm, i.e., if $A\in M_{n\times n}(\mathbb{R})$, then 
 \begin{align*}
     \left|A\right|=\sup_{x\in \mathbb{R}^n: \left|x\right|=1}\left|Ax\right|. 
 \end{align*}
 We then define the norm of  functions with domain $I$ in the usual way with this understanding. For example, 
\begin{align*}
\|f\|_{L^2(I)}=\left(\int_I \left|f(t)\right|^2dt\right)^{\frac{1}{2}},
\end{align*}
where $f$ can be a vector function or a matrix function.
We find $L^2(I)$ norms to be the most useful for the purposes of this paper, because it is quite straightforward to compute the $L^2$ norm of a given polynomial function $p(t)$, as opposed to, say, its $\sup$ norm. 

We can extend the $L^2(I)$ norm to the $H^k(I)$ norms defined inductively as:  
\begin{align*}
    \| f\|^2_{H^k(I)}=\|f\|^2_{H^{k-1}(I)}+\|h\|_{L^2(I)}^2
\end{align*}
for any vector or matrix function $f$, where $h$ is the $k$th derivative of $f$. 
Finally, we also consider the $C^0$ norm as
\begin{align*}
    \|f\|_{C^0(I)}=\sup_{t\in I}\left|f(t)\right|,
\end{align*}
as well as the $C^k$ norms defined inductively by 
\begin{align*}
    \| f\|_{C^k(I)}=\|f\|_{C^{k-1}(I)}+\|h\|_{C^0(I)}. 
\end{align*}
The following comparisons between norms of scalar functions $f:I\to \mathbb{R}$ will be helpful. 

\begin{lem}[\textsc{Embedding}]\label{Sobolevembedding}
     Choose an interval $I=[0,t^*]$ for some $t^*<1.5$.
     \begin{enumerate}[label=(\alph*)]
         \item For any smooth scalar function $f:I\to \mathbb{R}$ and any $t_0\in I$, we have \begin{align*}
\left|\left|f-f_0\right|\right|_{C^0(I)}\le\frac{5}{4} \left|\left|f'\right|\right|_{L^2(I)},
\end{align*}
where $f_0:I\to \mathbb{R}$ is the constant function $f_0(t)=f(t_0)$. 
\item   For any smooth scalar function $f:I\to \mathbb{R}$, we have \begin{align*}
\left|\left|f\right|\right|_{C^0(I)}\le 2 \left|\left|f\right|\right|_{H^1(I)}.
\end{align*}
     \end{enumerate}

\end{lem}
\begin{proof}
    H\"older's inequality gives
    \begin{align*}
        \left|\left|f-f_0\right|\right|_{C^0(I)}&=\sup_{t\in I} \left|\int_{t_0}^t f'(s)\dd s\right|\\
        &\le \int_I \left|f'(t)\right|\dd t\\
        &\le \frac{5}{4}\left|\left|f'\right|\right|_{L^2(I)},
    \end{align*}
    which gives (a). For (b), we use the mean value theorem to find a $t_0$ for which $\int_I f-f_0=0$. Then (a) gives 
    \begin{align*}
        \left|\left|f\right|\right|^2_{C^0(I)}&\le\left(\left|\left|f-f_0\right|\right|_{C^0(I)}+\left|\left|f_0\right|\right|_{C^0(I)}\right)^2 \\
        &\le \left(\frac{5}{4}\left|\left|f'\right|\right|_{L^2(I)}+\left|f_0\right|\right)^2\\
        &\le 2\left(\left(\frac{5}{4}\left|\left|f'\right|\right|_{L^2(I)}\right)^2+ \left|f_0\right|^2\right)\\
        &\le \frac{25}{8}\left(\left|\left|f'\right|\right|_{L^2(I)}^2+\left|\left|f-f_0\right|\right|^2_{L^2(I)}+t^*(f_0)^2\right)\\
        &=\frac{25}{8}\left(\left|\left|f'\right|\right|_{L^2(I)}^2+\left|\left|f-f_0+f_0\right|\right|_{L^2(I)}^{2}\right)\\
        &=\frac{25}{8}\left|\left|f\right|\right|_{H^1(I)}^2. 
    \end{align*}

\end{proof}

Lemma \ref{Sobolevembedding} can be interpreted as weak version of Morrey's inequality (an example of a Sobolev inequality).

\subsection{Estimates on terms in the ODE}
Choose $d_1,d_2\in \mathbb{N}\setminus\{0,1\}$. In Section \ref{geomprelim}, we wrote the Einstein equation for $\mathsf{O}(d_1+1)\times \mathsf{O}(d_2+1)$-invariant Einstein metrics on $S^{d_1+d_2+1}$ subject to one of the two singular boundary conditions as the following singular initial-value problem for $\eta(t)\in \mathbb{R}^5$: 
\begin{align}\label{etaequation}
    \eta'(t)=\frac{1}{t}L_{d_1d_2} \eta(t)+B_{d_1d_2}(\eta(t),\eta(t)), \qquad \eta(0)=(0,0,0,\alpha,\sqrt{\lambda}),
\end{align}
where 
\begin{align*}
    L_{d_1d_2}=\begin{pmatrix}
        0&0&0&0&0\\
        0&-d_2&0&0&0\\
        0&2&-2&0&0\\
        0&0&0&0&0\\
        0&0&0&0&0
    \end{pmatrix},
\end{align*}
and $B_{d_1d_2}:\mathbb{R}^5\to \mathbb{R}^5$ is the symmetric bi-linear form 
\begin{align*}
    B_{d_1d_2}(x,y)=
    \begin{pmatrix}
        -\frac{1}{2d_1}(x_2y_4+x_4y_2+x_1y_2+x_2y_1)\\
        -\frac{1}{2}(x_2y_3+x_3y_2)+d_1\left(x_1y_1+x_4y_4+x_1y_4+x_4y_1\right)-d_1x_5y_5\\
        -\frac{1}{d_2}x_3y_3+\frac{1}{d_2}\left(x_2y_3+x_3y_2\right)-(\frac{1}{d_1}+\frac{1}{d_2})x_2y_2-x_5y_5\\
        0\\
        0
    \end{pmatrix}.
\end{align*}  
In this subsection, we write some basic estimates for $L_{d_1d_2}$ and $B_{d_1d_2}$ in terms of the norms discussed in Subsection \ref{norms}. 

\begin{lem}
\label{linearprops}
The matrix $L_{d_1 d_2}$ for the linear term satisfies the following properties:
\begin{enumerate}[label=(\alph*)]
\item \label{linear:negdef} $-L_{d_1 d_2}$ is non-negative definite, that is, $x^\top L_{d_1 d_2} x \leq 0$ for all $x \in \mathbb{R}^5$; and
\item \label{linear:inv} for all $k=1,2,\dots$, $kI-L_{d_1d_2}$ is invertible, and $\left|(kI-L_{d_1d_2})^{-1}\right|\le \frac{5}{4k}$.
\end{enumerate}
\end{lem}
\begin{proof}
Beginning with \ref{linear:negdef}, for $x = (x_1,\dots,x_5)$ arbitrary, using Cauchy's inequality $2xy \leq x^2 + y^2$, 
\[
x^\top L_{d_1 d_2} x = -d_2 x_2^2 + 2 x_2 x_3 - 2 x_3^2 \leq -(d_2 - 1) x_2^2 - x_3^2 \leq 0,
\]
since $d_2 \geq 2$. To show \ref{linear:inv}, one can verify that 
    \begin{align*}
        (kI-L_{d_1d_2})^{-1}=\begin{pmatrix}
            \frac{1}{k}&0&0&0&0\\
            0&\frac{1}{d_2+k}&0&0&0\\
            0&\frac{2}{(k+2)(d_2+k)}&\frac{1}{k+2}&0&0\\
            0&0&0&\frac{1}{k}&0\\
            0&0&0&0&\frac{1}{k}
        \end{pmatrix}.
    \end{align*}
Using the fact that for all $k \geq 1$ and $d_2 \geq 2$,
\[
\frac{2}{(k+2)(k+d_2)} \leq \frac{2}{(k+2)^2} \leq \frac{2}{(k+2)^2 - (k-2)^2} = \frac{1}{4k},
\]
we find that 
\[
    \left|(kI-L_{d_1d_2})^{-1}x\right|\le \frac{1}{k}\left|x\right|+\frac{1}{4k}\left|x_2\right|\le \left(\frac{1}{k}+\frac{1}{4k}\right)\left|x\right|,
\]
as required. 
\end{proof}
We conclude our discussion on the linear term by understanding solutions of the linear problem 
\begin{align*}
    \mu'(t)=\frac{1}{t}L_{d_1d_2}\mu(t)+F(t), \qquad \mu(0)=0, 
\end{align*}
for time interval $[0,t^*]$. In the following Lemma, we assert that solutions to this problem are unique.
\begin{lem}\label{uniquenessODE}
Consider a differential equation of the form
\[
x'(t)=\frac{L_{d_1 d_2} x(t)}{t}+R(t,x(t))
\]
where $R$ is smooth in $t$ and $x(t)$. If a solution exists and $x(0)= 0$, then it is unique on any compact set.
\end{lem}
\begin{proof}
Let $x_1,x_2$ be two solutions such that $x_1(0)=x_2(0) = 0$. Now let $y = x_1 - x_2$, which satisfies
\[
y'(t) = \frac{L_{d_1 d_2} y(t)}{t} + R(t,x_1(t)) - R(t,x_2(t)). 
\]
Consider the evolution of $z(t)= \left|y(t)\right|^2$. Note that on any compact set, $R$ is Lipschitz continuous in both arguments. By Lemma \ref{linearprops} \ref{linear:negdef}, there exists some $C > 0$ such that
\begin{align*}
z'(t) = 2 y(t) \cdot y'(t) &= \frac{2 y(t)^\top L_{d_1 d_2} y(t)}{t} + y(t) \cdot [R(t, x_1(t)) - R(t,x_2(t))] \\
&\leq y(t) \cdot [R(t,x_1(t)) - R(t,x_2(t))] \\
&\leq C \left|y(t)\right|\left|x_1(t) - x_2(t)\right| = C z(t),
\end{align*}
and so by Gronwall's inequality, $z(t) = 0$ since $z(0) = 0$. 
\end{proof}

We use $\mathcal{L}_{d_1d_2}$ to denote the solution operator, i.e., the unique solution to this problem is given by $\mu=\mathcal{L}_{d_1d_2}F$. 
\begin{lem}\label{singularlinear}
    If $t^*\le 1.5$ and $d_2\in \{2,9\}$, then 
$\mathcal{L}_{d_1d_2}:C^k([0,t^*];\mathbb{R}^5)\to C^{k+1}([0,t^*];\mathbb{R}^5)$ 
    is bounded by $25$ if $k=0$, and by $600$ if $k=1$. 
\end{lem}
\begin{proof}
Observe that the relationship between $F$ and $\mu$ can be phrased as
    \begin{align*}
        \left(\text{exp}(-L_{d_1d_2}\ln(t))\mu(t)\right)'=\exp(-L_{d_1d_2}\ln(t))F(t),
    \end{align*}
    so that for each $0<t_0<t$, 
    \begin{align*}
        \mu(t)-\exp(L_{d_1 d_2}(\ln(t)-\ln(t_0)))\mu(t_0)=\exp(L_{d_1 d_2}\ln(t))\int_{t_0}^{t}\exp(-L_{d_1 d_2}\ln(s))F(s)\dd s.
    \end{align*}
   We compute 
\begin{align*}
    \text{exp}(L_{d_1d_2}(\ln(t)-\ln(s)))=
    \begin{cases}
        \begin{pmatrix}
           1&0&0&0&0\\
           0&\left(\frac{s}{t}\right)^2&0&0&0\\
           0&2\left(\frac{s}{t}\right)^2\ln\left(\frac{t}{s}\right)&\left(\frac{s}{t}\right)^2&0&0\\
           0&0&0&1&0\\
           0&0&0&0&1
       \end{pmatrix} \ \text{if} \ d_2=2,\\
       \begin{pmatrix}
           1&0&0&0&0\\
           0&\left(\frac{s}{t}\right)^{9}&0&0&0\\
           0&-\frac{2}{7}\left(\frac{s}{t}\right)^{9}\left(1-\left(\frac{t}{s}\right)^{7}\right)&\left(\frac{s}{t}\right)^2&0&0\\
           0&0&0&1&0\\
           0&0&0&0&1
       \end{pmatrix}\ \text{if} \ d_2=9.
    \end{cases}
\end{align*}

   Therefore, for each $s<t$, $\left|\text{exp}(L_{29}(\ln(t)-\ln(s)))\right|\le 2 $. We can therefore take $t_0\to 0$, giving 
    \begin{align*}
        \mu(t)&=\int_0^t \exp(L_{d_1d_2}(\ln(t)-\ln(s)))F(s)\dd s.
        \end{align*}
        Assuming further that $F$ is differentiable, the mean value theorem allows us to write 
        \begin{align*}
        \mu(t)&=\underbrace{\int_0^t \exp(L_{d_1d_2}(\ln(t)-\ln(s)))F(0) \dd s}_{A(t)}+\underbrace{\int_0^t \exp(L_{d_1d_2}(\ln(t)-\ln(s)))s F'(s^*)\dd s}_{B(t)}, 
    \end{align*}
    where $s^*\in [0,s]$ depends on $s$. 
    We now estimate the norm of $\mu$ in terms of $F$. First observe that 
    \begin{align*}
        |\mu(t)|\le \int_0^t 2|F(s)|\dd s\le 3.5 \|F\|_{C^0}.
    \end{align*}
    since we are only interested in values of $t\le 1.5$. Furthermore, 
    \begin{align*}
        \left|\frac{\mu(t)}{t}\right|\le \frac{1}{t}\int_0^t 2|F(s)|\dd s\le 2\|F\|_{C^0}.
    \end{align*}
    Since $\mu'(t)=\frac{1}{t}L_{d_1d_2}\mu(t)+F(t)$, and $|L_{d_1d_2}|\le 10$ we conclude that 
    \begin{align*}
\|\mu\|_{C^1}&=\|\mu\|_{C^0}+\|\mu'\|_{C^0}\\
&\le (3.5+20+1)\|F\|_{C^0}\\
&\le 25 \|F\|_{C^0}.
\end{align*}
For the $C^1\mapsto C^2$ estimate, observe that with the decomposition $\mu(t)=A(t)+B(t)$ described above, the second derivative of $A(t)$ vanishes, while $B'(0)=0$ and \begin{align*}
   B''(t)=\frac{L_{d_1d_2} B'(t)}{t}+F'(t)-\frac{L_{d_1d_2}B(t)}{t^2}. 
\end{align*} 
We estimate 
\begin{align*}
    \left|\frac{B(t)}{t^2}\right|&=\left|\frac{1}{t}\int_0^t \exp(L_{d_1d_2}(\ln(t)-\ln(s)))\left(\frac{s}{t}\right) F'(s^*)\dd s\right|\\
    &\le \sup_{0<s<t}|\exp(L_{d_1d_2}(\ln(t)-\ln(s)))F'(s^*)|\\
    &\le 2 \|F'\|_{C^0},
\end{align*}which implies 
\begin{align*}
    \left|B''(t)-\frac{L_{d_1d_2}B'(t)}{t} \right|\le 21 \|F'\|_{C^0},
\end{align*}
and the $C^0\mapsto C^1$ bound gives 
\begin{align*}
    \|B'\|_{C^1}\le 525 \|F'\|_{C^0}.
\end{align*} 
Combining with $\|B\|_{C^0}\le 6 \|F'\|_{C^0}$ and 
$\|A\|_{C^1}\le 3.5 \left|F(0)\right|$ gives the result. 
\end{proof}
Now we look at estimating the non-linear term. 
\begin{lem}\label{BE}
For all $x,y\in \mathbb{R}^5$, $\left|B_{d_1d_2}(x,y)\right|\le 3(d_1+1)\left|x\right|\left|y\right|$.
\end{lem}
\begin{proof}
  With $x=(x_1,x_2,x_3,x_4,x_5)$ and $y=(y_1,y_2,y_3,y_4,y_5)$, we have 
  \begin{align*}
      \left|x\right|^2\left|y\right|^2=\sum_{i,j=1}^{5}\left|x_i\right|^2\left|y_j\right|^2,
  \end{align*}
  and 
  \begin{align*}
      \left|B_{d_1d_2}(x,y)\right|^2&=\frac{1}{4d_1^2}(x_2y_4+x_4y_2+x_1y_2+x_2y_1)^2\\
      &+\left(-\frac{1}{2}(x_2y_3+x_3y_2)+d_1\left(x_1y_1+x_4y_4+x_1y_4+x_4y_1\right)-d_1x_5y_5\right)^2\\
      &+\left(-\frac{1}{d_2}x_3y_3+\frac{1}{d_2}\left(x_2y_3+x_3y_2\right)-(\frac{1}{d_1}+\frac{1}{d_2})x_2y_2-x_5y_5\right)^2.
  \end{align*}
 Therefore, using 
 $\left(\sum_{i=1}^{N} \beta_i\right)^2\le N \sum_{i=1}^{N} \beta_i^2$ (discrete H\"older's inequality), we obtain 
  \begin{align*}
      \left|B_{d_1d_2}(x,y)\right|^2&\le \frac{1}{d_1^2}\left((x_2y_4)^2+(x_4 y_2)^2+(x_1 y_2)^2+(x_2y_1)^2\right)\\
      &\le 7\left(\frac{(x_2y_3)^2+(x_3y_2)^2}{4}+d_1^2 \left((x_1y_1)^2+(x_4y_4)^2+(x_1y_4)^2+(x_4 y_1)^2+(x_5 y_5)^2\right)\right)\\
      &+ 5\left(\frac{(x_3y_3)^2+(x_2y_3)^2+(x_3y_2)^2}{d_2^2}+\left(\frac{1}{d_1}+\frac{1}{d_2}\right)^2(x_2 y_2)^2+(x_5 y_5)^2\right).
  \end{align*}
  Collecting like terms, we see that the largest coefficient lands on the $(x_5 y_5)^2$ term, which is $7 d_1^2+5$. Taking the square root gives the result.  
\end{proof}
\begin{lem}\label{Btermok}
   For all $x,y\in \mathbb{R}^5$, we have $\left|\langle x,B_{d_1d_2}(x,y)\rangle\right|\le C\left|x\right|^2$, 
   where 
   \begin{align*}
       C(d_1,d_2,y)= \left|y_1\right|(d_1+\frac{1}{4d_1})+\left|y_2\right|(\frac{1}{2d_1}+\frac{3}{2d_2}+\frac{1}{4})+\left|y_3\right|(\frac{1}{2d_2}+\frac{1}{2})+\left|y_4\right|(d_1+\frac{1}{4d_1})+\left|y_5\right|\frac{d_1+1}{2}.
   \end{align*} 
\end{lem}
\begin{proof}
    Using the triangle and Young's inequalities, we estimate
    \begin{align*}
       \left| x\cdot B(x,y)\right|&\le
        \left|-\frac{x_1}{2d_1}(x_2y_4+x_4y_2+x_1y_2+x_2y_1)\right|\\
        &+\left|
        -\frac{x_2}{2}(x_2y_3+x_3y_2)+d_1x_2\left(x_1y_1+x_4y_4+x_1y_4+x_4y_1\right)-d_1x_2x_5y_5\right|\\
        &
        +\left|-\frac{x_3}{d_2}x_3y_3+\frac{x_3}{d_2}\left(x_2y_3+x_3y_2\right)-x_3(\frac{1}{d_1}+\frac{1}{d_2})x_2y_2-x_3x_5y_5\right|\\
        &\le x_1^2 \left(\frac{\left|y_4\right|+3\left|y_2\right|+\left|y_1\right|}{4d_1}+d_1\frac{\left|y_1\right|+\left|y_4\right|}{2}\right)\\
        &+x_2^2\left(\frac{\left|y_4\right|+\left|y_1\right|}{4d_1}+\frac{\left|y_3\right|}{2}+\frac{\left|y_2\right|}{4}+d_1(\left|y_1\right|+\left|y_4\right|)+d_1 \frac{\left|y_5\right|}{2}+\frac{\left|y_3\right|}{2 d_2}+(\frac{1}{2d_1}+\frac{1}{2d_2})\left|y_2\right|\right)\\
        &+x_3^2\left(\frac{\left|y_2\right|}{4}+\frac{\left|y_3\right|+\left|y_2\right|}{d_2}+\frac{\left|y_3\right|}{2d_2}+(\frac{1}{2d_1}+\frac{1}{2d_2})\left|y_2\right|+\frac{\left|y_5\right|}{2}\right)\\
        &+x_4^2 \left(\frac{\left|y_2\right|}{4d_1}+\frac{d_1}{2}\left(\left|y_1\right|+\left|y_4\right|\right)\right)+x_5^2\left(\frac{(d_1+1)\left|y_5\right|}{2}\right).
    \end{align*}
    The result follows from the observation that $C(d_1,d_2,y)$ is greater than the supremum of these five coefficients. 
\end{proof}
 
\section{Plots of $(A(\alpha),\Omega(\omega))$}
In this Appendix, we use our high-precision numerical analysis to plot the $A(\alpha)$ and $\Omega(\omega)$ curves from \eqref{Aalphadef} and \eqref{Omegaomegadef} for various choices of $d_1,d_2$. Intersections coincide with $\mathsf{O}(d_1)\times \mathsf{O}(d_2)$-invariant Einstein metrics on $S^n$, where $n=d_1+d_2+1$. 
\subsection{$n = 9$}
B\"ohm \cite{Bohm98} demonstrated that there are infinitely-many intersections. We can see some of these by zooming in. 
\begin{center}
\includegraphics[height=5cm]{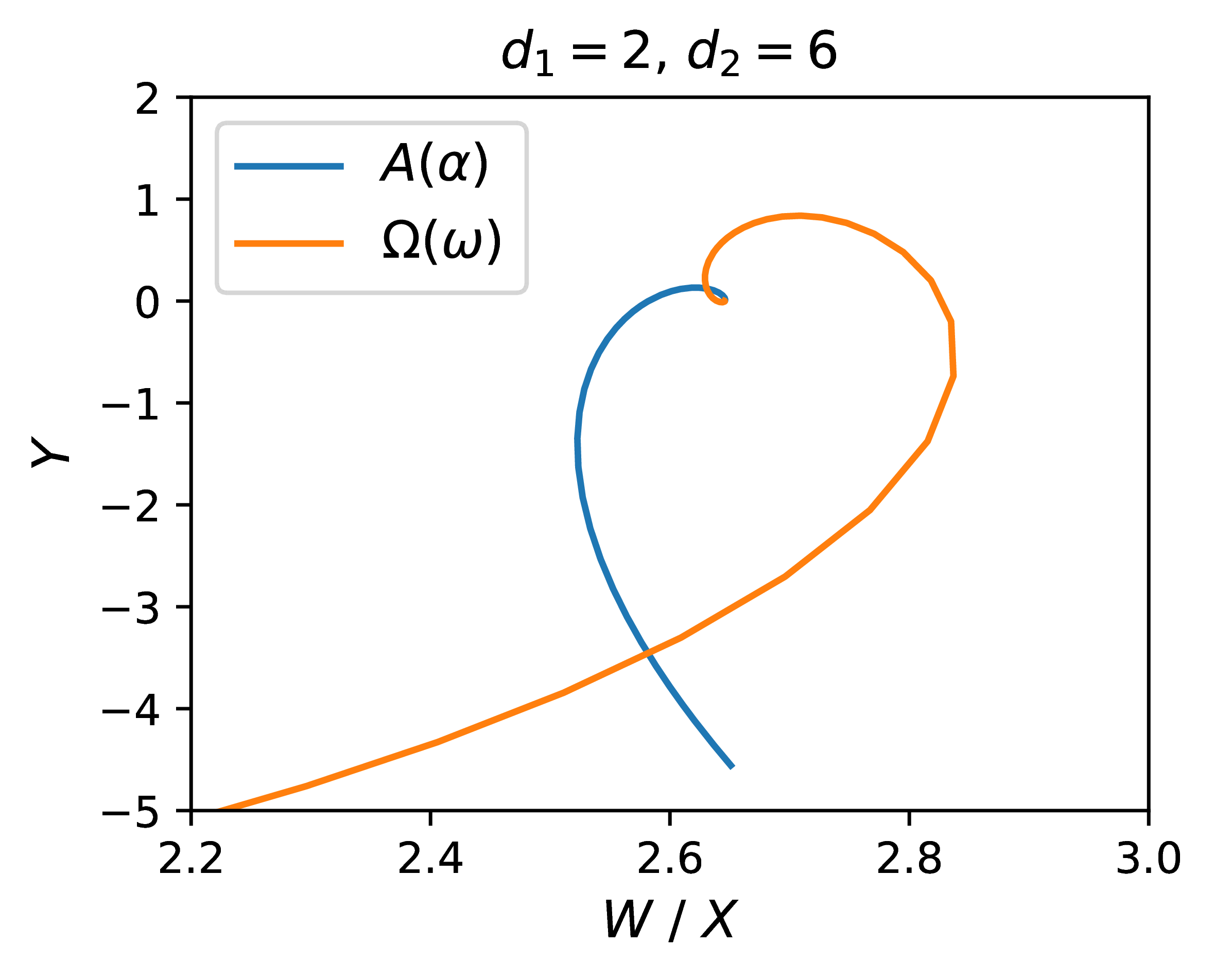}
\includegraphics[height=5cm]{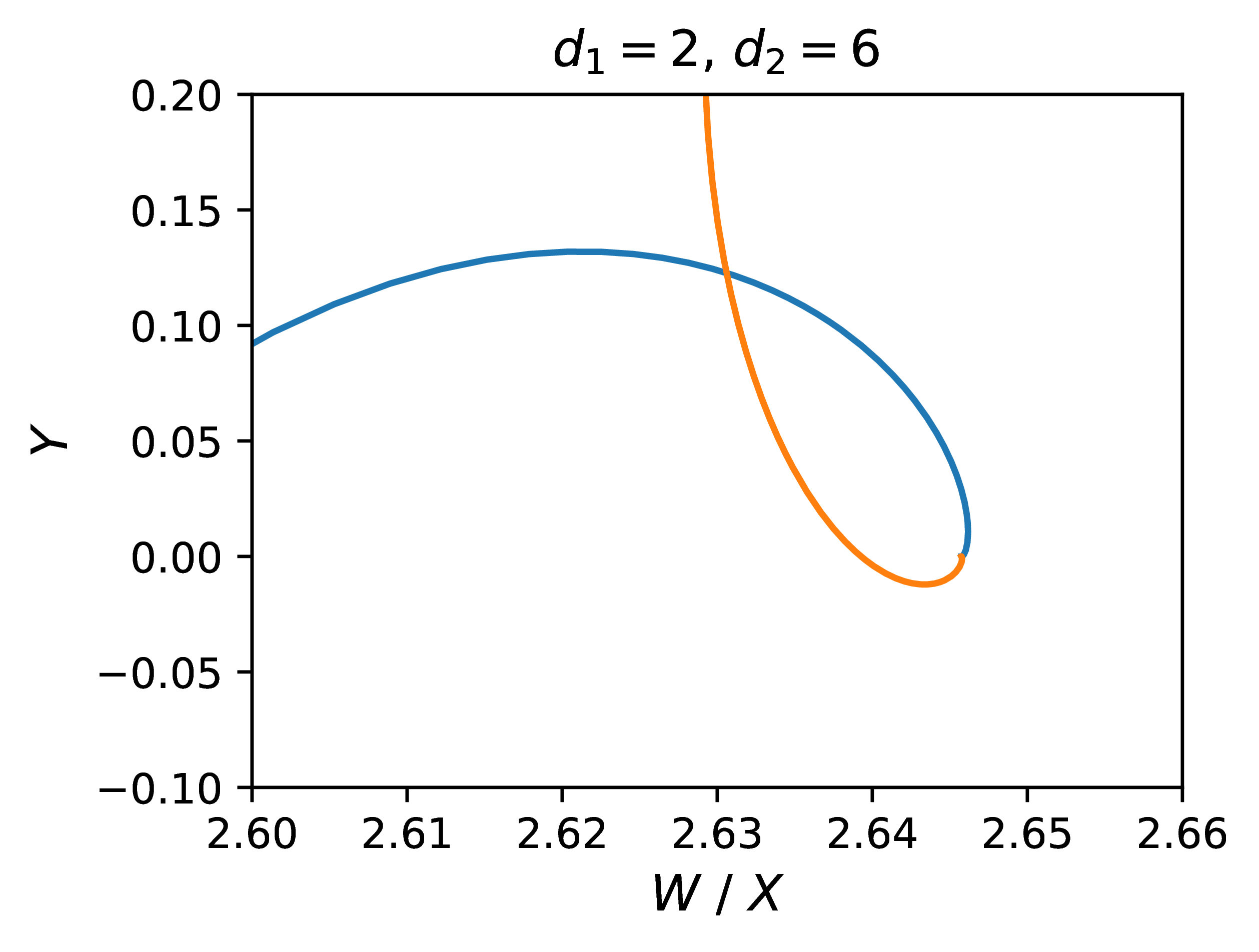}

\includegraphics[height=5cm]{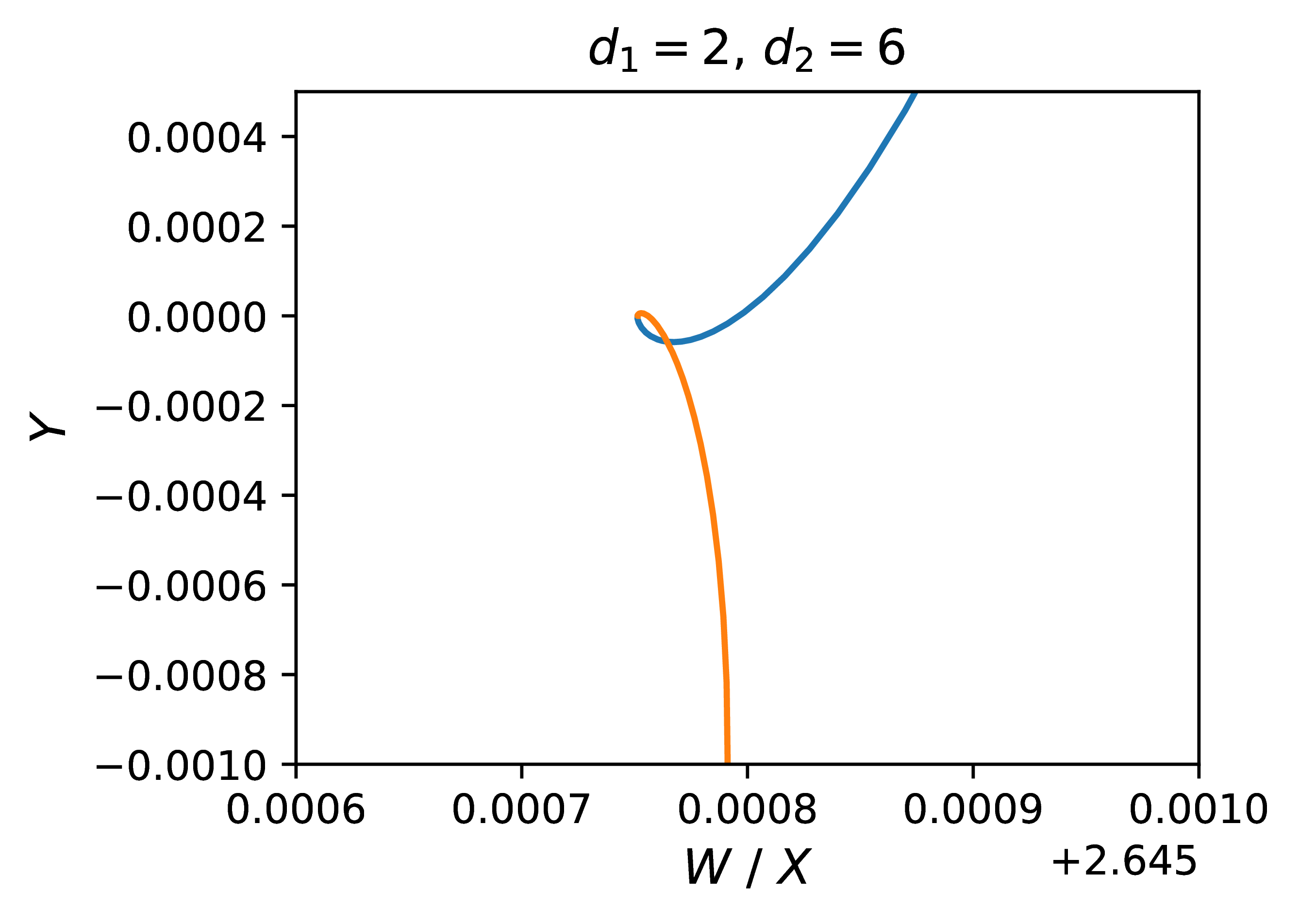}
\end{center}

\subsection{$n = 10$}
Nienhaus-Wink \cite{NienhausWink} showed that if $(d_1,d_2)\in \{(2,7),(3,6),(4,5)\}$, then there are two intersections (the round-Einstein metric and a non-round Einstein metric), and conjectured that there are no more intersections to find. 
\begin{longtable}{cc}
\includegraphics[height=5cm]{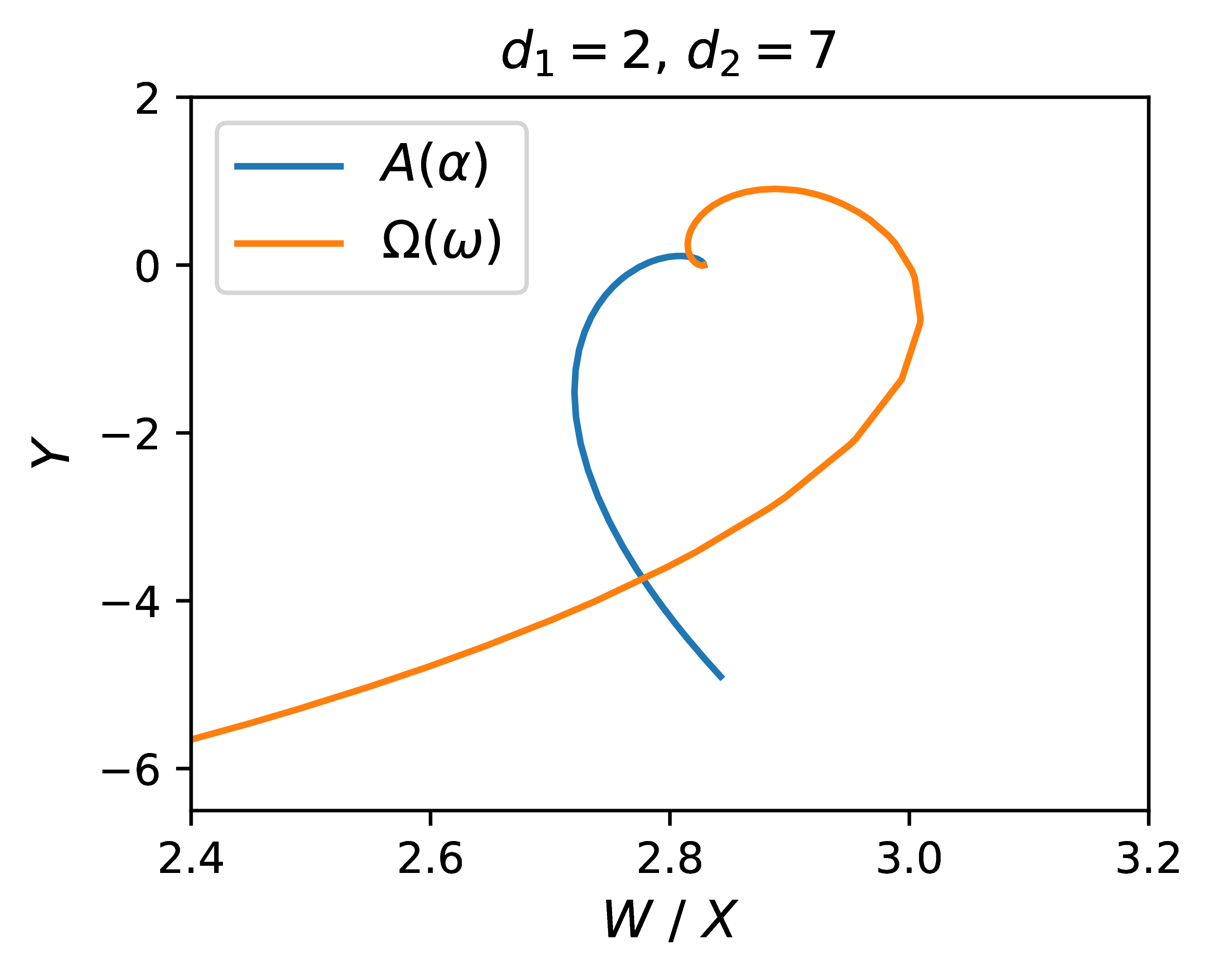} &
\includegraphics[height=5cm]{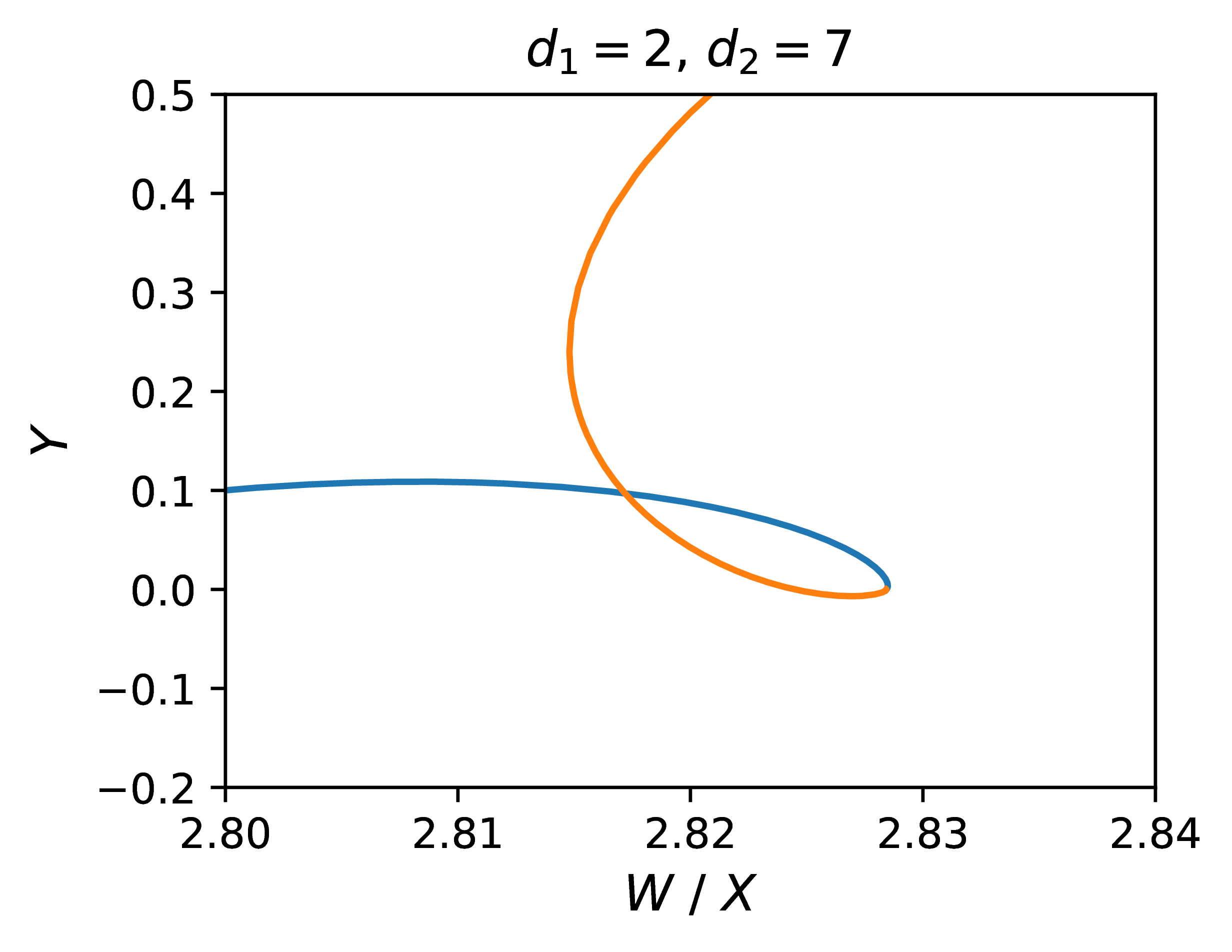} \\
\includegraphics[height=5cm]{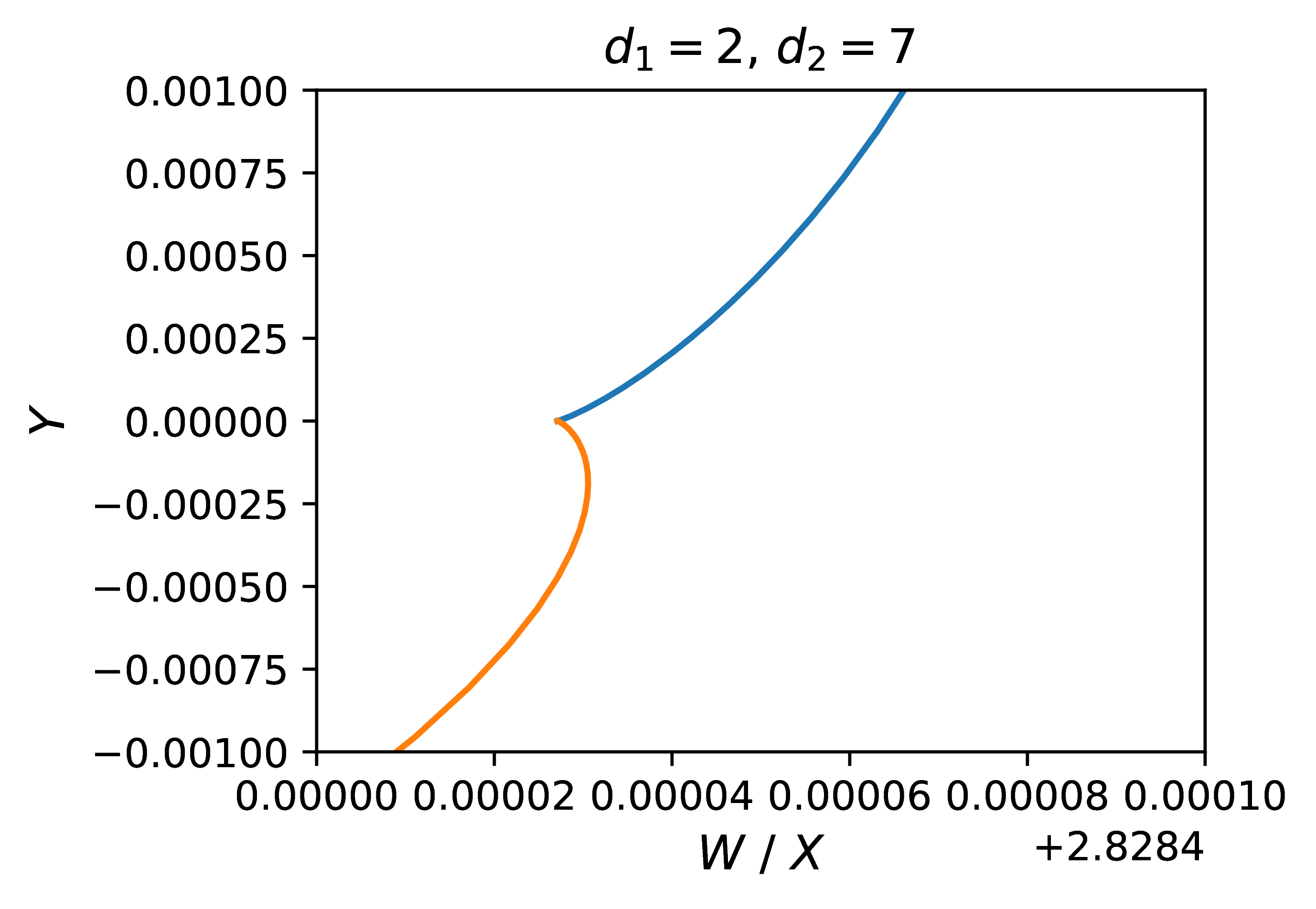} &
\includegraphics[height=5cm]{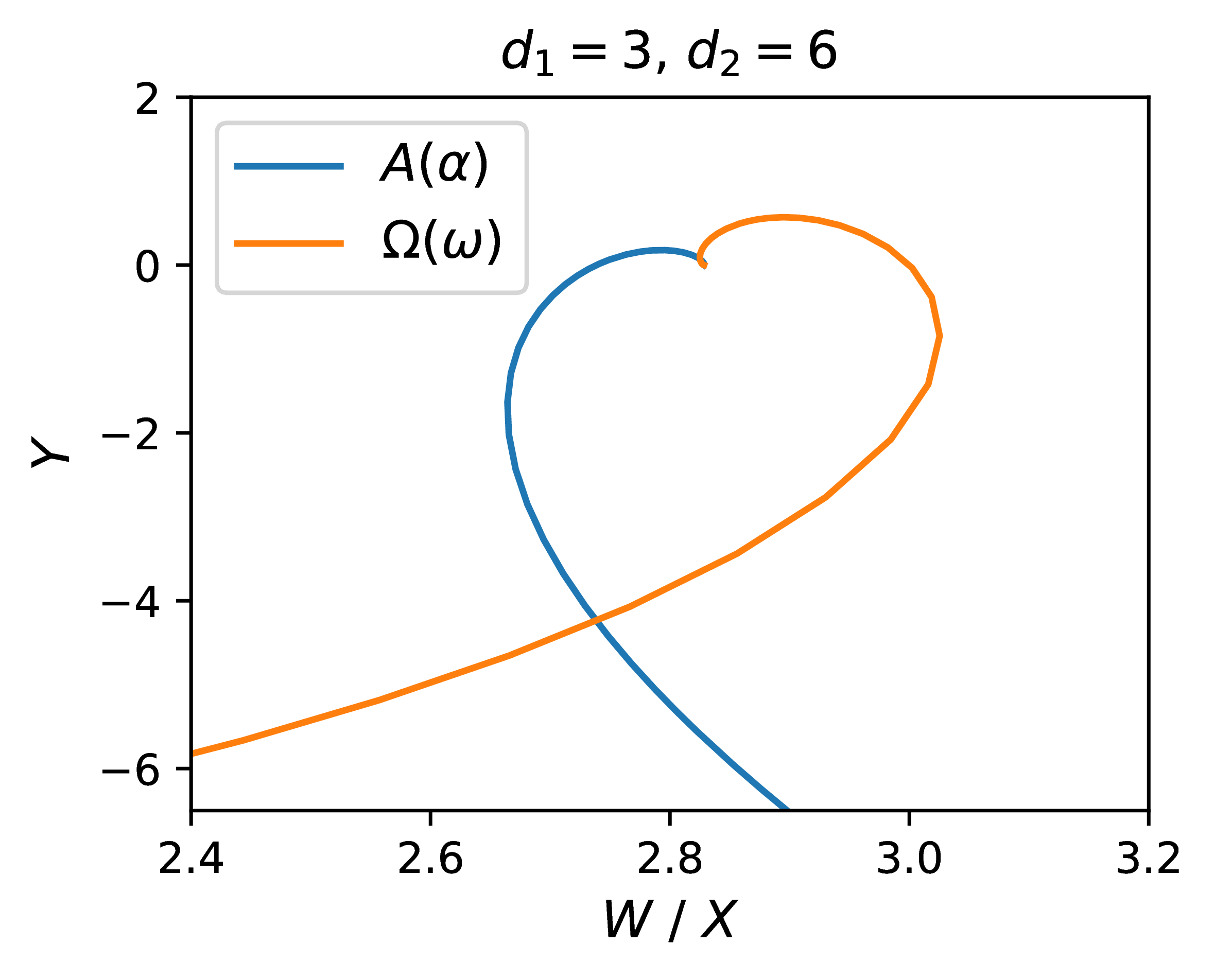} \\
\includegraphics[height=5cm]{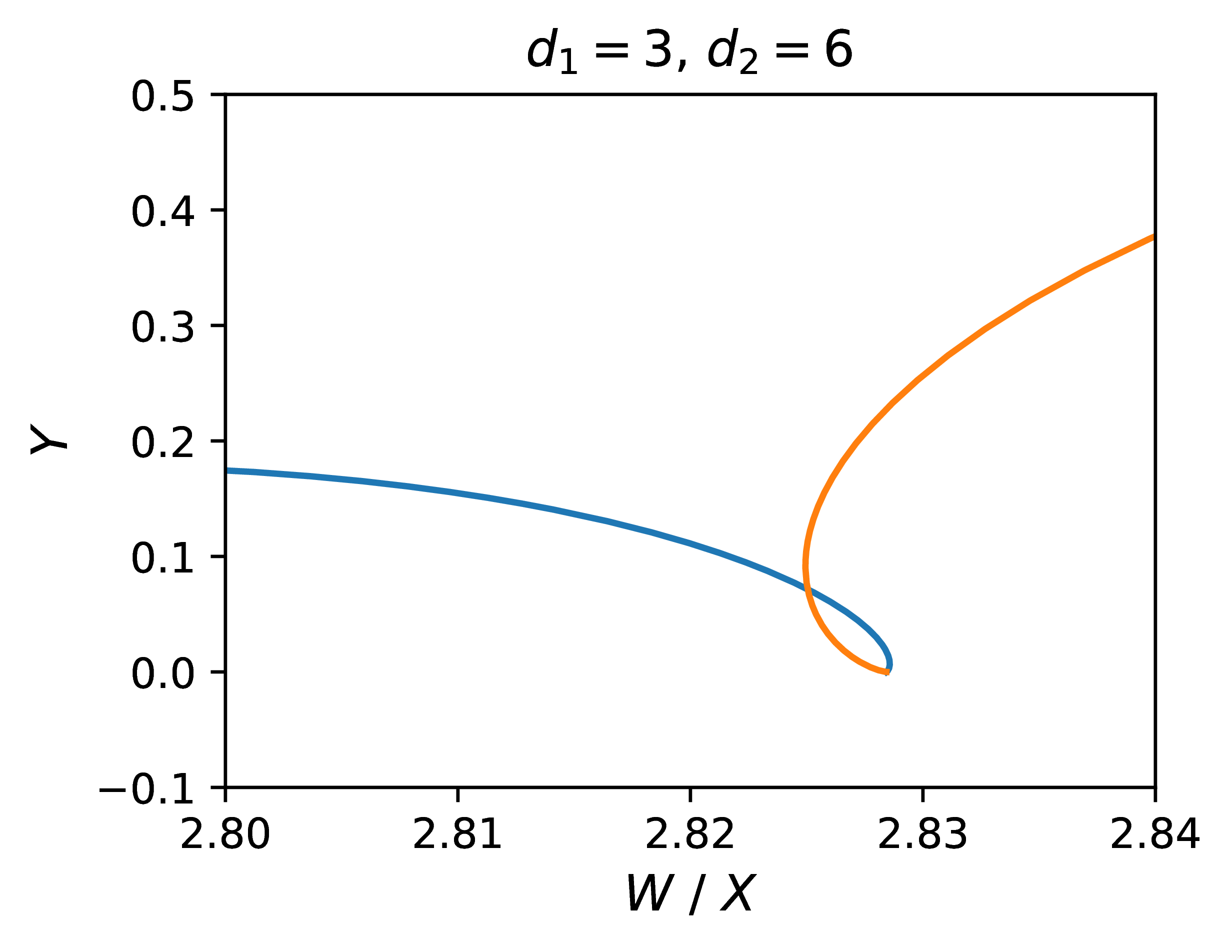} &
\includegraphics[height=5cm]{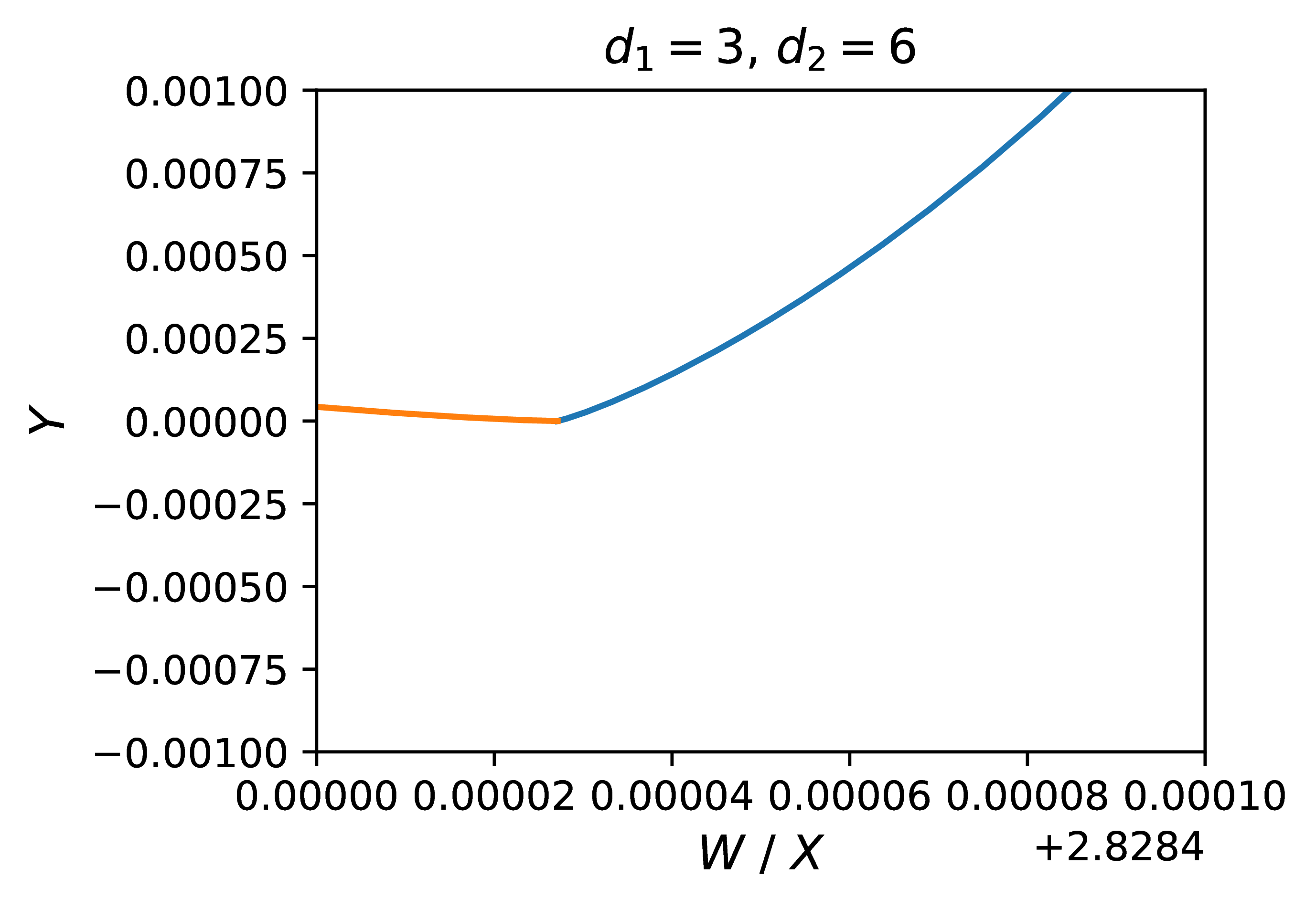} \\
\includegraphics[height=5cm]{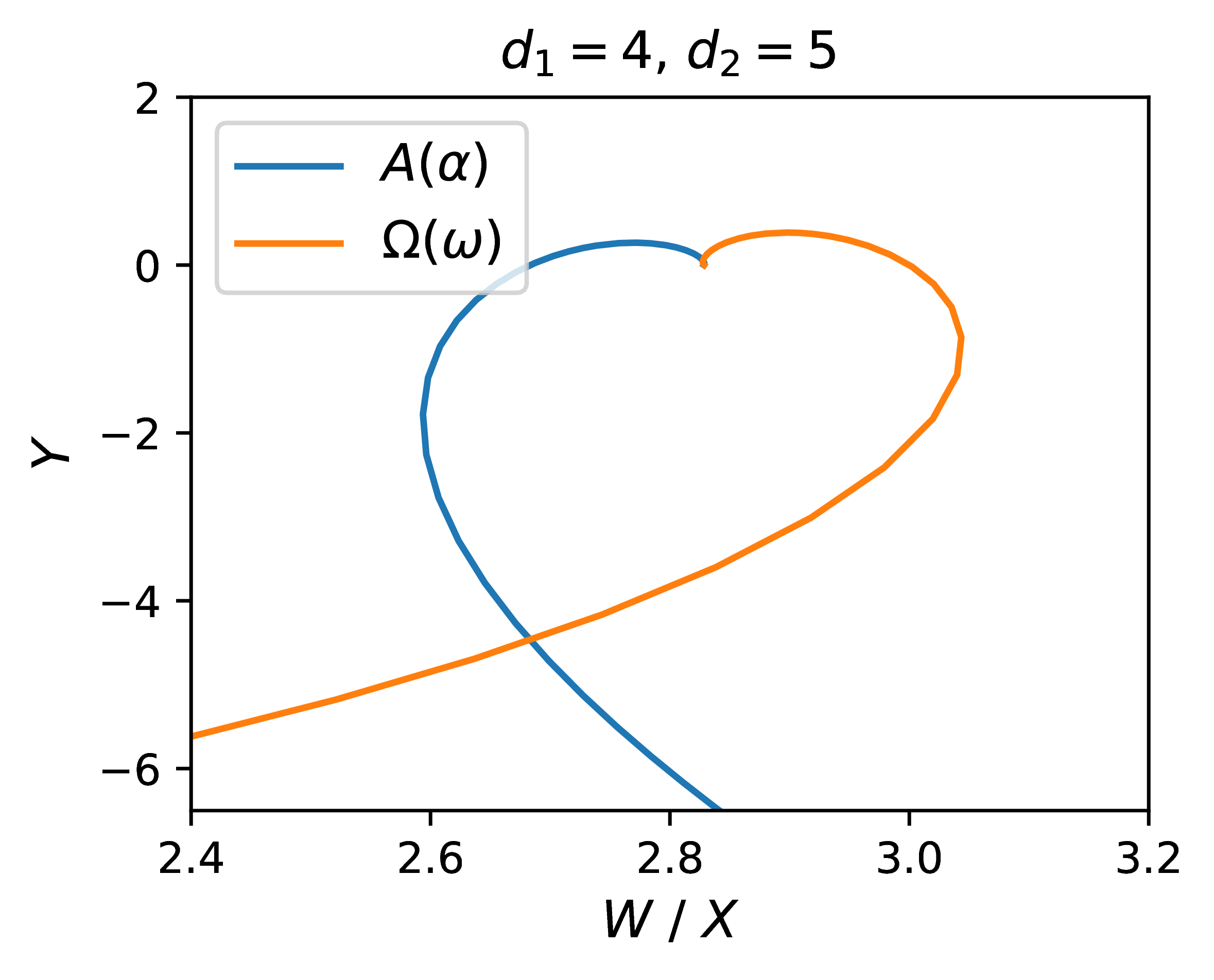} &
\includegraphics[height=5cm]{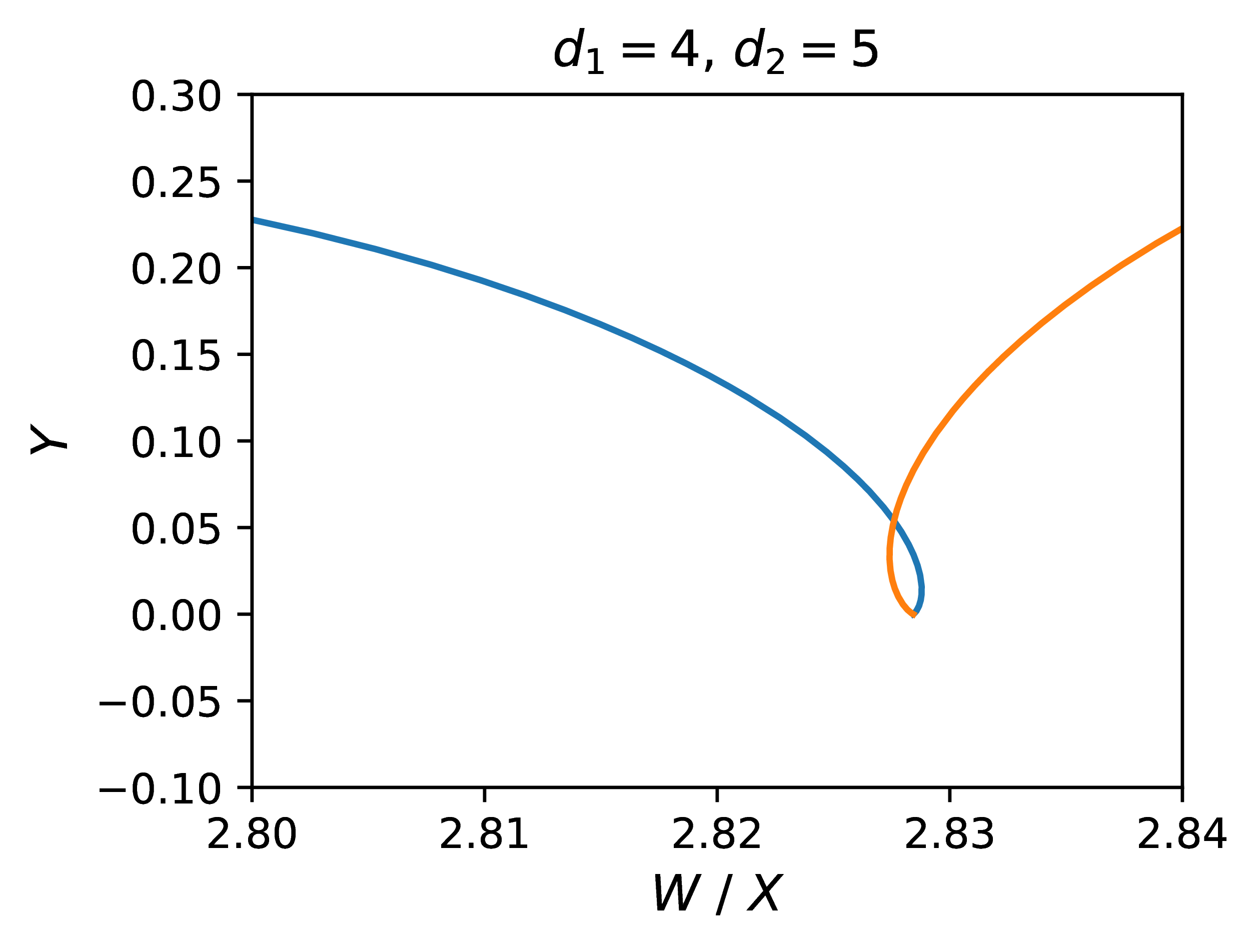}
\end{longtable}

\begin{center}
\includegraphics[height=5cm]{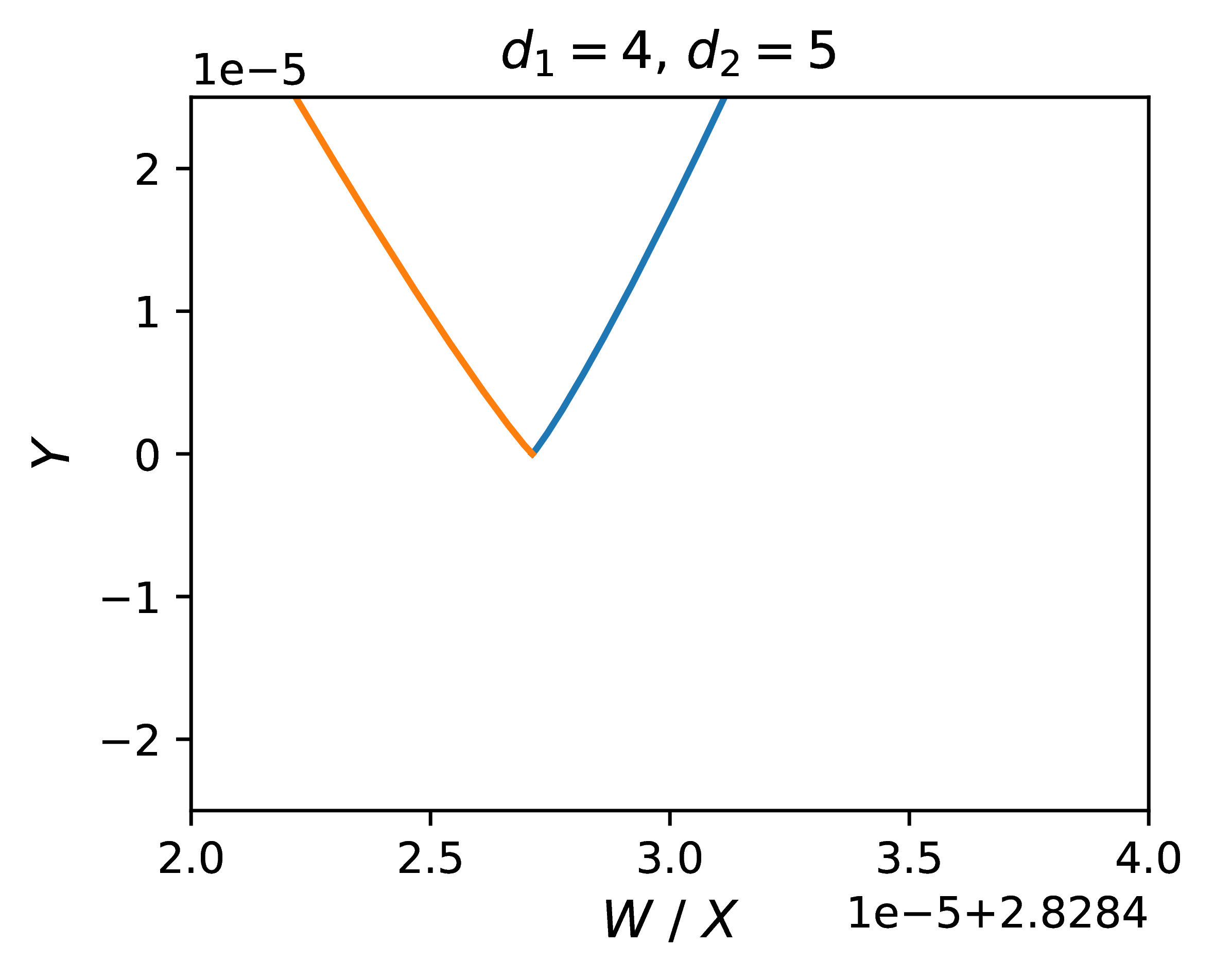}
\end{center}

\subsection{$n = 12$}
In this paper, we prove that if $(d_1,d_2)=(2,9)$, then there is a non-round Einstein metric. The plots below suggest that there are no new Einstein metrics for any other combination of $d_1,d_2$ summing to $11$. However, by zooming in on the $(2,9)$ case, there appears to be a \textit{second} non-round Einstein metric. 
\begin{center}
\includegraphics[height=5cm]{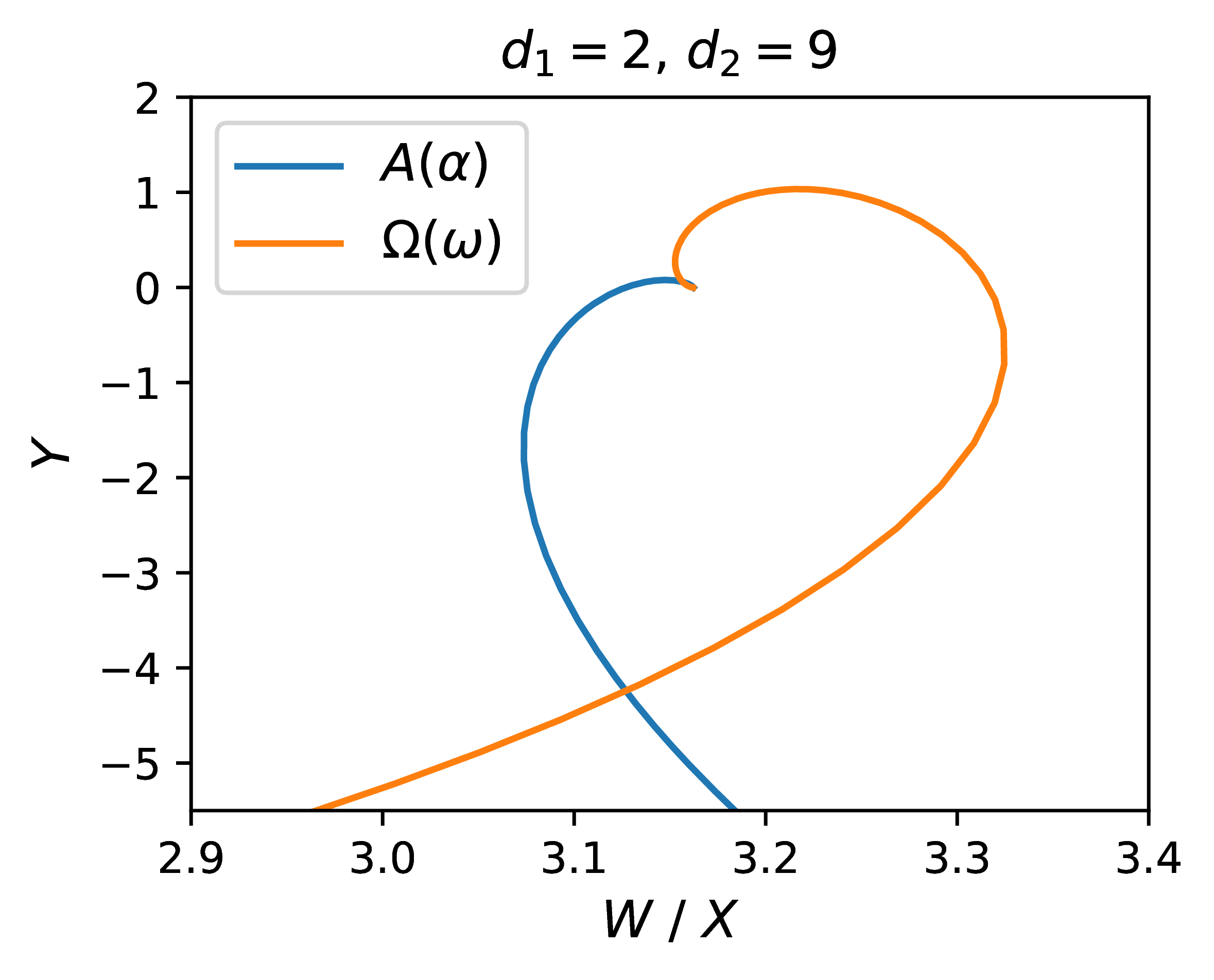} 
\includegraphics[height=5cm]{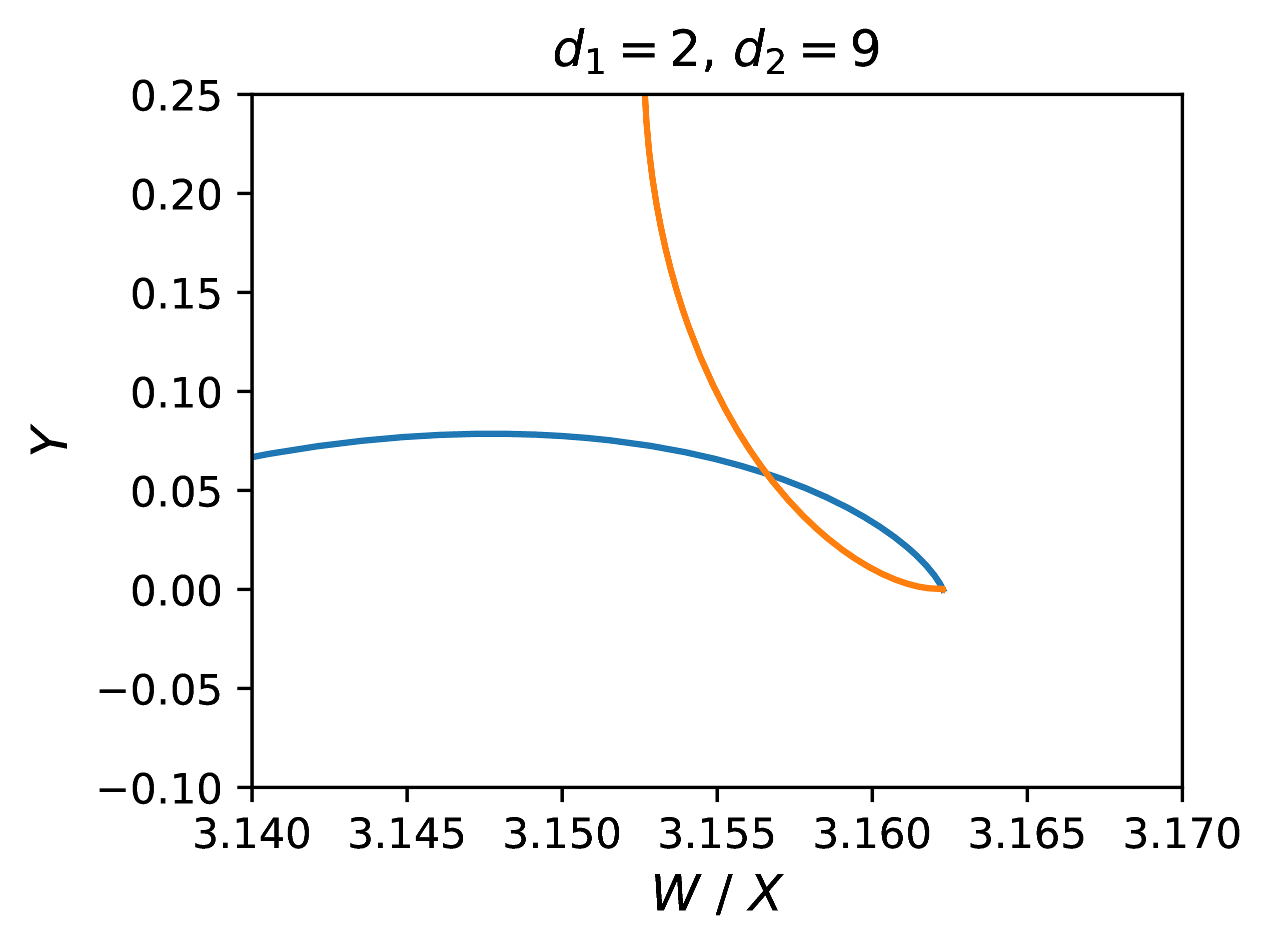}

\includegraphics[height=5cm]{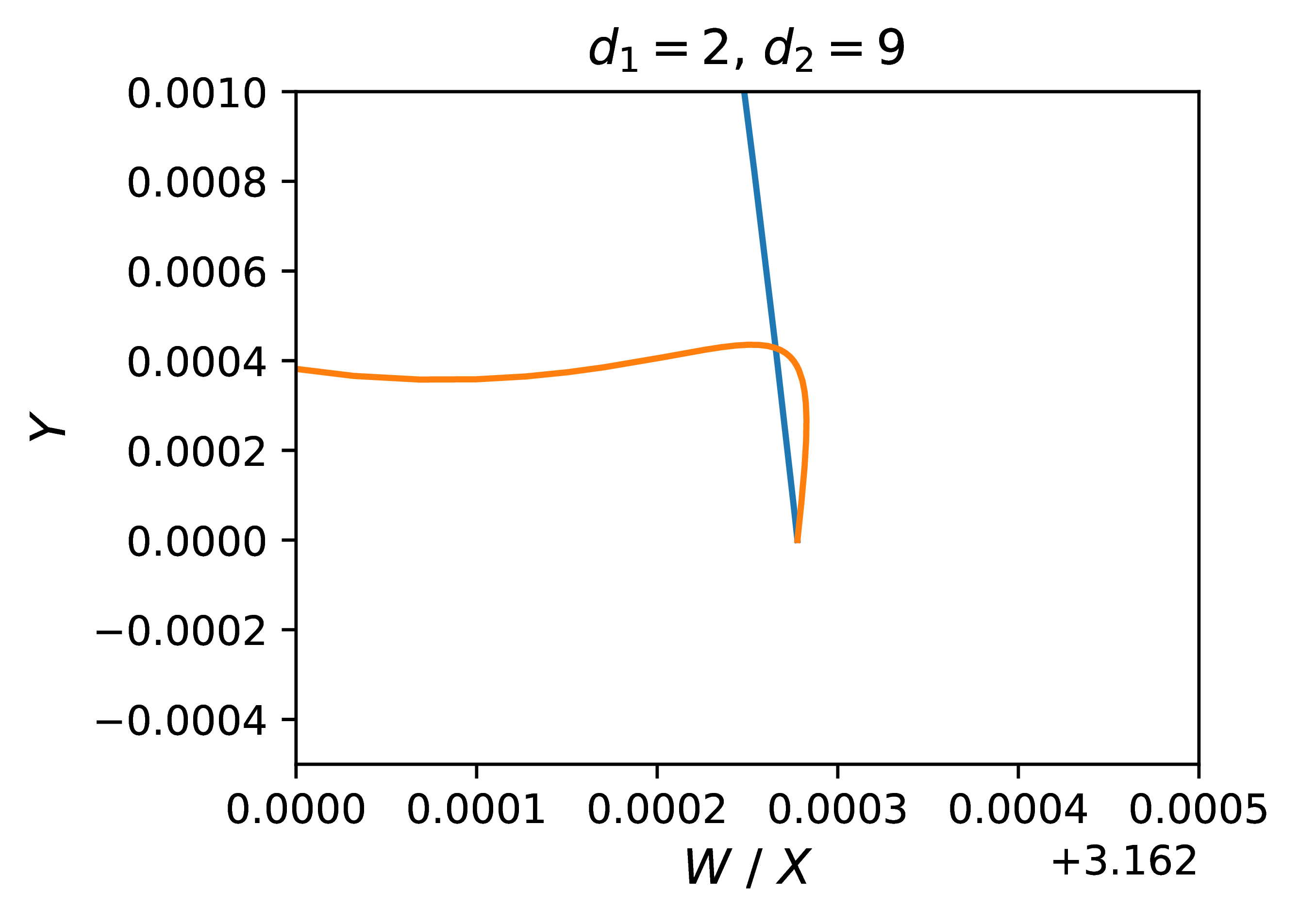}
\end{center}

\begin{longtable}{cc}
\includegraphics[height=5cm]{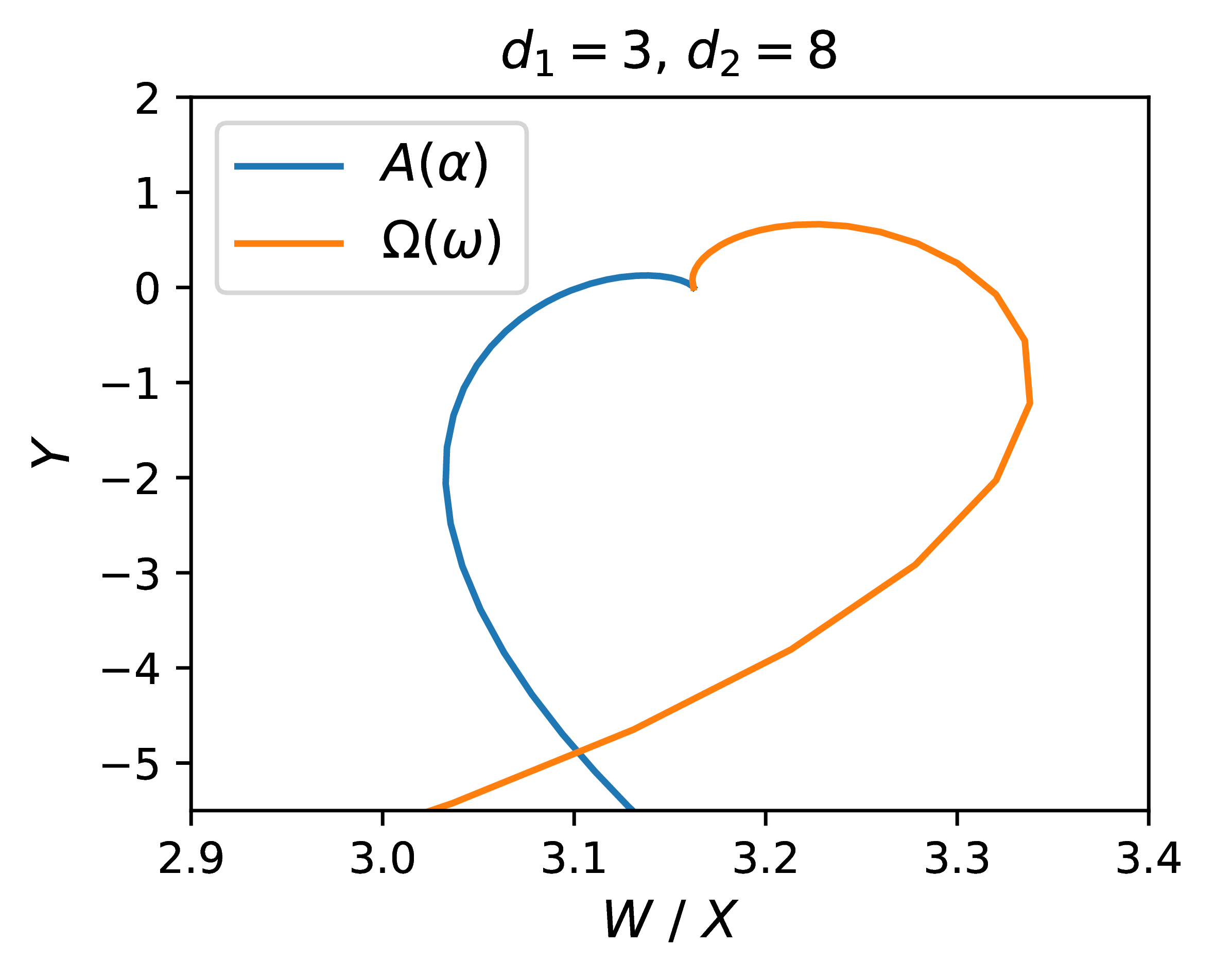} &
\includegraphics[height=5cm]{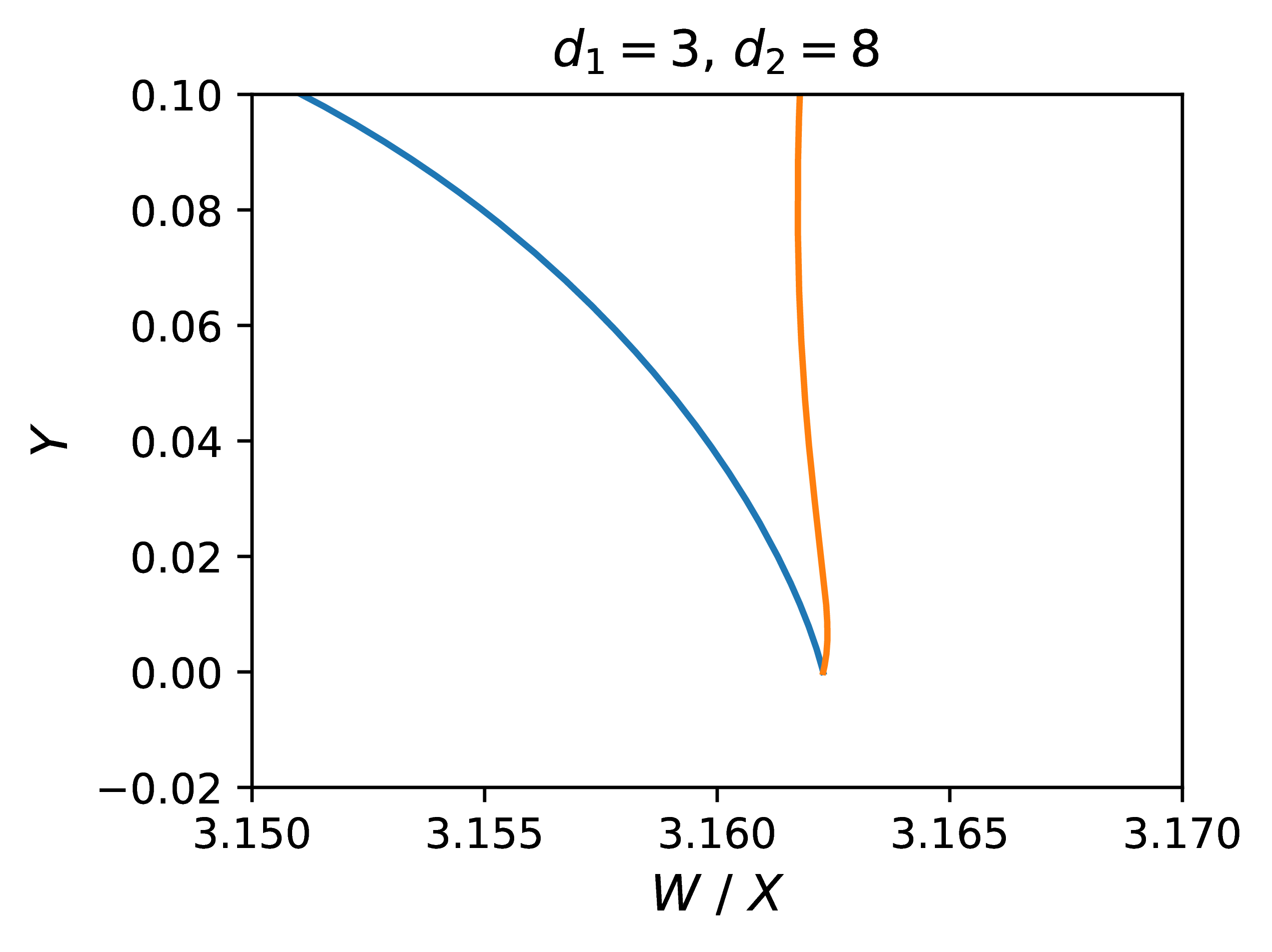} \\
\includegraphics[height=5cm]{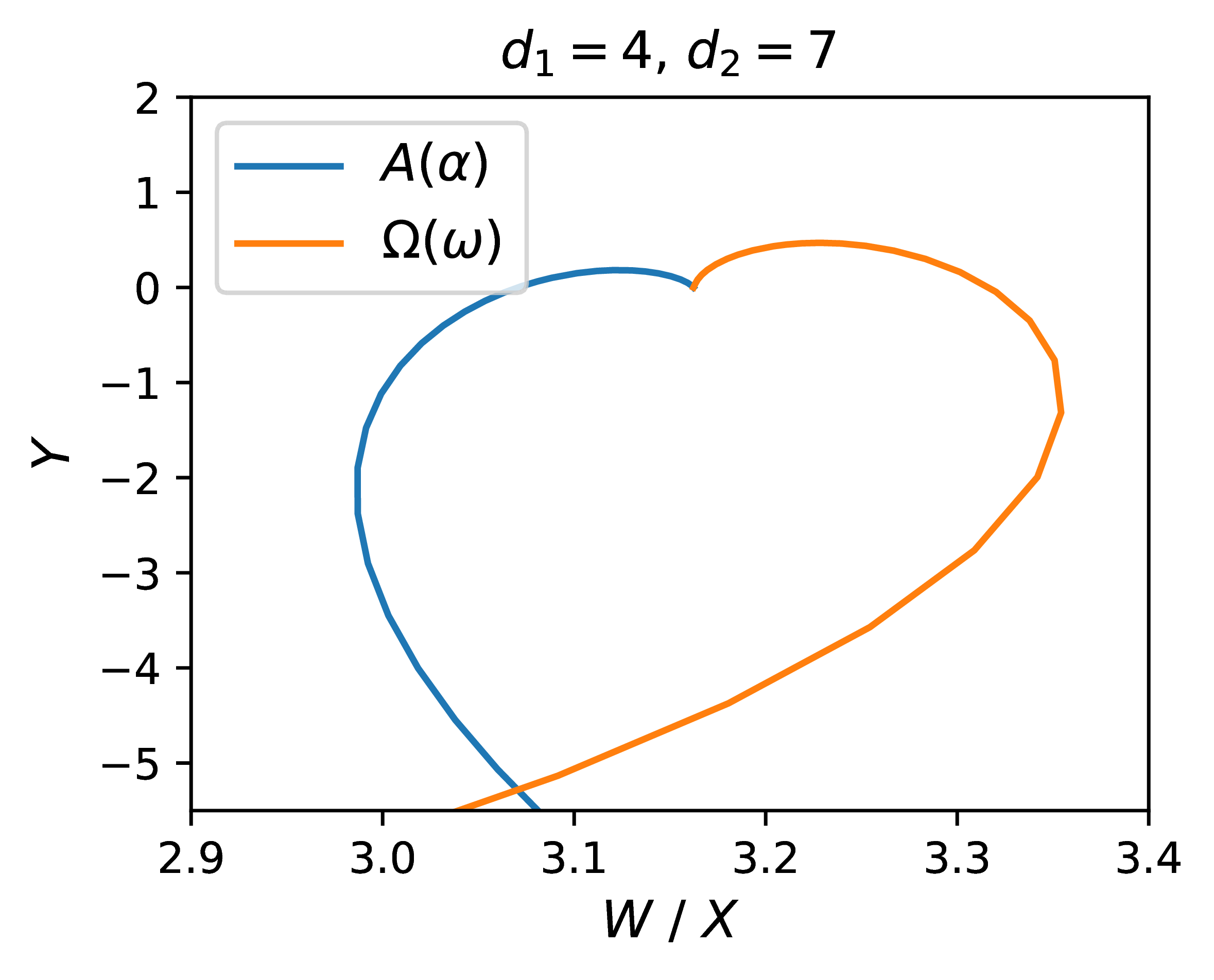} &
\includegraphics[height=5cm]{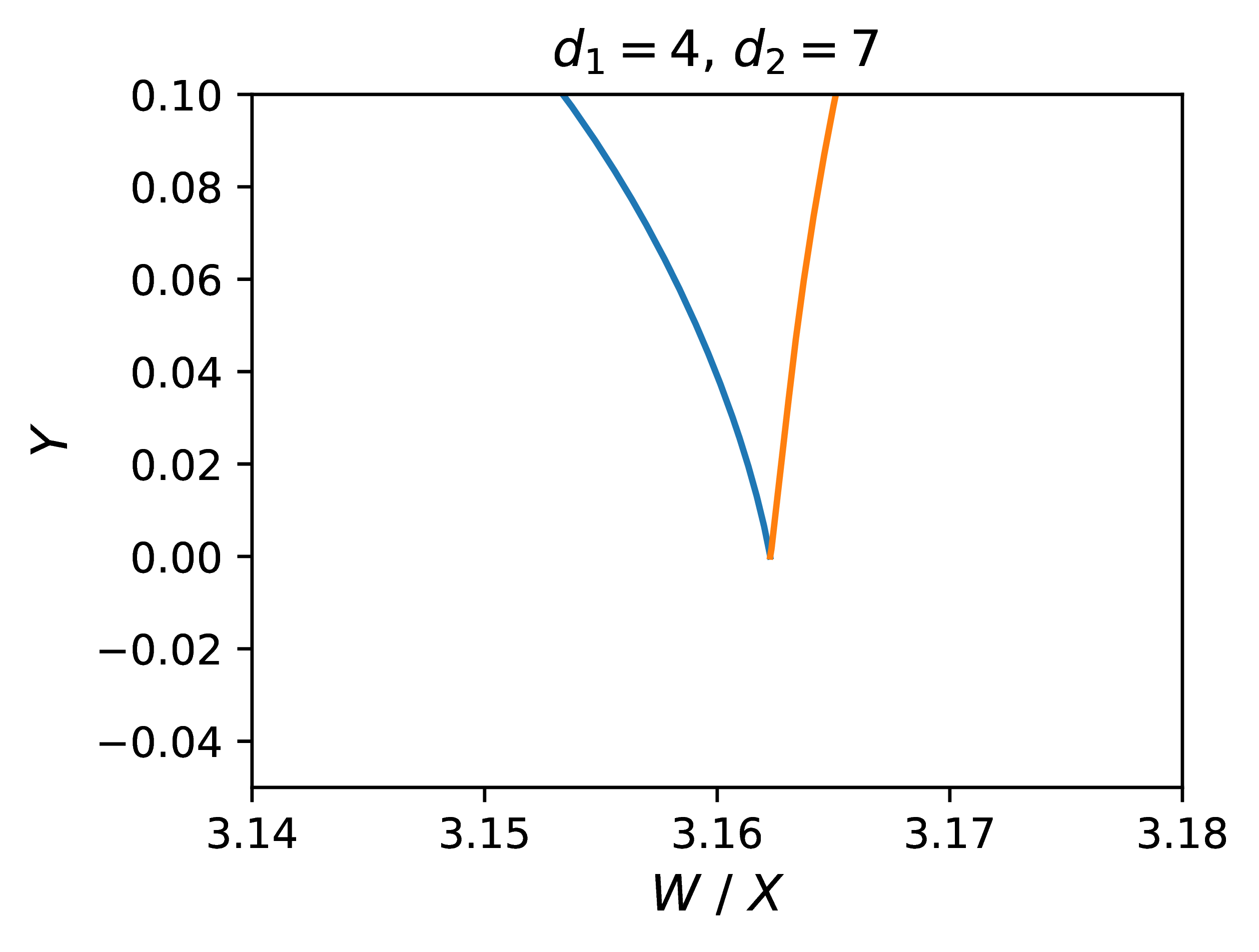} \\
\includegraphics[height=5cm]{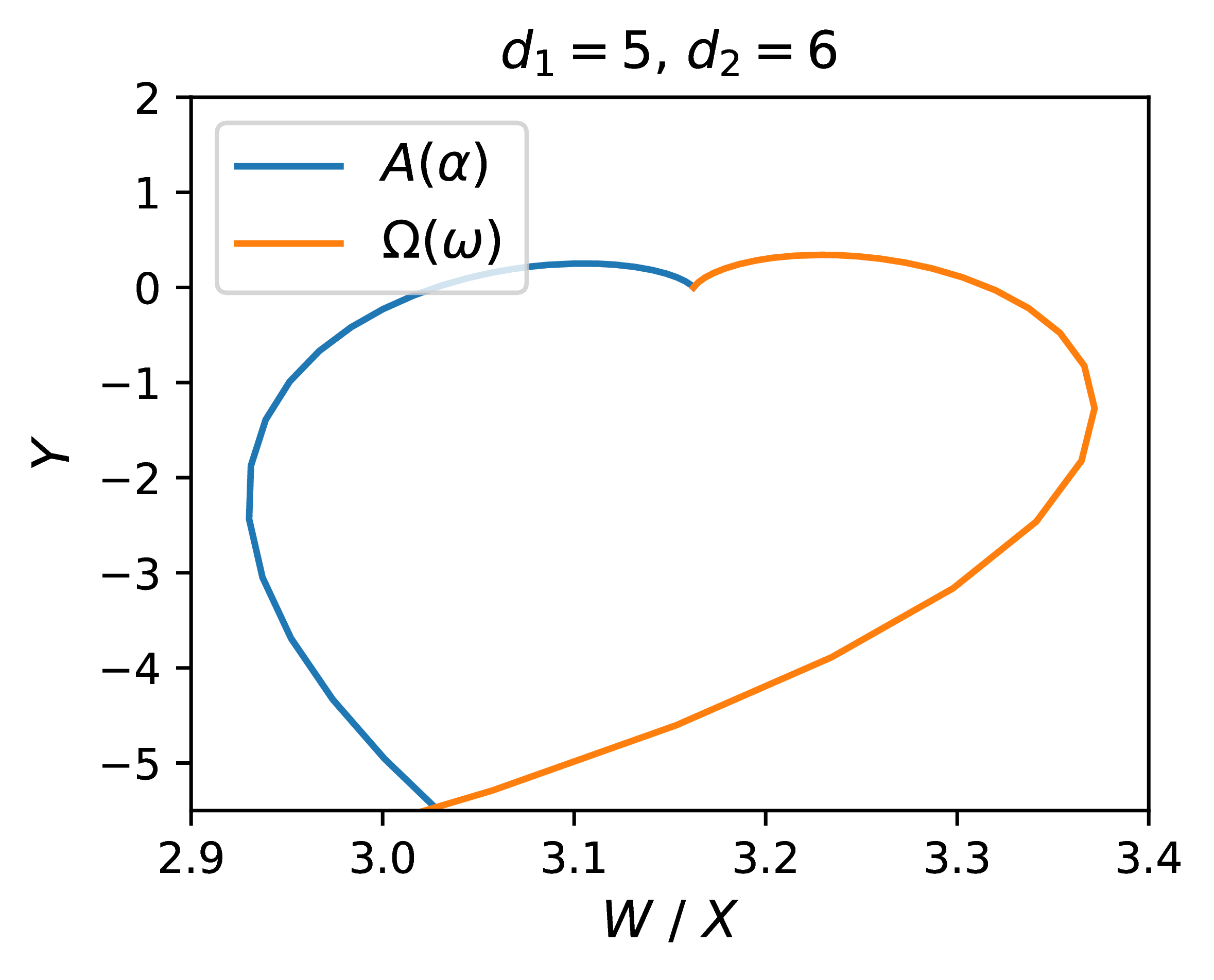} &
\includegraphics[height=5cm]{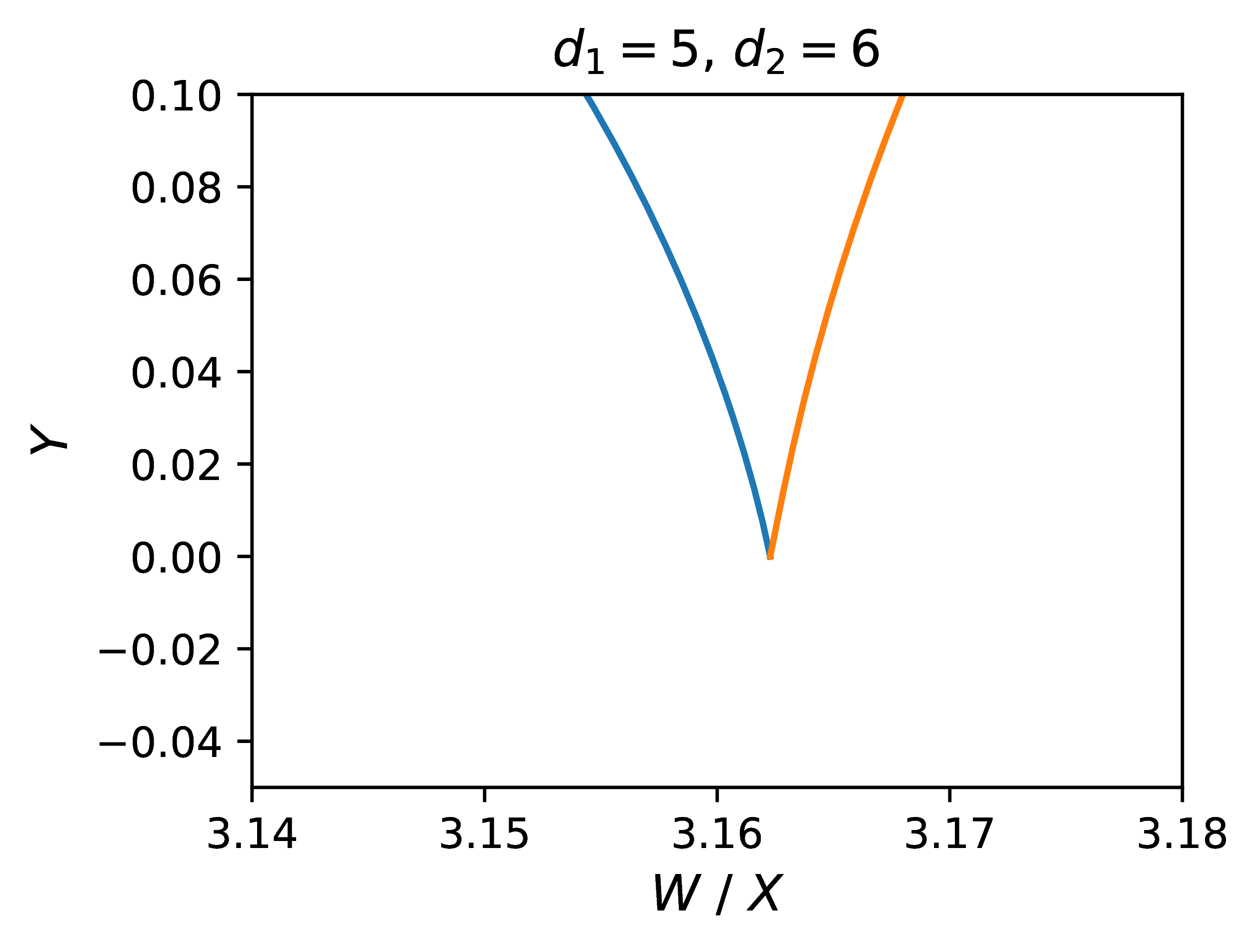} \\
\end{longtable}

\end{document}